\documentclass[makeidx]{amsart}
\usepackage{tikz-cd}
\usepackage[active]{srcltx}
\usepackage{atveryend}
\makeatletter
\let\origcitation\citation
\AtEndDocument{\def\mycites{}
  \def\citation#1{\g@addto@macro\mycites{#1^^J}\origcitation{#1}}}
\AtVeryEndDocument{\newwrite\citeout\immediate\openout\citeout=\jobname.cit
  \immediate\write\citeout{\mycites}\immediate\closeout\citeout}
\makeatother

\usepackage{enumerate}
\usepackage{cancel}
\usepackage{mathdots}
\usepackage{epsfig}
\usepackage{hyperref}
\usepackage{amsmath}
\usepackage{amssymb}
\usepackage{graphics}
\usepackage[mathscr]{euscript}
\newtheorem{Proposition}{Proposition}

  \newtheorem{Corollary}[Proposition]{Corollary}
  \newtheorem{Lemma}[Proposition]{Lemma}
  
  \newtheorem{Theorem}[Proposition]{Theorem}
 
 \newtheorem{Definition}[Proposition]{Definition}
 
  \newtheorem{Observation}[Proposition]{Observation}
 \newtheorem*{Notational Convention 1}{Notational Convention 1}
  \newtheorem*{Notational Convention 2}{Notational Convention 2}
\newtheorem*{Cautionary Note*}{Cautionary Note}
\newtheorem*{Main Theorem*}{Main Theorem}
\newtheorem*{Illustrations of Main Theorem*}{Illustrations of Main Theorem}

 \newtheorem{Note}[Proposition]{Note}
 \newtheorem*{Convention*}{Convention}

\DeclareFontFamily{U}{fsy}{}
\DeclareFontShape{U}{fsy}{m}{n}{<->s*[.9]psyr}{}
\DeclareSymbolFont{der@m}{U}{fsy}{m}{n}
\DeclareMathSymbol{\der}{\mathord}{der@m}{182}

\def\No {\mathbf{No}}

\def\CC{\mathbb{C}}
\def\ZZ{\mathbb{Z}}
\def\NN{\mathbb{N}}

\newcommand{\N}{{\mathbb N}}

\newcommand{\R}{{\mathbb R}}
\def \RR{{\mathbb R}}
\def \CC{{\mathbb C}}
\def\RR{\mathbb{R}}

\title[Integration on the surreals]{Integration on the surreals}
\author {Ovidiu Costin}
\address{Mathematics Department\\
The Ohio State University\\
231 w 18th Ave,
Columbus, OH 43210}
\email{costin@math.ohio-state.edu}

\author{Philip Ehrlich}
\address{Department of Philosophy\\
Ohio University\\
Athens, OH 45701}
\email{ehrlich@ohio.edu}

\begin{document}

\begin{abstract} 
Conway's real closed field $\No$ of surreal numbers is a sweeping generalization of the real numbers and the ordinals to which a number of elementary functions such as log and exponentiation have been shown to extend. The problems of identifying significant classes of functions that can be so extended and of defining integration for them have proven to be formidable. In this paper we address this and
related unresolved issues by showing that extensions to $\No$, and thereby integrals, exist for most functions arising in practical applications. In particular, we show they exist for a large subclass of the \emph{resurgent functions}, a subclass that contains the functions that at $\infty$ are semi-algebraic, semi-analytic, analytic,  meromorphic, and Borel summable as well as solutions to nonresonant linear and nonlinear meromorphic systems of ODEs or of difference equations. By suitable changes of variables we deal with arbitrarily located singular  points.  We further establish a sufficient condition for the theory to carry over to ordered exponential subfields of $\No$ more generally and illustrate the result with structures familiar from the surreal literature. The extensions of functions and integrals that concern us are constructive in nature, which permits us to work in NBG less the Axiom of Choice (for both sets and proper classes). Following the completion of the positive portion of the paper, it is shown that the existence of such constructive extensions and integrals of substantially more general types of functions (e.g. smooth functions) is obstructed by considerations from the foundations of mathematics.

 
 \smallskip
 \noindent \textbf{MSC classification numbers: Primary  03H05,12J15, 34E05, 03E15; Secondary 03E25, 03E35, 03E75}

\end{abstract}

\maketitle

\tableofcontents

\section{Introduction}\label{Intro}

In his seminal work \emph{On Numbers and Games} \cite{CO1, CO2}, J. H. Conway introduced  the system $\No$ of \emph{surreal numbers}, a strikingly inclusive real closed field containing the reals and the ordinals.  In addition to its inclusive structure as an ordered field,  \textbf{No} has a rich 
\emph{simplicity hierarchical} or \emph{s-hierarchical} structure, that depends upon
its structure as a \emph{lexicographically ordered full binary tree} and arises
from the fact that the sums and products of any two members of the tree are the simplest possible
elements of the tree consistent with  \textbf{No}'s structure as an ordered group and an ordered
field, respectively, it being understood that  $x$ is \emph{simpler than}  $y$  (written $x<_s y$) just in case  $x$ is a
predecessor of  $y$ in the tree \cite{EH4, EH5, EK0}. 

An important subsequent advance in the theory of surreal numbers was the extension from the reals to $\No$ of the exponential function by Kruskal and Gonshor \cite{CO2,GO}. The Kruskal-Gonshor exponential function  $\exp$,  like Conway's field operations on $\No$, is inductively defined in terms of $\No$'s simplicity hierarchical structure making use of the fact that for each pair of subsets $L$ and $R$ of $\bf{No}$ for which every member of $L$ precedes every member of $R$, there is a \emph{simplest} member of $\bf{No}$, denoted  $$ \{L \,| \,R\},$$ lying between the members of $L$ and the members of $R$. Conway \cite[pages 27, 225, 227]{CO2} refers to such definitions as \emph{genetic definitions}.\footnote{At present there is no universally accepted formal theory of Conway's loosely defined conception of a genetic definition in the literature, though \cite{FO}, \cite{Fornasiero08}, \cite{RSAS} and most recently \cite{B} have made contributions toward the development of such a theory. Nevertheless, following Conway, in our informal remarks we freely refer to certain inductive definitions as ``genetic''.\label{f1}}

There has been a longstanding
program, initiated by Conway, Kruskal and Norton, to develop analysis on $\mathbf{No}$, starting with a genetic definition of integration. In the case of Kruskal, it was motivated in large part by the broader goal of providing a new foundation for asymptotic analysis which would include new and more general tools for resumming divergent series and for solving complicated differential equations. However, the initial attempts at defining integration, in particular the genetic definition proposed by Norton \cite[page 227]{CO2}, turned out,  as Kruskal discovered, to have fundamental flaws  \cite[page 228]{CO2}. Despite this disappointment, the search for a theory of surreal integration has continued (see \cite{FO} and \cite{RSAS}), but has heretofore remained largely open.\footnote{By contrast, work of Berarducci and Mantova \cite{BM,BM2}, Aschenbrenner, van den Dries and van der Hoeven \cite{ADH2}, Bagayoko \cite{BV}, Bagayoko, van der Hoeven and Kaplan \cite{BHK}, Bagayoko and van der Hoeven \cite{BH1,BH2} and others (e.g. \cite{BHM,vdDE2,vvK,SC}) has made significant progress toward viewing the surreals as an ordered differential field. This work aims to bring a robust theory of asymptotic differential algebra to all of $\No$. Unlike the present work, which is concerned with derivations on surreal functions, the former work is concerned with derivations on surreal numbers. }In this paper, using a new approach, we construct a theory of integration that is of sufficiently wide applicability for most practical cases, pose questions about possible extensions of the theory, and elucidate the nature of the obstructions to a far more general extension.

In real analysis and mathematical physics, the asymptotic expansions at $\infty$ of solutions to nontrivial equations as well as perturbation expansions with respect to small parameters almost invariably have zero radius of convergence.  One of the simplest examples of a divergent series is $$\sum_{k=0}^{\infty}k!x^{-k-1}, \ x\to \infty,$$  a formal solution of the differential equation $y'+y=x^{-1}$, whose general solution is related to the antiderivatives of $e^x/x$ by $y(x)=e^{-x}\int^x e^s s^{-1}ds$. The problem of uniquely assigning functions to divergent expansions in a way that preserves such operations as addition, multiplication, differentiation, integration and composition is a very important and difficult one. A partial solution was provided by Borel summation; however, its domain of applicability is insufficient for many problems of interest in pure and applied analysis. Even for handling relatively common problems in analysis, a satisfactory solution had to wait until the work of \'Ecalle (see \cite{Ecalle1,Ecalle2}) which introduced (among other things) the notions of \emph{resurgent functions, resurgent transseries} and \emph{\'Ecalle-Borel summation} for overcoming the limitations of Borel summation. In $\mathbf{No}$, on the other hand, for all surreal $x>\infty$,  $\sum_{k=0}^{\infty}k!x^{-k-1}$ (and in fact, any formal series in powers of $1/x$ with real coefficients more generally) is \emph{absolutely convergent in the sense of Conway} (see \S\ref{S5})\footnote{In the context of discussions of $\mathbf{No}$, such references to ``$\infty$'' refer to the gap in $\mathbf{No}$ separating the positive finite surreals and positive infinite surreals. Similar references to ``$-\infty$'' are understood analogously. In our discussion, ``$\infty$'' and ``$-\infty$'' refer to both gaps and limits depending on context.\label{f2}} and therewith by comparatively simple means defines a unique function for all infinite surreal $x$. Accordingly, the question naturally arises as to whether building on absolute convergence in the sense of Conway and the ideas of \'Ecalle, we can find a theoretically satisfying way of extending functions and their integrals past $\infty$ or, more generally, past a singularity at which asymptotic expansions do not exist or are divergent? As we alluded to above, in this paper we provide a qualified affirmative answer to this question.

Making real progress towards solving the above-said integration problem, and more generally in interpreting divergent expansions by means of surreal analysis, requires finding a property-preserving operator (see Definition \ref{*E})
that extends the members of a wide body of important classical functions from $\R$ to $\No$. In turn, the existence of such an extension operator provides a theoretically satisfying and widely applicable definition of integration: in particular, the integral of an extension from $\R$ to $\No$ of a function on the reals can be defined to be the extension  of its  integral from $\R$ to $\No$.

Any such theory would have to keep in mind that functions whose
behavior can be described in terms of exponentials and logarithms are
remarkably ubiquitous. Indeed, as G. H. Hardy noted in 1910:

\begin{quote}

No function has presented itself in analysis the laws of whose
increase, in so far as they can be stated at all, cannot be stated,
so to say, in logarithmic-exponential terms. \cite[1st Edition, page 35; 2nd Edition, page 32]{Hardy} \footnote{Note that powers fall in this category since $x^a=e^{a\ln x}$.}$^{,}$\footnote{The work of \'Ecalle on transseries, and resurgent transseries in particular, sheds important  light on Hardy's observation. The system of transseries, which consists of formal series built up from $\RR$ and a variable $x> \RR$ using  powers, exponentiation, logarithms and infinite sums, is the closure of formal power series under a wide range of algebraic and analytical operations \cite{ADH1,CostinT, Edgar}. The subspace of resurgent transseries consists of those transseries which, loosely speaking, have origins in natural problems in analysis (see \S \ref{SecEBRT} as well as \cite{Ecalle1,Ecalle2}). There is compelling mathematical evidence, albeit thus far no rigorous proof, that resurgent transseries are also closed under the known algebraic and analytical operations. Moreover, they are associated with resurgent functions by means of \'Ecalle-Borel summation. These facts provide a theoretical basis for Hardy's observation that, in practice, functions whose asymptotic behavior can be described in logarithmic-exponential terms are the only ones that arise naturally as solutions of problems in analysis. It should be noted, however, that unlike the asymptotic expansions used at the time of Hardy's cited writings, 
the infinite sums of logarithmic-exponential terms occurring in \'Ecalle's theory include  sums of \emph{countable transfinite length $> \omega$.}}
\end{quote}
\noindent
Accordingly, developing a satisfactory theory of integration on the
surreals would require building on the exponential ordered field $(\No, \exp)$ of surreal numbers.

Against this backdrop, in the pages that follow, we show that an extension operator $\mathsf{E}$ as described above, and thereby extensions of integrals from $\RR$ to $\No$, exist for a large subclass $\mathcal{F}_\mathcal{R}$ of resurgent functions, which is related via \'Ecalle-Borel summation to a corresponding subclass $\mathbb{T}_\mathcal{R}$ of resurgent transseries, which contains all real functions that at $\infty$ are semi-algebraic, semi-analytic, analytic, and functions with divergent but Borel summable series (see \S \ref{SecTEB}), as well as solutions of nonresonant linear or nonlinear meromorphic systems of ODEs or of difference equations. As such, most classical special functions, such as Airy, Bessel, Ei, erf, Gamma, and Painlev\'e transcendents, are covered by our analysis.\footnote{Integration for functions with convergent expansions has been studied in the context of the non-Archimedean ordered field of left-finite power series with real coefficients and rational exponents in \cite{SB} and \cite{S}. In addition, for the category of semi-algebraic sets and semi-algebraic functions on arbitrary real-closed fields a full Lebesgue measure and integration theory has been developed in \cite{Kaiser 1} and \cite{Kaiser 2}.  See also \cite{Kaiser 3} for integration and measure theory on certain non-Archimedean ordered fields whose value groups have finite Archimedean rank, as well as \cite{Bo} for various positive and negative results on integration in general non-Archimedean fields.}

The definitions of the extension operator $\mathsf{E}$ and corresponding antidifferentiation and integral operators $\mathsf{A}_{\bf No}$  and $$ \int_{x}^{y} f:=\mathsf{A}_{\bf No}(f)(y)-\mathsf{A}_{\bf No}(f)(x)$$
 
\noindent
given below are \emph{not} genetic in Conway's sense (see Footnote \ref{f1}).  However, unlike Norton's aforementioned definition of integration which was found to be intensional \cite[page 228]{CO2}, ours are shown to depend solely on the values of the functions involved. $\mathsf{A}_{\bf No}$ is defined making use of an antidifferentiation operator $\mathsf{A}$ on $\mathbb{T}_\mathcal{R}$, which in turn is defined using an antidifferentiation operator $\mathsf{A}_{\mathbb{T}}$ on the exponential ordered field ${\mathbb{T}}$ of transseries (Proposition \ref{DMM1}). $\mathsf{A}$ has the property: for real $f$ the restriction to reals of $\mathsf{A}f$ has limit zero at $\infty$ whenever the limit exists. 
This can be viewed as the natural condition for an integral with an endpoint at $\infty$.

All of the members of $\mathcal{F}_\mathcal{R}$ are {\em resurgent at} $\infty$ (see Definition \ref{defEE}). Following our treatment of the just-said extension, antidifferentiation and integral operators based on $\mathcal{F}_\mathcal{R}$ or $\mathsf{E}(\mathcal{F}_\mathcal{R})$ we show by means of a simple change of variables argument that substantial extensions of those operators can be obtained building on a set of functions $\mathcal{F}_\mathcal{R}^*$ extending $\mathcal{F}_\mathcal{R}$ that contains functions that are resurgent at arbitrary points. 

More generally, the original portions of the paper consist of the following.
In \S\ref{MainDef} we introduce the definitions of \emph{extension, antidifferentiation and integral operators} and prove a preliminary result about the existence of integral operators. In \S\ref{Sconstr} we outline the difficulties of defining extensions and integration of functions, and our strategy for overcoming them. Following this, to prepare the way for the proof of the main antidifferentiation theorem, in \S\ref{SecEBRT2} we establish the requisite results concerned with resurgent functions, resurgent transseries and \'Ecalle-Borel summability. The definitions of the extension and antidifferentiation operators $\mathsf{E}$ and $\mathsf{A}_{\bf No}$, together with  proofs of the main antidifferentiation theorem (Theorem \ref{TAntidiff}) are given in \S\ref{CTSN}, along with mention of the uniqueness of $\mathsf{E}$ and $\mathsf{A}_{\bf No}$, the proofs of which are left for a separate paper. This is followed in \S\ref{EA} by the above-mentioned constructions of extensions of $\mathsf{E}$, $\mathsf{A}_{\bf No}$ and the corresponding integral operator, and in \S\ref{ILL} by illustrations of the antidifferentiation and/or extension theorems for exp, the exponential integral Ei, the imaginary error function erfi, the Airy functions Ai and Bi, the log-gamma function, the Gamma function and Jacobi's elliptic function $\theta_3$. In \S\ref{CTSN}(2), a substantially shorter and simplified version of the proof of the main extension theorem is provided for the proper subclass $\mathcal{F}_{conv}$ of $\mathcal{F}_{\mathcal{R}}$ consisting of all functions that, at $\infty$, have convergent series in integer or fractional powers of $1/x$ or more generally have convergent transseries. By a result of van den Dries \cite{VdDries}, these include the semi-analytic functions at $\infty$. In \S\ref{GEN} we generalize our main results by showing that \emph{closure} under absolute convergence in the sense of Conway is a sufficient condition for the theory of extension, antidifferentiation and integral operators outlined above to carry over to ordered exponential subfields of $(\No, \exp)$, and we illustrate the result with substructures of $(\No, \exp)$ that are familiar from the literature. Following this, in \S\ref{OQ} we  raise a problem and state two open questions that naturally arise from material in the preceding sections.

To help keep the paper self-contained we include three preparatory sections: \S\ref{S4} offers an overview of some basic ingredients of surreal theory; and \S\ref{SecTEB} offers an overview of transseries as well as those aspects of Borel summability theory that provide background for the preliminary discussion of \'Ecalle-Borel summability in \S\ref{SecEBRT}, which in turn provides background for \S\ref{SecEBRT2} and \S\ref{CTSN}.

In writings on surreal numbers it is customary to work in NBG (von Neumann-Bernays-G\"odel set theory with the Axiom of Global Choice (see, for example, \cite{ME}). However, in \S\ref{MainDef}-\S\ref{GEN}, which constitutes the positive portion of the paper, we need only work in ${\rm NBG}^-$  (NBG less the Axiom of Choice for both sets and proper classes), since the extensions of functions and integrals that concern us there have an explicitly ``constructive'' nature.

Whereas Kruskal hoped to appeal to Conway's notion of absolute convergence to construct new foundations for asymptotic analysis grounded in a robust theory of surreal integration and function extensions more generally, our theory is more modest in its potential scope, limiting its attention to a broad subclass of resurgent functions that arises in most applied settings. In fact, there are reasons to believe that deep hurdles lay in the way of realizing the lofty analytic goals of Kruskal. Indeed, in Section \S\ref{Neg} we reverse course and show that the existence of extensions and integrals for substantially more general classes of functions (e.g. the class of smooth functions) cannot be proved in ${\rm NBG}^-$, and is in fact obstructed by considerations from the foundations of mathematics.\footnote{Some of the material in \S\ref{MainDef}-\S\ref{ILL} of the present paper is a revised and substantially expanded version of material from the positive portion of the arXiv preprint \cite{CEF}. Further set-theoretic impediments to the realization of Kruskal's program are contained in the negative portion of \cite{CEF} and remain to be revised and expanded by Harvey Friedman and the first author.}

\section{Surreal numbers}\label{S4}

This section provides an overview of the basic concepts of the theory of surreal numbers, including the normal forms of surreal numbers, the aforementioned notion of absolute convergence in the sense of Conway and exponentiation. With the exception of Propositions \ref{sur7} and \ref{Pcop} and Notational Convention 1, which are concerned with absolute convergence in the sense of Conway (see Section \ref{S5}), all of the material in this section is known from the literature. 

To avoid possible confusion, we note that here and henceforth we follow the convention of \emph{excluding} $0$ from the set $\NN$ of natural numbers.

\subsection{The algebraico-tree-theoretic structure of ${\bf No}$}\label{PreS}

\begin{sloppypar}
There are a variety of constructions of the surreal numbers (e.g. \cite[pages 4-5, 15-16, 65]{CO2}, \cite{AE1,AE2,EH7,EH1}, \cite[page 242]{EH4}), each with its own virtues. For the sake of brevity, here we adopt the construction based on Conway's {\it sign-expansions} \cite[page 65]{CO2}, an approach which has been made popular by Gonshor \cite{GO}.
\end{sloppypar}
In accordance with this approach, a {\it surreal number} is 
a function $x : \lambda \to \{-,+\}$ where $\lambda$ is an ordinal called the \emph{length} of $x$.
The class $\No$ of surreal numbers so defined carries a canonical linear ordering $<$ as well as a  canonical partial
ordering $<_s$ defined by the conditions: $x< y$ if and only if $x$ is
(\emph{lexicographically}) less than $y$ with respect to the linear ordering on
 $\{-,+\}$, it being understood that $- < \text{\emph{undefined}}\ < +$; $x<_s y$ (read ``$x$ is \emph{simpler than} $y$'') if and only if $x$ is a proper initial segment of $y$. 

By a \emph{tree} $({A,{ < _A}})$ we mean a partially ordered class such that for each $x \in A$, the class 
${\rm pr}_A=\left\{ {y \in A: y { < _A} x} \right\}$ of \emph{predecessors} of $x$ is a \emph{set} well ordered by ${ < _A}$. The \emph{tree-rank} of $x \in A$, written `$\rho _A (x)$', is the ordinal corresponding to the well-ordered set $({\rm pr}_A (x), <_s)$. If $x, y \in A$, then $y$ is said to be an \emph{immediate successor} of $x$ if $x < _s y$ and $\rho _A (y) = \rho _A (x) + 1$; and if $(x_\alpha)_{\alpha < \beta }$ is a chain in $A$ (i.e., a subclass of $A$ totally ordered by $ < _s $), then $y$ is said to be an \emph{immediate successor of the chain} if $x_\alpha < _s y$ for all $\alpha < \beta$ and $\rho _A (y)$ is the least ordinal greater than the tree-ranks of the members of the chain. The \emph{length} of a chain $(x_\alpha)_{\alpha < \beta }$ in $A$ is the ordinal $\beta $. If each member of  $A$ has two immediate successors and every chain in $A$ of limit length (including the empty chain) has 
one immediate successor, the tree is said to be a \emph{full binary tree}.

\begin{Proposition}\label{sur1}{\rm

 $({{\bf No} \, <, <_s})$ is a lexicographically ordered full binary tree (\cite{EH5}, \cite[Theorem 11]{EH7}).}
\end{Proposition}

Central to the algebraico-tree-theoretic development of the theory of surreal numbers is the following consequence of Proposition \ref{sur1}, where a subclass $B$ of an ordered class $( A,< )$  is said to be \emph{convex}, if $z\in B$ whenever $x,y\in B$ and $x < z < y$.

\begin{Proposition}\label{sur2}{\rm
Every nonempty convex subclass of $\No$ has a simplest member. In particular, if $L$ and $R$ are (possibly empty) sub\emph{sets} of ${\bf No}$ for which every member of $L$ precedes every member of $R$ (written $L<R$), there is a \emph{simplest} member of ${\bf No}$ lying between the members of $L$  and the members of  $R$ \cite[Theorem 1 and Theorem 4 (i) and (ii)]{EH5}. 
 }
\end{Proposition}
 
 Co-opting notation introduced by Conway, the simplest member of  ${\bf No}$ lying between the members of  $L$ and the members of $R$  is denoted by $$\{L|R\}.$$ 

Following Conway \cite[page 4]{CO2}, if $x=\left\{ L|R\right\}$, we write $x^{L}$ for a
typical member of $L$ and $x^{R}$ for a typical member of $R$; $x=\left\{
a,b,c,...|d,e,f,...\right\} $ means that $x=\left\{ L|R\right\} $ where $%
a,b,c,...$ are typical members of $L$ and $d,e,f,...$ are typical 
members of $R$.      

Each  $x\in\bf No$ has a \emph{canonical representation} as the simplest member of ${\bf No}$ lying between its predecessors on the left and its predecessors on the right, i.e.$$x=\{L_{s\left ( x \right )}| R_{s\left ( x \right )}\},$$ where  $$L_{s\left( x \right)}=\left\{ {a \in {\bf No}:a <_s x\;{\rm{and}}\;a < x} \right\}\; {\rm and} \; R_{s\left( x \right)}=\left\{ {a \in {\bf No}:a <_s x\;{\rm{and}}\;x < a} \right\}.$$
 By now letting $x=\{L_{s\left ( x \right )}| R_{s\left ( x \right )}\}$ and $y=\{L_{s\left ( y \right )}| R_{s\left ( y \right )}\}$, $+ , - $ and $\cdot$ are defined by recursion as follows, where $x^L $,  $x^R$,  $y^L $ and  $y^R $ are understood to range over the members of  $L_{s\left( x \right)},R_{s\left( x \right)} ,L_{s\left( y \right)} $ and  $R_{s\left( y \right)} $, respectively.

\bigskip 
\emph{Definition of}  $x + y.$
 $$x + y = \left\{ {x^L  + y,x + y^L |x^R  + y,x + y^R } \right\}.$$
 
 \emph{Definition of}  $- x.$
 $$- x = \left\{ {- x^R | - x^L } \right\}.$$
 
 \emph{Definition of}  $xy.$
 
\begin{eqnarray*}
\!\!\!xy &=&\{x^{L}y+xy^{L}-x^{L}y^{L},x^{R}y+xy^{R}-x^{R}y^{R}| \\
&&\quad \quad \qquad \qquad \qquad \text{\quad }
x^{L}y+xy^{R}-x^{L}y^{R},x^{R}y+xy^{L}-x^{R}y^{L}\}\text{.}
\end{eqnarray*}

\medskip

Despite their cryptic appearance, the definitions of sums and products in $\No$ 
have natural interpretations that essentially assert that the sums and products of elements of $\No$ are the simplest elements of $\No$ consistent with $\No$'s structure as an ordered group and an ordered field respectively (see, for example, \cite[page 1236]{EH4}, \cite[pages 252-253]{EH5}). The constraint on additive inverses, which is a consequence of the definition
of addition \cite[page 1237]{EH5}, ensures that the portion of the surreal
number tree less than $0=\left \{ \varnothing |\varnothing \right \}$ is (in absolute value) a mirror image of the portion of the surreal number tree greater than $0$, $0$ being the simplest element of the surreal number tree (see Figure 1).

A subclass $A$ of $\No$ is said to be \emph{initial} if $x\in A$ whenever $y\in A$  and $x<_s y$. Although there are many isomorphic copies of the order field of reals in $\No$, only one is initial \cite[page 1243]{EH5}. This ordered field, which we denote $\R$, plays the role of the reals in $\No$. Similarly, while there are many subclasses $A$ of $\No$ that are well-ordered proper classes in which for all $x,y \in A$, $x<y$ if and only if $x<_s y$, only one is initial. The latter, which consists of the members of the rightmost branch of $({{\bf No} \, <, <_s})$ (see Figure 1), is identified as $\No$'s class $\mathbf{On}$ of ordinals.

The nonzero elements of $\No$ can be partitioned into equivalence classes, called \emph{Archimedean classes}, each consisting of all nonzero members $x,y$ of $\No$ that satisfy the condition: $m|x|>|y|$ and  $n|y|>|x|$ for some positive integers $m,n$. If $a$ and $b$ are members of distinct Archimedean classes and $\left |a \right |< \left |b \right |$, then we write $a \ll b$ and $a$ is said to be \emph{infinitesimal (in absolute value) relative to} $b$.

An element of  ${\bf No}$ is said to be a \emph{leader} if it is the simplest member of the positive elements of an Archimedean class of  ${\bf No}$. Since the class of positive elements of an Archimedean class of  ${\bf No}$ is convex, by the first part of Proposition \ref{sur2} the concept of a leader is well defined. There is a unique mapping--\emph{the} $\omega$\emph{-map}--from ${\bf No}$ onto the ordered class of leaders that preserves both $<$ and $<_s$. The image of $y$ under the $\omega$-map is denoted $\omega^y$, and in virtue of its order preserving nature, we have: for all $x,y\in \bf {No}$, $$ \omega^x \ll \omega^y \;\text{if and only if} \; x<y.$$

Using the $\omega$-map along with other aspects of ${\bf No}$'s $s$-hierarchical structure and its structure as a vector space over $\R$, every surreal number can be assigned a canonical ``proper name'' or \emph{normal form} that is a reflection of its characteristic
$s$--hierarchical properties. These normal forms are expressed as
sums of the form  $$\sum\limits_{\alpha  < \beta } {\omega ^{y_\alpha  } .r_\alpha} $$
where  $\beta $ is an ordinal,  $\left( {y_\alpha} \right)_{\alpha  < \beta } $ is a strictly decreasing
sequence of surreals, and  $\left( {r_\alpha  } \right)_{\alpha  < \beta } $ is a sequence of nonzero
real numbers. Every such expression is in fact the normal
form of some surreal number, the normal form of an ordinal being just its \emph{Cantor normal form} (\cite[pages 31-33]{CO2}, \cite[\S3.1 and \S5]{EH5}, \cite{EH6}).

Making use of these normal forms, Figure ~\ref{fig:first} offers a glimpse of the some of the early stages of the recursive unfolding of  ${\bf No}$.

\begin{figure}[ht!]

\centering\includegraphics[width=0.95\textwidth]{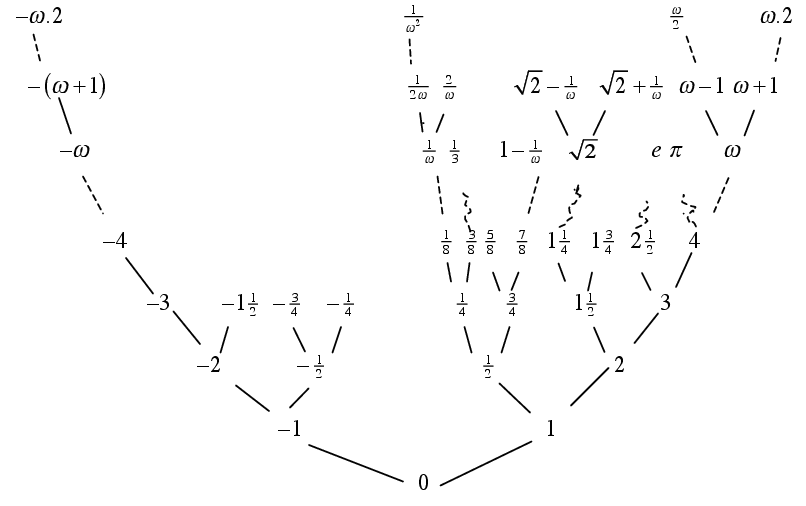}

\caption{Early stages of the recursive unfolding of ${\bf No}$}
\label{fig:first} 
\end{figure}

When surreal numbers are represented by their normal forms, order, addition and multiplication in $\No$ assume more tractable forms with the order defined lexicographically and addition and multiplication defined  as for polynomials with $\omega^{x} \omega^{y} = \omega^{x + y}$ for all $x, y\in \No$.

\begin{Definition}\label{part}
{\rm An element $x$ of an ordered field is said to be \emph{infinitesimal} if $\left| x \right| < 1/n$ for every positive integer $n$ and it is said to be \emph{infinite} if $\left| x \right| > n\cdot1$ for every positive integer $n$. Thus, in virtue of the lexicographical ordering on normal forms, a surreal number is infinite (infinitesimal) just in case the greatest exponent in its normal form is greater than (less than) $0$. As such, each surreal number $x$ has a canonical decomposition into its \emph{purely infinite part}, its \emph{real part}, and its \emph{infinitesimal part},  consisting of the portions of its normal form all of whose exponents are $>0$, $=0$, and $<0$, respectively. A surreal number, and a member of an ordered field more generally, will be said to be \emph{finite} if it is not infinite.}
\end{Definition}

\subsection{Absolute convergence in the sense of Conway}\label{S5}

There is a notion of convergence in ${\bf No}$ for sequences and series of surreals that can be conveniently expressed using normal forms supplemented with dummy terms whose coefficients are zero.  Let $x\in\No$ and for each $y\in\No$, let $r_{y}(x)$ be the coefficient of $\omega^y$ in the normal form of $x$, it being understood that $r_{y}(x)=0$, if $\omega^y$ does not occur.  Also let $\{x_{n}:n\in \N\cup\{0\}\}$ be a sequence of surreals so written. Following Siegel \cite[page 432]{SI}, we write \[x=\lim _{n\rightarrow \infty }x_{n}\]

\noindent to mean \[r_y\left ( x \right )=\lim _{n\rightarrow \infty }r_y \left ( x_{n} \right ),  \:   \text{for all}\ y\in\No,\]

\noindent and say that $\{x_{n}:n\in \N\cup\{0\}\}$ \emph{converges} to $x$. We also write \[x=\sum_{n=0}^{\infty }x_n\]

\noindent to mean the partial sums of the series converge to $x$.

Among the convergent sequences and series of surreals are those whose mode of convergence is quite distinctive. In particular, for each $y\in\bf {No}$, there is a nonnegative integer $m$ such that $r_{y}(x_n)=r_{y}(x_m)$ for all $n\geq m$. Thus, for each $y\in\bf {No}$,
 \[r_y\left ( x \right )=\lim _{n\rightarrow \infty }r_y \left ( x_{n} \right )=r_{y}(x_m),\]

\noindent where $m$ depends on $y$.
Following Conway, we call this mode of convergence \emph{absolute convergence}. 

{\begin{Notational Convention 1}
{\rm We will call the normal form to which an absolutely convergent series $\{x_{n}:n\in \N\cup\{0\}\}$ of normal forms converges the \emph{Limit} of the series and denote it using 
\begin{equation}
\mathop{\mathrm{Lim}}_{n\to \infty}x_n. \label{eq:NC1}
\end{equation}
\noindent
We use ``Limit'' as opposed to ``limit'' and ``$\mathop{\mathrm{Lim}}$'' as opposed to ``$\lim$'' to distinguish the surreal notion from its classical counterpart.}
\end{Notational Convention 1}}

Relying on the above and classical combinatorial results of Neumann (\cite[pages 206-209]{N},\cite[Lemma 3.2]{SI}, \cite[pages 260-266]{AL}), one may prove \cite[pages 432-434]{SI} the following theorem of Conway \cite[page 40]{CO2}, which is a straightforward application to $\No$ of a classical result of Neumann (\cite[page 210]{N}, \cite[page 267]{AL}).

\begin{Proposition}\label{sur5}{\rm Let $f$ be a formal power series with real coefficients, i.e. let

 \begin{equation}
    \label{eq:csoc}
 f\left ( x \right )=\sum_{n=0}^{\infty }r_{n}x^{n}
  \end{equation}
\noindent where the $r_n$'s are reals. Then $f\left ( \zeta \right )$ is absolutely convergent for all infinitesimals $\zeta$ in $\No$, i.e.,
 \begin{equation}
    \label{eq:LimS}
 f( \zeta )=\mathop{\mathrm{Lim}}_{n\to\infty}\sum_{m=0}^{n}r_{m}x^{m}.
  \end{equation}

}
 \end{Proposition}

Conway's theorem also has the following multivariate formulation \cite[page 435]{SI}.

\begin{Proposition}\label{sur6}{\rm  Let $f$ be a formal power series in $k$ variables with real coefficients, i.e. let \[f\left (x_1,...,x_k \right )\in\R[[x_1,...,x_k]].\]
\noindent Then $f\left ( \epsilon_1,..., \epsilon_k \right )$ is absolutely convergent for every choice of infinitesimals $\epsilon_1,..., \epsilon_k$ in $\No$.} 
 \end{Proposition}
 
 This can also be written in the following useful form.
  
\begin{Proposition}\label{sur7}{\rm Let $\{c_{\mathbf{k}}: \mathbf{k}\in(\N\cup\{0\})^m\}$ be any multisequence  of real numbers and $h_1,...,h_m$ be  infinitesimals. Also let $\mathbf{h}^{\mathbf{k}}=h_1^{k_1}\cdots h_m^{k_m}.$ Then
  \begin{equation}
    \label{eq:sumc}
    \sum_{|\mathbf{k}|\ge 0} c_{\mathbf{k}} \mathbf{h}^{\mathbf{k}}
  \end{equation}
is a well-defined element of $\No$. 
}\end{Proposition}

 The following result, in which  $\{x_{n}:n\in \N\cup\{0\}\}$ and  $\{y_{n}:n\in \N\cup\{0\}\}$ are absolutely convergent series of normal forms, collects together some elementary properties of absolute convergence in $\No$. Several are very similar to the properties of the usual limits. 

\begin{Proposition}\label{Pcop}{\rm 
Let ${\rm Lim}_{n\to \infty}x_n=x$ and ${\rm Lim}_{n\to \infty}y_n=y$, and further let $h\ll 1$, $\tau > 0$ and $a,b\in\No$. Then 
\begin{align}
  \label{eq:convprop}
(a)\ &   {\rm Lim}_{n\to \infty}(a x_n+b y_n)=a x+b y       ;\nonumber\\(b)\ & {\rm Lim}_{n\to \infty}x_ny_n=xy ;\nonumber\\ (c)\ & x\ne 0\Rightarrow {\rm Lim}_{n\to \infty} \frac{1}{x_n}=\frac{1}{x};\\ (d)\   & (\exists k)(\forall n)(|x_n|<k);\nonumber\\ (e)\ & {\rm Lim}_{n\to \infty}h^n=0 ;\nonumber\\ (f) \ & (\forall n)(|x_n|\le \tau)\Rightarrow |x|\le \tau.\nonumber
\end{align}
   } \end{Proposition}

 \begin{proof}[Proof of Proposition \ref{Pcop}]
  
  (a) and (b) are proved in \cite[page 271]{AL}, (d) is evident since no set is cofinal with $\bf No$, (e) follows from Proposition \ref{sur2} and (f) follows from (e).
  For (c), since ${\rm Lim}_{n\to \infty}x_n\neq0$, there is a greatest $y\in\bf{No}$ such that $r_y(x_n)$ is not eventually zero. Thus, for sufficiently large $n$, $x_n=r_y\omega^{y}(1+h_n)$, where $h_n$ is infinitesimal, and, so, it suffices to establish the result for $x_n$ of the form $1/(1+h_n)$. Since $1/(1+h_n)-1=-h_n(1+h_n)^{-1}$ and ${\rm Lim}_{n\to \infty}h^n=0$ the coefficients of leaders in $h_n$ eventually vanish, and, as such, eventually vanish for $-h_n(1+h_n)^{-1}$.
  \end{proof}

\subsection{Surreal exponentiation}\label{Sexp}
As was mentioned above, $\mathbf{No}$ admits an inductively defined exponential function $\exp$. $(\No,\exp)$ is in fact an elementary extension of the exponential ordered field $(\R , e^x)$ of real numbers \cite{vdDE}. The exponential function on $\No$ was introduced by Kruskal, and reconstructed and substantially developed by Gonshor \cite[Chapter 10]{GO}. While the definition of $\exp$ is quite complicated for the general case, it reduces to the following simpler and more revealing forms for the three theoretically significant cases.

\begin{Proposition}[Gonshor \cite{GO}]
\begin{enumerate}[(i)]{\rm \item $\exp (x)=e^x$ for all $x \in \mathbb{R}$; 
 \item $\exp (x)=\sum_{\ n =0 }^{\infty }x^{n}/{n!}$ for all infinitesimal $x$;
 \item if $x$ is purely infinite, then
\[
\exp (x)\ =\ \Big \{ 0,\, (\exp x^L)(x-x^L)^{n}\, \Big|\, \frac{\exp x^R}{(x^R-x)^n}\Big \},
\]
where $x^L$ and $x^R$ range over all the purely infinite predecessors of $x$ with $x^L<x<x^R$.}

\end{enumerate}
\end{Proposition}

The significance of cases (i)--(iii) emerges from the fact that for an arbitrary surreal $x$, $\exp (x)\ =\ \exp(x_P)\cdot \exp(x_R)\cdot \exp(x_I)$, where $x_P$, $x_R$ and $x_I$ are the purely infinite, real and infinitesimal components of $x$, respectively.

Shedding further light on $\exp(x)$ when $x$ is purely infinite is:

\begin{Proposition}[Gonshor \cite{GO}]

{\rm The restriction of $\exp$ to the class of purely infinite surreal numbers is an isomorphism of ordered groups onto {\bf No}'s class $\{\omega^x:x \in {\bf No} \}$ of leaders.}
\end{Proposition}

 In subsequent sections of the paper, for the sake of simplicity, we will occasionally write $e^x$ in place of $\exp x$ for the surreal extension of the real function $e^x$. Readers seeking additional background in the theory of surreal exponentiation may consult \cite{GO, vdDE, EK, BM, BKM}.

\section{Extension, antidifferentiation and integral operators}\label{MainDef}

To introduce the requisite conceptions of extension, antidifferentiation and integral operators, we require some preliminary notions concerning intervals, extensions of functions from the reals to {\bf No} and restrictions of surreal functions to $\RR$, where $\RR$ is understood to be the canonical copy of the reals in {\bf No} (see \S\ref{PreS}).

By an \emph{interval} $I$ of an ordered class $A$ we mean a convex subclass of $A$. In addition to the more familiar types of intervals of $\RR$ and $\No$ we will consider are $(a,\infty):=\{x \in \R: x > a \}$  and $(a,{\bf On}):=\{x \in \No: x > a \}$, where $a \in \RR$. In \S\ref{EOP} a simple condition is specified under which the forthcoming developments of our theory also apply to the intervals $(-\infty, a):=\{x \in \R: x < a \}$ and $({-\bf On}, a):=\{x \in \No: x < a \}$, for  $a \in \RR$.

\subsection{Derivatives}\label{D}
To formulate the appropriate notions of extension and antidifferentiation operators, we require a generalization of the idea of a \emph{derivative of a function at a point}.
\begin{Definition}[Derivative]\label{DefDeriv}{\rm 
Let $K$ be an ordered field. A function $f$ defined on an interval around $a$ is differentiable at $a$ if there is an $f'(a)\in K$ such that $(\forall \epsilon >0 \in K)(\exists \delta >0 \in K)$ such that $$(\forall x\in K)(|x-a|<\delta\Rightarrow \displaystyle\left|\frac{f(x)-f(a)}{x-a}-f'(a)\right|<\epsilon).$$ As usual, $f'(a)$ is said to be \emph{the derivative of $f$ at $a$} and $f$ is said to be \emph{differentiable} if the derivative of $f$ exists at each point of its domain. The definition generalizes to higher order derivatives in the usual way. 
}\end{Definition}
{\rm It is straightforward to check that the derivative so defined on {\bf No} has the same {\bf local} properties (linearity, chain rule, etc.) as its real counterpart. However, because {\bf No} is disconnected, global properties such as Rolle's theorem and its consequences may fail.}

\subsection{Extension operators}
If $f$ is a function, then by  dom$(f)$ and ran$(f)$ we mean the domain and range of $f$ respectively. We define
$\lambda f$ and $f+g$ for functions $f,g$ as usual, where dom$(f+g)$ = dom$(f)\cap$dom$(g)$. 

\begin{Definition}\label{*E}
  {\rm  \ Let $I$ be an interval of $\RR$ and $J$ be an interval of {\bf No} that contains $I$. 
    
 (1) As usual, we say that $g:J\to  {\bf No}$ extends $f:I\to \RR$ if for every $x\in I$ we have $g(x)=f(x)$, and  we denote by $g|I$ the restriction of $g$ to $I$. 
  
(2) Let $\mathcal{F}$ be a set of real-valued functions defined on intervals  of $\RR$. \rm{By an \emph{extension operator} $E$ on $\mathcal{F}$ we mean a map that associates to each function $f:I\to\RR$ in $\mathcal{F}$ a function $E\, f:J\to{\bf No}$ in such a manner that
}
\begin{enumerate}[i.]
  \item for all $f\in \mathcal{F}$,  $E\,f $ is an extension of $f$;
  \item (Linearity)  for all $g,h \in \mathcal{F}$ and $C\in\RR$, $E(C g)$ = $C E\,g$ and $E(g+h) = E\,g+E\,h$;
  \item if $\beta,\lambda\in\RR$, $n\in\NN\cup \{0\}$, $g(x)=x^\beta e^{\lambda x}$ and $h(x)=x^n\log(x)$ for all $x\in I$, then $(E\,g)(x)=x^\beta e^{\lambda x}$ and $(E\,h)(x)=x^n\log(x)$ for all $x\in J$.
  \item $E\,f' =(E\,f)'$.
   \end{enumerate}
}\end{Definition}
For some important classes of problems we construct extensions that are {\em multiplicative} or, in other words, that  preserve multiplication in the following sense. 
\begin{Definition}\label{*E10}
  {\rm  An extension operator $E$ is multiplicative on an algebra of functions if for all $f$ and $g$ in the algebra we have $E(fg)=(Ef)(Eg)$}.
\end{Definition}

\subsection{Antidifferentiation and integral operators}

The following definition provides definitions of both real and surreal antidifferentiation operators.

\begin{Definition}\label{Dd2}
\rm Let $\mathcal{F}$ be a set of real-valued (surreal-valued) functions whose domains are intervals of $\RR$ ({\bf No}).  \rm{An \emph{antidifferentiation operator} on $\mathcal{F}_1\subseteq\mathcal{F}$ is a function $A:\mathcal{F}_1\to \mathcal{F}$ such that for all $f,g \in  \mathcal{F}_1$:
\begin{enumerate}[i.]
  \item \label{one} $A\,f$ is differentiable and $(A\,f)'=f$;
  \item \label{two} For any $\lambda\in\RR$ ($\lambda \in {\bf No}$), $A(\lambda f)$ = $\lambda A\,f$, $A(f+g) = A\,f+A\,g$;
  \item \label{iii.} If $y\ge x$ and $f\ge0$, then $(A\,f)(y)-(A\,f)(x)\ge0$.
  \item \label{four} $\forall n\in\N, A\,(x^n)=\frac{1}{n+1}x^{n+1}$ (the right side being the monomial in $\mathcal{F}$).
  \item \label{five} $A\,(\exp)$ equals the real
    (surreal) exponential.
    \item \label{six} If $F\in \mathcal{F}_1$ and $F'=f\in \mathcal{F}_1$, then there is a $C\in\RR$ ($C \in {\bf No}$) such that $A\,f$ exists and equals $F+C$.
\end{enumerate}
}
\end{Definition}

For suitable integrals to exist, we need the ``second half'' of the fundamental theorem of  calculus to hold. This is the motivation for the following convention.

\begin{Notational Convention 2}  

  \rm Let $A$ be an antidifferentiation operator on $\mathcal{F}_1 \subseteq \mathcal{F}$, and let $f\in \mathcal{F}_1$ and $x,y\in\No$. Define
  \begin{equation}
    \label{eq:defint1}
    \int_{x}^{y} f:=A(f)(y)-A(f)(x).
  \end{equation}
\end{Notational Convention 2}  

The following result demonstrates that the existence of an antidifferentiation operator on $ \mathcal{F}_1 \subseteq \mathcal{F}$ implies that $\int_x^y f$ is an operator on $\mathcal{F}_1$ whose properties make it worthy of the appellation ``integral operator''.

In the following proposition, $\alpha, \beta, a,b, a_{1}, a_{2}, a_{3}\in$ {\bf No}, and $f,g,fg,f\circ g,f',g'$ are understood to be elements of $\mathcal{F}_1$ on $[a,b]$, $[a_{1},a_{2}]$, $[a_{2},a_{3}]$ or $[a_{1},a_{3}]$ where applicable. In our constructions we will specify which spaces are closed under the above-said operations.
\begin{Proposition}[\emph{\rm Integral operators}]\label{existint}

{\rm Let $A$ be an antidifferentiation operator on $\mathcal{F}_1 \subseteq \mathcal{F}$. Then $\int_x^y f$ is an \emph{integral operator} on $\mathcal{F}_1$, meaning  a function of three variables, $x,y\in\No$ and $f\in \mathcal{F}_1$, with the properties: 
\begin{enumerate}[(a)]
 \item $\displaystyle \left(\int_a^x f\right)'=f$;
 \item $\displaystyle \int_a^b(\alpha f+\beta g)=\alpha \int_a^b f+\beta \int_a^b g$;
\item $\displaystyle \int_a^b f' =f(b)-f(a)$;
 
  \item  $\displaystyle \int_{a_1}^{a_2} f+\int_{a_2}^{a_3}f=\int_{a_1}^{a_3} f$;
 
  \item  $\displaystyle\int_a^b f'g=fg|_a^b-\int_a^b fg'$ if $f$ and $g$ are differentiable on $(a,b)$;
  \item \label{itemT5}   $\displaystyle\int_a^x (f\circ g)g'=\int_{g(a)}^{g(x)}f$ whenever $g\in \mathcal{F}_1$ is differentiable on $(a,x)$.
   
  \item If $f$ is a positive function and $b>a$, then $\displaystyle \int_a^b f>0$.

  \end{enumerate}
  
}\end{Proposition}

\begin{proof}
All these are straightforward. (a) follows from Definition \ref{Dd2} i, and differentiating \eqref{eq:defint1}. (b) follows from Definition \ref{Dd2}, ii. (e) follows similarly using the chain rule and, taking $g=1$,  it implies (c). (d) follows from Equation \eqref{eq:defint1}. For (f), note that since $f\in \mathcal{F}_1$ we have $f=F'$ for some $F$, and hence $(F\circ g)'=(f\circ g)g'$; the rest is a consequence of (c). And, finally, (g) follows from Definition \ref{Dd2}, \ref{iii.}.

\end{proof}

As we alluded to in the introduction, in \S\ref{CTSN} we construct a wide class of functions defined on intervals of $\RR$ of the form $(a,\infty)$, where $a$ may depend on the function, that is closed under antidifferentiation in the sense of Definition \ref{Dd2}, and which we extend in the sense of Definition \ref{*E} to surreal functions defined on $(a,{\bf On})$. By contrast, for our negative result we only retain some very basic properties of antidifferentiation and work on a space of functions with ``very good properties''. This is spelled out below in \S\ref{Neg}.

\section{Difficulties of defining extensions and integration of functions, and our strategy for overcoming them}\label{Sconstr} 

One of the sources of difficulty in extending more general classes of classical functions to {\bf No} and in defining integration for them is the fact that the topology of surreal numbers is totally disconnected, and as such processes {\em other} than the usual ``extensions by continuity'' must be employed.  A natural class of functions on which extensions and integration can be naturally defined in a way that preserves the expected properties are the analytic functions. This is due to their {\em unique representations as power series}, which at $\infty$ take the form
\begin{equation}
  \label{eq:powerseries}
 f(x)= \sum_{k=0}^\infty \frac{c_k}{x^k}
\end{equation}
where for some positive real $R$ and all $k \in \mathbb{N}\cup \{0\}$ we have $|c_k|\le R^k$; of course the series in Equation \eqref{eq:powerseries} converges for all $x\in \RR^+$ such that $x>R$. We can make use of normal forms to define  $\mathsf{E}f(x)$ for all surreal numbers greater than $R$ in a way that ensures that $\mathsf E$ preserves all operations  that are preserved by Limits (see \S\ref{S5}). 
For this, relying on Proposition \ref{sur5} and the definition of ``$\mathrm{Lim}$'' (see \S\ref{S5}), we simply write
\begin{equation}
  \label{eq:powerseries1}
{\mathsf E} f(x)= \sum_{k=0}^\infty \frac{c_k}{x^k}=\mathop{\mathrm{Lim}}_{N\to \infty}\sum_{k=0}^N \frac{c_k}{x^k}.
\end{equation}
Similarly, for all ${x \in \bf No}$ such that $x>R$ and $f$ as in Equation \eqref{eq:powerseries}, we let 
\begin{equation}
  \label{eq:powerseriesInt1}
\mathsf{A}_{\bf No}\,f (x)=c_0 x +c_1\log x -\sum_{k\ge 2}^\infty \frac{c_k}{(k-1)x^{k-1}}.
\end{equation}
Based on Proposition \ref{Pcop}, it is an easy exercise to check that $\mathsf{A}_{\bf No}$ so defined is an antidifferentiation operator on the class of functions analytic at $\infty$. In fact, for the class of functions analytic at $\infty$ and $O(x^{-2})$ for large $x$ ($c_0=c_1=0$ in \eqref{eq:powerseries}) this is an antiderivative with ``zero constant at $\infty$'' or from $\infty$. Integration is defined for $R<a\le x$ by Equation \eqref{eq:defint1}.

With obvious adaptations, these definitions, constructions and results extend to functions that are given at $\infty$ by convergent Puiseux fractional power series or, far more generally, by \emph{convergent transseries} (see \S\ref{STrans} and e.g. \cite[page 143]{Book}). 

While divergent series and transseries {\em as formal objects} can be associated  in much the same way with actual surreal functions defined on the positive infinite elements of {\bf No}, the difficulty in these cases is to pair them with functions on the \emph{finite} surreals in a unique way that, additionally, is compatible with common operations in analysis. Indeed, while in classical analysis convergent expansions correspond to a unique function, this is {\em not} the case for divergent expansions. We overcome this difficulty by using techniques of resurgent functions and \'Ecalle-Borel summation (\S\ref{CTSN}).

The following simple example based on the exponential integral Ei illustrates the non-uniqueness problem in the divergent case. The function $y(x)=e^{-x}\mathrm{Ei}(x)$ is given by
\begin{equation}
  \label{eq:defei}
  y(x)=e^{-x} \mathop{\rm PV}\int_{-\infty}^x\frac{e^s}{s}ds
\end{equation}
where PV stands for the \emph{Cauchy principal value}: for $x>0$ this is defined as the symmetric limit $\lim_{\epsilon\to 0^+}\left(\int_{-\infty}^{-\epsilon}+\int_\epsilon^x\right)$.

This $y(x)$ has the asymptotic series
\begin{equation}
  \label{eq:divser}
  y(x)\sim \sum_{k=0}^\infty \frac{k!}{x^{k+1}},\ \ x\to \infty.
\end{equation}
Since $y(x)=\sum_{k=0}^\infty  \frac{k!}{x^{k+1}}$ is well defined for all $x\in {\bf No}$ via Limits, as in  \eqref{eq:powerseries1}, it would be tempting to define the integral PV$\int_{-\infty}^x\frac{e^s}{s}ds$ for all $x>\infty$ as $e^x y(x)$. But here we face a non-uniqueness problem: for any $a\in \RR$, the function $y_a(x)=e^{-x} \mathop{PV}\int_{a}^x\frac{e^s}{s}ds$ has the {\bf same} asymptotic series as $y(x)$ given in Equation  \eqref{eq:divser}. This is because $y(x)-y_a(x)=C e^{-x}$ (where the constant $C$ is  $\mathop{PV}\int_{-\infty}^a\frac{e^s}{s}ds$) and the {\em power series} asymptotics of $e^{-x}$ for large $x$ is zero. In fact, {\em classical} asymptotic analysis cannot distinguish between $y$ and the whole family of $y_a$'s.  (Contrast this with the fact that two different analytic functions cannot share the same Taylor series). 

As a consequence of this type of non-uniqueness, in Section \S\ref{Neg} we are able to show that a linear association between functions and {\bf general} divergent series requires a relatively strong consequence of the Axiom of Choice (and as such cannot be instituted based on a specific definition, something which will be the subject of another paper). Accordingly, {\em the class of divergent series needs to be restricted!} With this in mind, as was mentioned in the introduction, we {\em limit our analysis to a proper subclass of the resurgent functions}, a subclass that appears to be wide enough to contain those functions which occur commonly in applications. As such, from a practical standpoint, our restriction appears to be relatively mild.

In \S \ref{SecEBRT} we introduce the idea of a resurgent function and the closely related idea of a resurgent transseries. The resurgent  transseries are of particular importance to us since a unique association can be carried out in a constructive fashion between the class of resurgent divergent transseries, on the one hand, and  the class of  resurgent functions, on the other. For example, the resurgent function associated with the series in \eqref{eq:divser} is $e^{-x}$Ei$(x)$. Moreover, this association preserves all the local operations with which the summation of convergent Taylor series do.  We will use the just-said association to define our desired integrals for {\em the positive infinite case} invoking a pair of isomorphisms--one between a subclass of resurgent functions and a subspace of transseries, and the other between the just-said subspace of transseries and a class of functions on {\bf No.} It is through this pair of isomorphisms (see Figure 2) that we extend resurgent functions to infinite surreals and define their integrals. Moreover, the integrals so-defined on surreal extensions of resurgent functions (as well as on transseries) have the  properties specified in Proposition \ref{existint}.

{\begin{figure}[h!]
  \begin{center}
 \begin{tikzcd}[column sep=small]
  &\framebox{Transseries} \arrow{dr}{\tau} & \\
\framebox{Resurgent functions} \arrow{rr}[shorten <= 1pt]{\tau\circ\mathrm{Tr}} \arrow{ur}{{\rm Tr}:=({\mathcal{L}\circ\mathbf{mon}\circ\mathcal{B}})^{-1}} & & \framebox{Surreal functions} 
\end{tikzcd}
 \caption{{{\em The extension operator} $\mathsf{E}$ {\em restricted to the positive infinite case} is the composition of two intermediate  isomorphisms: {\em transseriation}, i.e. ${{\rm Tr}:=(\mathcal{L}\circ\mathbf{mon}\circ\mathcal{B})^{-1}}$, from a subspace of  resurgent functions to  a subspace of transseries, where $\mathcal{L}\circ\mathbf{mon}\circ\mathcal{B}$ is {\em \'Ecalle-Borel summation}, and a map $\tau$ from the just-said subspace of transseries to surreal functions. }}
\end{center}
\end{figure}}

We remind the reader that by convention we set the point where our functions have divergent expansions to be at the gap $\infty$ (see Footnote \ref{f2}), and as such the only gap past which defining integration is difficult is $\infty$ itself.

To prepare the way for our discussion of resurgent functions and resurgent transseries, in the following section we will first review some classical results  in the theory of Borel summability and the theory of transseries and then prove a new result (Proposition \ref{PB2}) concerning the existence of antiderivatives. Like Proposition \ref{PB2}, most of the material in \S\ref{SecTEB}  from subsection 5.5 on is new.

\section{Transseries, Borel summation and Borel summable subspaces of transseries}
\label{SecTEB} Typically, Borel summability and \'Ecalle-Borel summability deal with series of the form
\begin{equation}
  \label{eq:bsumbeta}
 \tilde{f}:=\sum_{k=M}^\infty c_k x^{-k\beta},\beta>0;\  M\in\mathbb Z
\end{equation}
where the coefficients $\{c_k\}_{k\ge M}$ and $\beta$ are real. The Borel sum of a finite sum is by definition the identity. Hence,   we can assume without loss of generality that $M=1$. 
\subsection{Classical Borel summation of series}\label{BSumDef}The following definition collects together some of the basic concepts and observations we will employ in this and subsequent sections.
\begin{Definition}[Laplace transform, Borel transform, Borel sum and critical time]\label{DefBE}

{\rm  For suitable functions $F$ for which the integral exists, the \emph{Laplace transform} $\mathcal{L}F$ of $F$ is defined as: $$(\mathcal{L}F)(x)=\int_0^\infty e^{-xp}F(p)dp.$$ The \emph{(formal) inverse Laplace transform} of a series $\tilde{f}=\sum_{k=0}^\infty c_{k} x^{-(k+1)\beta}$ is defined as a term-by-term transform of the series $$\mathcal{L}^{-1}\tilde{f}=\sum_{k=0}^\infty c_k p^{-k\beta-1}/\Gamma(k\beta),$$
where $\Gamma$ is the Gamma function; if $n$ is a positive integer, $\Gamma(n)=(n-1)!$.

   The \emph{Borel transform} $\mathcal{B}\tilde{f}$ of a formal series $\tilde{f}$ given by Equation \eqref{eq:bsumbeta} with $M=1$ (see the remarks following Equation \eqref {eq:bsumbeta}) is the series obtained by taking the term-by-term inverse Laplace transform of $\tilde{f}$ in normalized form. If $\beta=1$, then $\mathcal{B}\tilde{f}$ is analytic at $p=0$; otherwise it is ramified-analytic and $\mathcal{B}\tilde{f}=p^{-1}A(p^\beta)$ where $A$ is analytic. It is often relatively easy to reduce to the case  $\beta=1$, which we will assume in the following.

        The \emph{Borel sum} of $\tilde{f}$ along $\RR^+$  exists if after taking the Borel transform $\mathcal{B}\tilde{f}$ of $\tilde{f}$ the following two conditions are satisfied:

(i) The series $\mathcal{B}\tilde{f}$ is convergent, and its sum (by abuse of notation also written $\mathcal{B}\tilde{f}$) is analytic on $\RR^+$.\footnote{Mathematically, $\mathcal{B}\tilde{f}$ is a formal series, albeit convergent, and is distinct from its sum--a germ of an analytic function--which in turn is distinct from its analytic continuation on $\RR^+$. These distinctions are typically dropped whenever no confusion is possible. For instance, we write with a tacit license that $\mathcal{B}\tilde{f}$ is analytic on $\RR^+$.}

(ii) $\mathcal{B}\tilde{f}$ has exponential bounds on $\RR^+$, i.e., $\exists \nu>0\text{ such that}\; \sup_{p> \nu}|e^{-\nu p} (\mathcal{B}\tilde{f})(p)|<\infty$.

When this is the case, the Borel sum of $\tilde{f}$ is by definition $\mathcal{LB}\tilde{f}$.

For example,

\begin{equation}
  \label{eq:EQeIEX}
  \mathcal{LB}\sum_{k=0}^\infty k!(-1)^k x^{-k-1}=\mathcal{L}F=-e^x\text{Ei}(-x);\ \ F(p):=\frac{1}{1+p}.
\end{equation}

 The coefficients $c_k$ of asymptotic series occurring in applications have at most power-of-factorial growth $c_k\sim (k!)^p$ for some (usually integer) $p$. To apply Borel summation  or the more general \'Ecalle-Borel summation to a series of a factorially divergent series, one needs to {\em normalize} the series by passing to the  power of $x$ that ensures that the growth of the coefficient of $x^{-(k+1)\beta}$ is, to leading order, $\Gamma(k\beta)$. The power of the variable with respect to which this precise factorial growth is achieved is called {\em \'Ecalle critical time}. An illustration is provided by the asymptotic series of $e^{x^2}\mathrm{erfc}(x)$ as $x\to \infty$,
      \begin{equation}
        \label{eq:erfc}
        e^{x^2}\mathrm{erfc}(x)\sim \frac{1}{\sqrt{\pi} x}-\frac{1}{2\sqrt{\pi} x^3}+\frac{3}{4\sqrt{\pi}x^3}+\cdots=x^{-1}\sum_{k=1}^\infty \frac{c_k}{x^{2k}},
      \end{equation}
      where $\pi c_k=(-1)^{k}\Gamma(k-1/2)$. To ensure that the growth of the coefficients of the series matches the power of the variable as explained, we need to change the variable to $t=\sqrt{x}$. In this example the critical time is $t=x^{1/2}$.
}\end{Definition}

A calculation shows that
  \begin{equation}
    \label{eq:convoprod}
    \mathcal{B}(\tilde{f}\tilde{g})=(\mathcal{B}\tilde{f})*(\mathcal{B}\tilde{g}),
  \end{equation}
  where $``*"$ is the \emph{Laplace convolution}
  \begin{equation}
    \label{eq:lapconv}
    (F*G)(p)=\int_0^p F(s)G(p-s)ds.
  \end{equation}

  \begin{Proposition}[The space ${S}_{\mathcal{B}}$ of Borel summable series]\label{PBorel}{\rm Let ${S}_{\mathcal{B}}$ be the space of series which are Borel summable. Then:
  
  (i) ${S}_{\mathcal{B}}$ is a differential algebra (with respect to
    formal addition, multiplication, and differentiation of power
    series), and $\mathcal{LB}$ is an isomorphism of differential algebras. 

(ii) If ${S}_c\subset
  {S}_\mathcal{B}$ denotes the differential algebra of convergent power
  series, and we identify a convergent power series with its sum, then
  $\mathcal{LB}$ is the identity on ${S}_c$.

(iii) 
For $\tilde{f}\in S_B$ and $x$ in the open right half plane,
$\mathcal{L}\mathcal{B}\tilde{f}$ is asymptotic to $\tilde{f}$ as $|x|\rightarrow\infty$.

(iv) The subspace of $S_\mathcal{B}$ consisting of series whose Borel transforms are analytic in a disk around the origin and in a nonempty open sector is closed under composition. More precisely, if $\tilde{f}$ and $\tilde{g}$ are elements of this subspace,  then so is $\tilde{f}\circ(I+\tilde{g}),$ $I$ being the identity map.

(v) Borel summation is a proper extension of the usual summation. More precisely, if $\tilde{f}=\sum_{k\ge 1} c_k x^{-k}$ converges to $f$ in a neighborhood of $\infty$, then $\mathcal{B}\tilde{f}$ is entire, exponentially bounded and $\mathcal{LB}\tilde{f}=f$.

}\end{Proposition}
\begin{proof}
Statements (i)--(iii) and (v) are proved in (\cite{Book}, p. 106); and for the proof of (iv), see (\cite{Sauzin-Book} p. 159).
\end{proof}

\noindent
{\bf Note.} {\em Borel sums are analytic for large argument $x$.}  Standard arguments from complex analysis  (e.g. combining Morera's theorem with Fubini) show that $\mathcal{LB}\tilde{f}$ is {\em real analytic} for all sufficiently large $x\in\R$.

\begin{Definition}[Borel summation]\label{Notherpoints}{\rm 
The operator of \emph{Borel summation} is defined at any point $x_0\in \RR$ (or $\CC$) by moving $x_0$ to $\infty$, performing Borel summation at $\infty$ and moving the point at $\infty$ back to $x_0$. That is, we define $(\mathcal{LB})_{x_0}=M^{-1}\circ \mathcal{LB}\circ M$ where $M$ is the M\"obius transformation $x\mapsto x_0+x^{-1}$ (see also Definition \ref {defM}). 
}\end{Definition}

On Borel summed series that are $O(x^{-2})$, we now define an  operator having some of the properties of an antidifferentiation operator in the sense of Definition \ref{Dd2}.
\begin{Definition}\label{diff}{\rm 
  Let $S_{\mathcal{B};2}$ be the space of Borel summable series that are $O(x^{-2})$. Further, let $s\in S_{\mathcal{B};2}$, $S=\mathcal{LB}s$, and $\mathsf{A}_{\mathcal{B}}S=-\int_0^\infty p^{-1}e^{-xp}(\mathcal{B}s)(p)dp$.  Asymptotic series at infinity are particular cases of transseries at infinity to which $\mathsf{A}_{\mathcal{B}}$ is successively extended in \S\ref{S53}, \S\ref{S56} and \S\ref{SecEBRT}.
}\end{Definition}
We note that by the general properties of the Laplace transform we have $(A_{\mathcal{B}}f)'=f$ and $A_{\mathcal{B}}f=O(x^{-1})$ for large $x$. Hence,  $\mathsf{A_{\mathcal{B}}}f=\int_{\infty}^x S(t)dt$.
\begin{Proposition}\label{PADpowser}{\rm 
    $\mathsf{A}_{\mathcal{B}}$, as defined in Definition \ref{diff}, is well defined  on Borel sums of real-valued series and has Properties i--iii and vi from Definition \ref{Dd2}.

}\end{Proposition}
\begin{proof}
  If $s=O(x^{-2})$ for large $x$, then by definition, $(\mathcal{B}s)(p)=O(p)$ for small $p$. Since $\mathcal{B}s$ is analytic at zero, we have $\mathcal{B}s=pH(p)$ where $H$ is analytic at zero, and hence, $p^{-1}\mathcal{B}s=H(p)$ is analytic at zero as well.  Clearly, $\mathcal{B}s$ has analytic continuation on $\RR^+$ if and only if $H(p)$ is has analytic continuation on $\RR^+$. It is also straightforward to check that $\mathcal{B}s$ is exponentially bounded for large $p$ if and only if $H$ is exponentially bounded for large $p$. This establishes the existence of $\mathsf{A}_{\mathcal{B}}S$.

    Using the exponential bounds and dominated convergence we see that we can differentiate under the integral sign and get $(\mathsf{A}_{\mathcal{B}}S)'=S$, thereby establishing Property i of  Definition \ref{Dd2}. 

    Proposition \ref{PBorel} (iii) shows that if $s$ is positive (meaning that the coefficient of the highest power of $x$ is positive), $S$ is a positive function for large $x$. The positivity of the coefficient of the highest power of $x$ is equivalent to the positivity of $H(0)$, which in turn shows that $\mathsf{A}_{\mathcal{B}}S$ is negative and increasing for large $x$, establishing Property iii of  Definition \ref{Dd2}. Property ii of Definition \ref{Dd2}, i.e., linearity, follows from the linearity of $\mathcal{B}$ and $\mathcal{L}$, and of multiplication by $p^{-1}$. Property vi follows from the fact  noted above that $\mathsf{A}_{\mathcal{B}}f=\int_{\infty}^x S(t)dt$ and the fundamental theorem of calculus.  In fact, using the remark following Definition \ref{diff}, we have $C=0$.
\end{proof}

\subsection{Transseries: an overview}\label{STrans}
\begin{sloppypar}
 As was mentioned in the introduction, a \emph{transseries} over $\RR$ is a formal series built up from $\RR$ and a variable $x> \RR$ using  powers, exponentiation, logarithms and infinite sums.  \'Ecalle's classical construction of the ordered differential field $\mathbb{T}$ of transseries over $\RR$ is inductive, beginning with log-free transseries \cite{Ecalle1a}\footnote{Motivated by a problem of Tarski on the model theory of $(\RR, e^x)$,  Dahn and G\"oring \cite{DG} independently introduced $\mathbb{T}$ as an exponential ordered field.}. There have been  a number of alternative constructions since (e.g. \cite{ ADH1, DMM, SC, CostinT,  Joris, BM2}). For a self-contained introduction to transseries, see  \cite{Edgar}.
\end{sloppypar}
Transseries are formal series of the following form in the variables $\mu_1,\mu_2,...,\mu_n$, called {\em transmonomials}:
 \begin{equation}\label{defT}
\tilde{T}=   \sum_{\mathbf{k}>-M}c_{\mathbf{k}}\boldsymbol{\mu}^{\mathbf{k}} :=  \sum_{k_1,k_2,...,k_n>-M}c_{k_1,k_2,...,k_n}\mu_1^{k_1}\mu_2^{k_2}\cdots \mu_n^{k_n},
  \end{equation}
\noindent 
where the transmonomials are functions of $x$, the coefficients are members of $\RR$ and $M \in \ZZ$. The set of tuples of integers bounded below used as indices in \eqref{defT} are well-ordered lexicographically;  this indexation, which emphasizes the nature of the generators (transmonomials) is preferable, in the applications we are considering, to one using the corresponding ordinals.

Transseries have (exponential) \emph{heights} and {\em (logarithmic) depths} that emerge from their inductive construction, but in our discussion we will only be concerned with {\em log-free, height one}  and {\em height one, depth one} transseries, and these are characterized below in Definitions \ref{DEBstr} and \ref{NNN}, respectively. Since context should prevent confusion, we will freely write exponential and logarithmic terms in transseries using $e$ and $\log$, respectively.

In the case of transseries over $\RR$ the component terms in $\tilde{T}$ are descendingly well ordered with respect to the \emph{asymptotic order relation} $\gg$; for example, for the transseries $e^{x}+ x + \log{x} +1+ x^{-1}$ we have $e^x\gg x\gg \log x\gg 1 \gg x^{-1}$, where $a\gg b$ indicates that $a$ is large (i.e. infinitely large) compared with  $b$.

We say that a transseries $\tilde{T}$ is {\bf positive} if the largest transmonomial of $\tilde{T}$ with respect to $\gg$ has a positive coefficient, negative if $-\tilde{T}$ is positive, and $\tilde{T}=0$ if all of its coefficients are zero.

There is a striking similarity between transseries over the reals and surreal numbers written in normal form.  Aschenbrenner, van den Dries and van der Hoeven \cite{ADH3} have exhibited a canonical elementary embedding  $\iota$ of the ordered differential field $\mathbb{T}$ of transseries into $(\mathbf{No} , \partial)$ that is the identity on $\RR$ and sends $x$ to $\omega$, where $\partial$ is the derivation on $\mathbf{No}$ due to Berarducci and Mantova \cite{BM2}. By appealing to Berarducci and Mantova's construction of $\iota (\mathbb{T}):=\mathbb{R}((\omega))^{LE}$ \cite{BM2}, Ehrlich and Kaplan \cite{EK} have shown that $\mathbb{R}((\omega))^{LE}$ is initial. We will have more to say about $\mathbb{R}((\omega))^{LE}$ in \S\ref{GEN}.
   
The similarity between transseries over the reals and surreal numbers carries over to the fact that the topology generated by Conway's notion of absolute convergence (see \S\ref{S5}) is \emph{mutatis mutandis} the same as the following ``transseries topology'' in the space of transseries.

\begin{Definition}\label{defconv}{\rm 
  The \emph{transseries topology} on $\mathbb{T}$  (see \cite[p. 131]{Book},\cite{Edgar}) is defined by the following convergence notion. Let  $\sum_{\mathbf{k}>-M}c_{\mathbf{k}}^{[m]}\boldsymbol{\mu}^{\mathbf{k}}$ be a sequence of transseries, where the superscript $[m]$ designates the $m$th element of the sequence and $c_{\mathbf{k}}^{[m]}$ designates the sequences of coefficients of the $m$th element. Then,  
  \begin{equation}
    \label{eq:defc}
    \lim_{m\to\infty} \sum_{\mathbf{k}>-M}c_{\mathbf{k}}^{[m]}\boldsymbol{\mu}^{\mathbf{k}}= \sum_{\mathbf{k}>-M}c_{\mathbf{k}}\boldsymbol{\mu}^{\mathbf{k}} \text{ if and only if } \forall\mathbf{k}\,\exists n \text{ such that }\forall m>n, c_{\mathbf{k}}^{[m]}=c_{\mathbf{k}},
  \end{equation}
i.e., {\em if and only if all the coefficients eventually become  those of the limit transseries (rather than merely converge to them).}
 }\end{Definition}
 \subsection{Differentiation of transseries}\label{S-diff-trans} $\mathbb{T}$ is closed under differentiation, where differentiation of transseries is defined by induction on transseries height as termwise differentiation \cite{DMM, Joris, Book, Ecalle1a}. It is shown in \cite{Book} that the termwise differentiation of a transseries is convergent in the transseries topology.

\subsection{Integration of transseries}\label{S53} $\mathbb{T}$ is also closed under integration. More specifically, we have:

\begin{Proposition}[van den Dries, Macintyre and Marker \cite{DMM}] \label{DMM1}
{\rm There is an antidifferentiation operator on $\mathbb{T}$, henceforth $\mathsf{A}_{\mathbb{T}}$, that is unique up to an additive real constant.}
\end{Proposition}

An independent, alternative proof (in the same spirit) of the existence portion of Proposition \ref{DMM1} was later given in  \cite[p. 143, Proposition 4.221]{Book}.\footnote{The second author wishes to thank Lou van den Dries for helpful remarks on Proposition \ref{DMM1} and the relation between its proof and that of the above-mentioned result in \cite{Book}.} In the latter treatment, the operator $\mathsf{A}_{\mathbb{T}}$ is defined as the unique fixed point of a linear inhomogeneous equation whose linear part is contractive in a suitable sense (see \cite[Definition 4.186, p. 132]{Book}). While the definition is constructive, the expression of the operator is not explicit, in general.
\begin{Note}
  \rm{Although antidifferentiation in $\mathsf{A}_{\mathbb{T}}$ is unique up to a real constant, there is a natural choice of an antidifferentiation, the one whose values are transseries with zero finite part. The interpretation of this choice is that of integration from $\infty$,  the only point all one-point compactifications of $(1,\infty)$ have in common. However,  any other choice of real constant would lead to the same definite integration operator, since the integral is a difference of two antiderivatives, and the constants would cancel.}

\end{Note}
The following result collects together a number of simple consequences of the above results, taken from \cite[p. 143-144, Propositions 4.224-4.225]{Book}.
\begin{Proposition}
\label{Pintegr}{\rm 
The antidifferentiation operator $\mathsf{A}_{\mathbb{T}}$ on $\mathbb{T}$ from Proposition \ref{DMM1} has the following properties for all transseries ${\tilde{T}}, {\tilde{T}}_1$ and ${\tilde{T}}_2$:
$\mathsf{A}_{\mathbb{T}}$ is an antiderivative without constant terms, i.e., 
$$\mathsf{A}_{\mathbb{T}}{\tilde{T}} = L + s, $$
where $L$ is the purely infinite part of $\mathsf{A}_{\mathbb{T}}{\tilde{T}}$ (i.e. all terms in $L$ are $\gg 1$)  and $s$ is the small part of $\mathsf{A}_{\mathbb{T}} {\tilde{T}}$ (i.e $1\gg s$). Here $\mathsf{A}_{\mathbb{T}} {\tilde{T}}$ is written as a sum as in Equation \eqref{defT}.

\noindent Moreover,
\begin{align}
  \label{Epropint}
\mathsf{A}_{\mathbb{T}}( {\tilde{T}} _1+ {\tilde{T}} _2)=\mathsf{A}_{\mathbb{T}}{\tilde{T}} _1+\mathsf{A}_{\mathbb{T}}{\tilde{T}} _2,\nonumber\\
(\mathsf{A}_{\mathbb{T}}{\tilde{T}} )'= {\tilde{T}}, \ \mathsf{A}_{\mathbb{T}}{\tilde{T}} '= {\tilde{T}}_{\overline{0}}\nonumber,\\
\mathsf{A}_{\mathbb{T}}( {\tilde{T}} _1 {\tilde{T}} _2')= ({\tilde{T}} _1 {\tilde{T}} _2)_{\overline{0}}-\mathsf{A}_{\mathbb{T}}( {\tilde{T}} _1' {\tilde{T}} _2),\\
 \text{if\ } {\tilde{T}} _1\gg {\tilde{T}} _2,\text{then} \ \mathsf{A}_{\mathbb{T}}{\tilde{T}} _1\gg \mathsf{A}_{\mathbb{T}}{\tilde{T}} _2\nonumber,\\\nonumber
 \text{if\ }{\tilde{T}} >0 \text{ and } {\tilde{T}}\gg 1, \text{then} \ \mathsf{A}_{\mathbb{T}}{\tilde{T}} >0,
\end{align}
\noindent where ${\tilde{T}}_{\overline{0}}$ is the constant-free part of ${\tilde{T}}$, that is,
 $$\text{if\ } {\tilde{T}} =\sum_{{\bf k}\ge{\bf k}_0}c_{\bf k}\mu^{{\bf k}}, \text{\ then\ } {\tilde{T}}_{\overline{0}}=\sum_{{\bf k}\ge{\bf k}_0;{\bf k}\not =
  0}c_{\bf k}\mu^{{\bf k}}\;$$
\noindent  
and where $({\tilde{T}} _1 {\tilde{T}} _2)_{\overline{0}}$ is the transseries ${\tilde{T}} _1 {\tilde{T}} _2$ with the constant term chosen to be zero.
}\end{Proposition}
``Hands-on'' constructions of antiderivatives of special transseries that will concern us will be given in Subsection \ref{S:antider}.
\subsection{Some subspaces of transseries} 
In this section we introduce and analyze three spaces of transseries: $\mathbb{T}_{-}, \mathbb{T}_{\ell}$ and $\mathbb{T}_{+}$. Transseries in $\mathbb{T}_{-}$ occur as solutions of nonlinear ODEs, difference equations and a variety of other nonlinear problems. Transseries in  $\mathbb{T}_{+}$ arise in linear problems and  $\mathbb{T}_{\ell}$ is a space that is generated by repeated antidifferentiation. The minus subscript stands for the fact that all the arguments of the exponentials in the members of the space are nonpositive; the subscript ``$\ell$'' indicates the absence of exponentials, but possible presence of logarithms; and the plus subscript indicates that all the arguments of the exponentials in the members are positive.

The space $\mathbb{T}_{-}$ is actually a differential algebra.  Nonlinear problems rely on the algebraic structure, which we analyze. For the other two spaces we are only interested in their linear properties. The space  $\mathbb{T}_{\ell}\oplus \mathbb{T}_{-}$ is closed under antidifferentiation.

  \begin{Definition}[The space $\mathbb{T}_{-}$ of log-free, height one transseries]\label{DEBstr}
{\rm 
Let $n\in\NN$, $\boldsymbol{\beta},\boldsymbol{\lambda}$ be vectors in $\RR^n$,  with $\lambda_i>0$ for $i=1,...,n$,  and define
\begin{equation}\label{trans21}
  \tilde{T}_-= \sum_{\mathbf{k}\ge 0,l\ge 0}c_{\mathbf{k},l}x^{\boldsymbol{\beta}\cdot \mathbf{k}}e^{-\mathbf{k} \cdot \boldsymbol{\lambda} x}x^{-l}=   \sum_{\mathbf{k}\ge 0}x^{\boldsymbol{\beta}\cdot \mathbf{k}}e^{-\mathbf{k} \cdot \boldsymbol{\lambda} x}\tilde{y}_{\mathbf{k}}(x),  \end{equation}
where the $\tilde{y}_{\bf k}$ are formal power series which are $o(1/x)$ for large $x$. In applications in which $\tilde{y}_{\bf 0}$ starts with a constant, this constant can be subtracted out. To arrange that $\tilde{y}_{\bf k}=o(1/x)$ for all ${\bf k}\ne 0$ we can simply change $\beta_i$ to $\beta_i+1$ (with the effect of dividing $\tilde{y}_{\bf k}$ by $x^{|\bf k|}$).
We denote the space of such $\tilde{T}_-$ by $\mathbb{T}_{-}$.\footnote{Equation \eqref{trans21} depicts  a simple case of a level one transseries, also referred to as a {\em mixed series}. } The parameters $n$, $\boldsymbol{\lambda},\boldsymbol{\beta}$ depend on the transseries; when combining two transseries one first  needs to embed all of these  in a larger parameter space.

}\end{Definition}

The  condition in the above definition that the $\tilde{y}_{\bf k}$ are $o(1/x)$ for large $x$ is a useful convention because it ensures that the only common element of $\mathbb{T}_{\ell}$ and $\mathbb{T}_{-}$ is zero, and thereby leads to the uniqueness of decompositions expressed in Proposition \ref{T:uniquedec} and elsewhere. To achieve the same end, $o(x^{-m})$ or, equivalently, $O(x^{-(m+1)})$ could have been used for other values of $m\ge 1$. Our convention explains the choice we adopt in the sequel of writing expressions of the form  $e^{k x}$ with $k>0$ as $(e^{-x})^{-k}$, {\em as well as the fact that at times we have negative indices in sums} (see, for example, Definition \ref{D:plus}).

The condition $\lambda_i>0$ ensures that there is no infinite ascending chain of terms with respect to the asymptotic order relation.  The form expressed by Equation \eqref{trans21} is the most general type of log-free transseries occurring in usual applications.
\begin{Note}{\rm 
With $\RR^n$ replaced by $\CC^n$, Equation \eqref{trans21} represents the most general transseries solution of generic, normalized,  nonlinear systems of meromorphic ODEs. For such systems, $c_{\mathbf{k},l}$ are vectors, a generalization that can be easily dealt with. On the other hand, allowing for complex coefficients would pose various technical problems in our setting which we prefer to avoid. The term ``normalized'' refers to the fact that the exponentials are of the form $e^{-ax}$, that is, the exponents are linear in $x$. Had we started with $e^{-a x^b}$, $b\ne 1$, we would {\em normalize} the transseries by changing the variable to $t=x^b$ (also see Note \ref{Nothertimes}); it can be shown that $t$ thus defined {\em coincides} with the \'Ecalle critical time introduced in Definition \ref{DefBE}.
}\end{Note}

\begin{Proposition} \label{P:algebra} {\rm 
   The linear combination and multiplication of two transseries $\tilde{T}^{(1)}$ and  $\tilde{T}^{(2)}$ are defined as follows:   
 $$a^{(1)} \tilde{T}^{(1)}+a^{(2)}\tilde{T}^{(2)}= \sum_{\mathbf{k}\ge 0}x^{\boldsymbol{\beta}\cdot \mathbf{k}}e^{-\mathbf{k} \cdot \boldsymbol{\lambda} x}(a^{(1)} \tilde{y}^{(1)}_{\mathbf{k}}(x)+ a^{(2)} \tilde{y}^{(2)}_{\mathbf{k}}(x))$$ where $a^{(1)}$ and $a^{(2)}$ are real numbers.
 
  \begin{equation}   \label{eq:defmul}
   \tilde{T}^{(1)}\tilde{T}^{(2)}=    \sum_{\mathbf{k}\ge 0}x^{\boldsymbol{\beta}\cdot \mathbf{k}}e^{-\mathbf{k} \cdot \boldsymbol{\lambda} x}\sum_{\mathbf j=0}^{\bf k}\tilde{y}^{(1)}_{\mathbf{j}}(x)\tilde{y}^{(2)}_{\mathbf{k-j}}(x).
    \end{equation}
     With respect to these operations, $\mathbb{T}_{-}$ is a commutative algebra.
}\end{Proposition}

\begin{proof}
Straightforward verification.  
\end{proof}
Repeated antidifferentiation of elements in $\mathbb{T}_-$ results in polynomials combined with logs which generate the space $\mathbb{T}_{\ell}$ below,  which needs to be adjoined to our construction.
\begin{Definition}[The Space  $\mathbb{T}_{\ell}$]\label{D:Tell}{\rm 
Let $\mathbb{T}_{\ell}$ be the space of transseries of the form
  \begin{equation}
    \label{eq:eqlog}
  \tilde{T}=\sum_{k=0}^nc_k \mathsf{A}_\mathbb{T}^k(1/x)+R,
  \end{equation}
  where $n\in\NN\cup \{0\}$, $c_k\in\RR\ (k=0,...,n)$, $\mathsf{A}_{\mathbb{T}}^k(1/x)$ is the $k$th antiderivative without constant term of $1/x$, and $R$ is a polynomial of $1/x$ without constant term in $1/x$.
}\end{Definition}

\begin{Proposition}\label{N:log}{\rm 
  $\mathbb{T}_{\ell}$ is a space of functions that coincide with their transseries, and is closed under  $\mathsf{A}_{\mathbb{T}}$ (see Definition \ref{D:Tell}).  Moreover, each element of $\mathbb{T}_{\ell}$ can be written uniquely in the form
    $$ \tilde{T}_{\ell}=P(x)\log+Q(x)+R(x)$$
    where $P, Q$ and $R$ are polynomials and $R$ has no constant term. 
  }\end{Proposition}
\begin{proof}
  Straightforward: all these are elementary functions.
\end{proof}

 \begin{Definition}[The Space $\mathbb{T}_+$]\label{D:plus}{\rm 
For $j\in\{-M,...,-1\}$, let the $\lambda_{j}$ be a descending sequence of positive reals and let the $\beta_j$ and the $c_{j,l}$ be arbitrary sequences of reals. Subject to these conditions, further let
      \begin{equation}
        \label{eq:tplus}
        \tilde{T}_+=\sum_{-M\le j\le -1;\ l\ge 1}c_{j,l}x^{\beta_j}e^{\lambda_{j} x}x^{-l}=
  \sum_{j=-M}^{-1}x^{\beta_j}e^{\lambda_{j} x}{\tilde{y}}_j(x),
      \end{equation}
 where the $\tilde{y}_j$ are  formal power series in powers of $1/x$ satisfying $\tilde{y}_j=O(1/x)$ (as is implied by the expanded form of $\tilde{T}$ in the middle term in Equation \eqref{eq:tplus}). 
We denote the space of all transseries of type $\tilde{T}_+$ by $\mathbb{T}_{+}$.} 
\end{Definition}

\noindent{\bf{Comment.}} In the rightmost expression in Equation \eqref{eq:tplus}, integer powers of $x$ can be traded between $\tilde{y_j}$ and $x^{\beta_j}$, an ambiguity which is immaterial as the middle term in \eqref{eq:tplus} shows. 
\begin{Definition}[The space $\mathbb{T}_1$ of height one, depth one transseries]\label{NNN}{\rm 
Employing the notations from Definitions \ref{DEBstr}, \ref{D:Tell} and \ref{D:plus}, henceforth we denote by $\mathbb{T}_1$ the space $\mathbb{T}_{+}\oplus \mathbb{T}_{\ell}\oplus \mathbb{T}_{-}$.} 
\end{Definition}
  \begin{Proposition}\label{T:uniquedec}{\rm 
Every $\tilde{T}\in \mathbb{T}_1$ can be written uniquely in the form $\tilde{T}=\tilde{T}_{+}+\tilde{T}_{\ell}+ \tilde{T}_{-}$ where $\tilde{T}_{+}\in \mathbb{T}_{+}$, $\tilde{T}_{\ell}\in \mathbb{T}_{\ell}$ and $\tilde{T}_{-}\in \mathbb{T}_{-}$.
 } \end{Proposition}
  \begin{proof}
  This follows from Definition \ref{NNN}, the descendingly well ordering of the 
    component terms of a transseries, and the definitions of the three subspaces, the latter of which collectively imply $\tilde{T}_{+} \cap \tilde{T}_{\ell}=\tilde{T}_{\ell} \cap \tilde{T}_{-}=\tilde{T}_{+} \cap \tilde{T}_{-}=\{0\}$ and $\tilde{T}_{+} \gg \tilde{T}_{\ell} \gg \tilde{T}_{-} $ whenever these component transseries are nonzero.
  \end{proof}

 {\bf Differentiation.} It is easy to verify that, if $\tilde{y}$ is a power series, then
               \begin{equation}
                 \label{eq:difser}
                \left( x^\beta  e^{-\lambda x}\tilde{y}\right)'=x^\beta e^{-\lambda x}\Big[(\beta x^{-1}-\lambda )\tilde{y}+\tilde{y}'\Big]
               \end{equation}
               where $\tilde{y}'$ is the termwise differentiation of $\tilde{y}$.
               \begin{Note}\label{N:aboutpositivity}\rm{
  The right side of Equation \eqref{eq:difser} is negative since $\tilde{y}$ is a series with positive coefficients and, as is the case with any asymptotic power series, $\tilde{y}'\ll \tilde{y}$. 
               }\end{Note}
         \begin{Definition}\label{DdefP}{\rm Differentiation for the $\mathbb{T}_{\ell}$ component is termwise differentiation of the constituent monomials; see also Proposition \ref{N:log}. For the other two components,  it is defined as termwise differentiation,  namely,
             \begin{multline}
               \label{eq:eqdif}
                \Big(   \sum_{j=-M}^{-1}x^{\beta_j}e^{\lambda_{j} x}{\tilde{y}}_j(x)+\sum_{\mathbf{k}\ge 0}x^{\boldsymbol{\beta}\cdot \mathbf{k}}e^{-\mathbf{k} \cdot \boldsymbol{\lambda} x}\tilde{y}_{\mathbf{k}}(x)\Big)'\\=  \sum_{j=-M}^{-1}\left(x^{\beta_j}e^{\lambda_{j} x}{\tilde{y}}_j(x)\right)'+\sum_{\mathbf{k}\ge 0}\left(x^{\boldsymbol{\beta}\cdot \mathbf{k}}e^{-\mathbf{k} \cdot \boldsymbol{\lambda} x}\tilde{y}_{\mathbf{k}}(x)\right)'\\= \sum_{j=-M}^{-1}x^{\beta_j}e^{\lambda_{j} x}\left[(\beta_jx^{-1}+\lambda_{j})\tilde{y}_j+\tilde{y}{_j}'\right]\\+\sum_{\mathbf{k}\ge 0}x^{\boldsymbol{\beta}\cdot \mathbf{k}}e^{-\mathbf{k} \cdot \boldsymbol{\lambda} x}\left[(\boldsymbol{\beta}\cdot \mathbf{k}x^{-1}-\mathbf{k} \cdot \boldsymbol{\lambda})\tilde{y}+\tilde{y}'_{\mathbf{k}}(x)\right].
              \end{multline}
              
               Differentiation of an element of $\mathbb{T}_1$ is defined as the sum of the derivatives of its  $+,\ell \;{\rm and\;} -$ components.   
                       
 }\end{Definition}   
  The infinite sums in Equation \eqref{eq:eqdif} converge  in the transserries topology; for a proof see \cite{Book}.

  \subsubsection{The definition of the operator $\mathsf{A}_{\mathbb T}$ on $\mathbb{T}_-$} \label{S:antider} We first define $\mathsf{A}_{\mathbb{T}}$ on the individual components of the transseries, namely on $t_{\mathbf{k}}=x^{\boldsymbol{\beta}\cdot \mathbf k}e^{-\mathbf k\cdot \boldsymbol{\lambda}x } \tilde{y}_{\mathbf k}(x)$
 and on $t=x^{\beta_j}e^{\lambda_jx }\tilde{y}_j(x), j\in \{-M,...,-1\}$. To this end, we solve, in transseries, the ODE $\tilde{v}'=t$. The  terms  $t_{\mathbf{k}}$
and $t$ are treated very similarly, and we analyze only $t_{\mathbf{k}}$. If ${\bf k}=0$ and $\tilde{y}_0=\sum_{l\ge 2}c_l x^{-l}$, then $\tilde{v}=-\sum_{l\ge 2}(l-1)^{-1}c_l x^{-l+1}$.  If ${\bf k}\ne 0$ then the substitution $\tilde{v}_{\mathbf{k}}(x)=x^{\boldsymbol{\beta}\cdot \mathbf k}e^{-\mathbf k\cdot \boldsymbol{\lambda}x }e^{-\mathbf k\cdot \boldsymbol{\lambda}x }w(x)$ in the ODE
                 \begin{equation}
                   \label{eq:eqtk}\tilde{v}'_{\mathbf{k}}=t_{\mathbf{k}}
                     \end{equation}
brings it to the form
               \begin{equation}
                 \label{eq:intterm}
                 w'-\mathbf k\cdot (\boldsymbol{\lambda}-\boldsymbol{\beta} x^{-1})w=y,
               \end{equation}
               which has the power series solution $w(x)=\sum_{j\ge 1}w_jx^{-j}$, where the coefficients $w_j$ are uniquely determined by the recurrence relation
               \begin{equation}
                 \label{eq:recw}
\mathbf k\cdot \boldsymbol{\lambda} c_{j+1}    -(\mathbf k\cdot \boldsymbol{\beta} - j)c_j =-\tilde{y}_{\mathbf k,j};\ \ c_1=\frac{\tilde{y}_{\mathbf k,1}}{\mathbf k\cdot \boldsymbol{\lambda}}.
               \end{equation}
             Next, we define
               \begin{equation}
                 \label{eq:defint3}
                 \mathsf{A}_{\mathbb{T}}\left(x^{\boldsymbol{\beta}\cdot \mathbf k}e^{-\mathbf k\cdot \boldsymbol{\lambda}x } \tilde{y}_{\mathbf k}(x)\right)= x^{\boldsymbol{\beta}\cdot \mathbf k}e^{-\mathbf k\cdot \boldsymbol{\lambda}x }w(x),
               \end{equation}
               where $w(x)$ is characterized as above.

 \subsubsection{The  definition of the operator $\mathsf{A}_{\mathbb T}$ on $\mathbb{T}_+$}\label{S:antider+} To define $\mathsf{A}_{\mathbb{T}}\left(x^{\beta_j}e^{\lambda_j}{\tilde{y}}_j(x)\right)$  we proceed as in \S\ref{S:antider}: we write a differential equation $x^{\beta_j}e^{\lambda_j}{\tilde{w}}_j(x)=x^{\beta_j}e^{\lambda_j}{\tilde{y}}_j(x)$, and obtain
  $$ \left(\frac{\beta_j}{x}+\lambda_j \right) \tilde{w}_j \! \left(x \right)+\frac{d}{d x}\tilde{w}_j \! \left(x \right)-\tilde{y}_j \! \left(x \right)=0.$$

\noindent  
Writing  $\tilde{y}_j=\sum_{j=1}^\infty d_j x^{-j}$, the coefficients $\{c_m\}_{m\in\NN}$ of the power series $\tilde{w}_j$ satisfy  the recurrence relation $c_{m}=\lambda_j^{-1} \left[d_m+(m-1-\beta_j)c_{m-1}\right];\ m\ge 1; c_{0}=0$. This shows existence and uniqueness of a solution with zero constant term.

Using Proposition \ref{N:log} and the results in \S\ref{S:antider} and \S\ref{S:antider+} we now extend antidifferentiation to $\mathbb{T}_1$.
               \begin{Definition}[Definition $\mathsf{A}_{\mathbb T}$ on $\mathbb{T}_1$] {\rm 
                  $\mathsf{A}_{\mathbb{T}}$ is defined by linearity  on $\mathbb{T}_1= \mathbb{T}_+\oplus\mathbb{T}_{\ell}\oplus\mathbb{T}_-$, by writing               
                \begin{equation}
                 \label{eq:genA}
 \mathsf{A}_{\mathbb{T}}\tilde{T}:= \sum_{j=-M}^{-1}\mathsf{A}_{\mathbb{T}}\left(x^{\beta_j}e^{\lambda_j}{\tilde{y}}_j(x)\right)+\mathsf{A}_{\mathbb{T}}\tilde{T}_{\ell} + \sum_{\mathbf{k}\ge 0}\mathsf{A}_{\mathbb{T}}\left(x^{\boldsymbol{\beta}\cdot \mathbf{k}}e^{-\mathbf{k} \cdot \boldsymbol{\lambda} x}\tilde{y}_{\mathbf{k}}(x)\right), 
\end{equation}
\noindent

} \end{Definition}

The infinite sum defined above in Equation \ref{eq:genA} converges in the transserries topology;  a general proof is provided in \cite{Book}.  The derivative and the antiderivative are inverses of each other.

                    \begin{Proposition}{\rm 
     Replacing the functions with elements of $\mathbb{T}_1$ everywhere in Definition \ref{Dd2}, the operator $\mathsf{A}_{\mathbb{T}}$ restricted to $\mathbb{T}_1$ satisfies the properties i--iv and vi listed there. 
               }\end{Proposition}
               \begin{proof}
                 The proof is a straightforward verification.
               \end{proof}

\subsection{Watson's Lemma} The following classical result is essential in determining the asymptotic behavior of Laplace transforms.
 \begin{Lemma}[Watson's Lemma (see, e.g., \cite{Book}, p. 31)] \label{L:Watson} {\rm Assume  that $F$ is locally integrable and exponentially bounded on $\RR^+$, $a,b>0$ and  $F(p)\sim \sum_{k=0}^\infty f_k p^{ka+b}$ for small $p>0$. Then
   \begin{equation}
     \label{eq:watson}
     \int_0^\infty e^{-xp}F(p)dp\sim\sum_{k\ge 0}\frac{f_k\Gamma(ka+b+1)}{x^{ka+b+1}}\ \ \ \text{as}\  x\to \infty.
   \end{equation}
 }\end{Lemma}

  \subsection{Borel summable subspaces of transseries}\label{S:Trans-Borel-summable} 
                       
             \begin{Definition}[The Borel summable subspace $\mathbb{T}_{\mathcal B}$ of $\mathbb{T}_1$]\label{D26}
               {\rm  We say that a transseries  is \emph{Borel-summable} if all power series  $\tilde{y}_{\bf k}$ and $\tilde{y}'_j(x)$ in Equation \eqref{trans21} are Borel summable and there are positive constants $c_1,c_2,c_3$ (which may depend on $\tilde{T}$) such that for all $\mathbf{k}$ and $p\in\RR^+$ we have
                 \begin{equation}
                   \label{eq:condconv}
                   \left|(\mathcal{B} \tilde{y}_{\mathbf{k}})(p)\right|\le c_1 c_2^{|\mathbf{k}|} e^{c_3p}\  \text{and}\ \left|(\mathcal{B} \tilde{y}_{j})(p)\right|\le c_1  e^{c_3p}.
                 \end{equation}
                 In view of the summability results we rely on in the sequel, we
                 impose the \emph{nonresonance} condition
             \begin{equation}  \label{eq:nonres2}
                                         (\mathbf{k-k'})\cdot\boldsymbol{\lambda}+\lambda_i-\lambda_j=0\ \text{if and only if}\   \mathbf{k-k'}=0 \ \text{and}\ i=j\  \text{for}\ i,j\in\{-M,...,-1\}; 
                 \end{equation}
                                  that is, the condition that linear combinations  of the exponents with integer coefficients permitted by our assumptions can only vanish trivially. 
                 
                       Henceforth, by $\mathbb{T}_{\mathcal B}$ we mean the subspace of $\mathbb{T}_1$ whose members are Borel summable. We write $\mathbb{T}_{+,\mathcal{B}},\mathbb{T}_{-,\mathcal{B}}$ for the  Borel summable subspaces of $\mathbb{T}_{+},\mathbb{T}_{-}$. By Proposition \ref{N:log} (a), we may identify $\mathbb{T}_{\ell,\mathcal B}$ with $\mathbb{T}_{\ell}$  and write $\mathcal{LB}\tilde{T}_\ell=\tilde{T}_\ell$.
}\end{Definition}
For clarity of notation we do not follow the multiindex convention, but, instead,  by $|\mathbf k|$ we mean $\sqrt{k_1^2+\cdots+k_n^2}$. 
\begin{Note}
  {\rm  \begin{enumerate}
      \item[(a)] The assumption that  all power series  $\tilde{y}_{\bf k}$ in Equation \eqref{trans21} are Borel summable does {\bf not} hold, generically, for nonlinear systems of ODEs. Instead, these series are {\em resurgent}, a case we study in the next section.
      
        \item[(b)] Using \eqref{eq:nonres2} we have the linear ordering $\mathbf k_1>\mathbf k_2$ if and only if $\mathbf k_1\boldsymbol\lambda >\mathbf k_2\boldsymbol\lambda$. By the discussion at the beginning of \S \ref{STrans}, and assuming the formal power series below are nonzero, we have: if $\lambda_1>\lambda_2$, then $x^{\beta_1} e^{\lambda_1}\tilde{y}_1(x) \gg x^{\beta_2} e^{\lambda_2}\tilde{y}_2(x) $, and if $\mathbf k_1>\mathbf k_2$, then $x^{\boldsymbol{\beta}\cdot \mathbf{k}_1}e^{-\mathbf{k}_1 \cdot \boldsymbol{\lambda} x}\tilde{y}_{\mathbf{k}_1}(x)\ll x^{\boldsymbol{\beta}\cdot \mathbf{k}_2}e^{-\mathbf{k}_2 \cdot \boldsymbol{\lambda} x}\tilde{y}_{\mathbf{k}_2}(x) $.
  \end{enumerate}
 }\end{Note}
    \begin{Proposition}\label{N27}
      {\rm 
      \begin{enumerate}
      \item[(a)] There exist positive constants $c_1,c_2,c_3$ such that for all $ x>r>c_3$, all $j<0$, and all $\mathbf k$ we have  \begin{equation}
        \label{eq:converg}
        |(\mathcal{LB} \tilde{y}_{\mathbf{k}})(x)|\le c_1c_2^{|\bf k|}(x-c_3)^{-1}
        \text{\ and\,}    | (\mathcal{LB} \tilde{y}_j)(x)|\le c_1(x-c_3)^{-1}   \end{equation}
     Moreover, if $c\ne 0$ and $\tilde{y}=cx^{-m}(1+o(1))$, then  $ \mathcal{LB}[ e^{\lambda x}x^\beta \tilde{y}]=cx^{-m}e^{\lambda x}x^\beta (1+o(1))$, for some $c\in\RR^+$.
   
       \item[(b)] Let $\lambda_1,...,\lambda_n\in\RR^+$, $\beta_1,...,\beta_n\in\RR$, $\lambda=\min\{\lambda_1,...,\lambda_n\}$ and $\beta=\max\{\beta_1,...,\beta_n\}$. Also let $x_0$ be such that for all $x>x_0$ we have $c_3e^{-\lambda x}x^\beta <1$. Then,  for all $ x>\max\{c_3,x_0\}$  we have
        \begin{equation}
        \label{eq:converg}
    \sum_{\mathbf{k}>M}x^{\boldsymbol{\beta}\cdot \mathbf{k}}e^{-\mathbf{k} \cdot \boldsymbol{\lambda} x}\left|\mathcal{LB}\tilde{y}_{\mathbf{k}}\right|\le c_1(c_2x^\beta e^{-\lambda x})^{NM}\frac{1}{(1-c_2x^\beta e^{-\lambda x})^N}(x-c_3)^{-1},
  \end{equation} where $N \in \mathbb{N}$, $M\in\ZZ$. In particular, the infinite sum converges uniformly and absolutely (in the analytic sense) for $x>x_0$ if $x_0$ satisfies $c_2x_0^\beta e^{-\lambda x_0}<1$.
  \item[(c)] If, in addition,  $M'\le M$ and $\beta=\min\{\beta_0,...,\beta_{M'}\}$,
  \begin{equation}
  \label{eq:T+est}
  \sum_{j=-M'}^{-1}x^{\beta_j}e^{\lambda_{j} x}|\mathcal{LB}{\tilde{y}}_j(x)|\le (M'+1)c_1e^{-\lambda_{M'}}x^{\beta}(x-c_3)^{-1}. 
\end{equation}
\item[(d)] $\mathbb{T}_{-,\mathcal B}$ is an algebra, i.e., if $\tilde{T}^{(1)}$ and $\tilde{T}^{(2)}$ are elements of $\mathbb{T}_{-,\mathcal B}$ and $a^{(1)},a^{(2)}\in\RR$, then so are $a^{(1)} \tilde{T}^{(1)}+a^{(2)} \tilde{T}^{(2)}$ and $\tilde{T}^{(1)}\tilde{T}^{(2)}$.

\end{enumerate}
       }\end{Proposition}
       \begin{proof}
         For the first part of (a) we simply note that, by assumption, $|\mathcal{LB} \tilde{y}_{\bf k}|\le c_1 c_2^{|{\bf k}|}\mathcal{L}(e^{c_3 p})=c_1 c_2^{|{\bf k}|}(x-c_3)^{-1}$, while the second part follows from Proposition \ref{PBorel} and Lemma \ref{L:Watson}. For (b), by assumption and using (a), we majorize each term in the infinite sum in \eqref{eq:converg} by $c_1 (c_2 x^\beta e^{-\lambda})^{k_1+...+k_N}(x-c_3)^{-1}$ (the terms of a geometric series) and the result follows.  The proof of (c) is similar and, in fact, simpler: we estimate a finite sum in terms of its largest term.

         For (d), if the constants in the bounds expressed in \eqref{eq:condconv} for $\tilde{T}^{(i)},i=1,2$ are  $c_k^{(i)}$, where $k=1,2,3$ and $i=1,2$, and if by $c_k$ we denote $\max\{c_k^{(1)},c_k^{(2)}\}$, then a bound of the type \eqref{eq:condconv} for $\tilde{T}^{(1)}+\tilde{T}^{(2)}$ is $(c_1(|a^{(1)}|+|a^{(2)}|)c_2^{|\mathbf k|}e^{c_3 p}$.   By linearity and the polarization identity $2ab=(a+b)^2-a^2-b^2$, for the product it is enough to show that $\mathbb{T}_{\mathcal{B}}$ is closed under squaring. If $\tilde{y}$ satisfies $|\mathcal{B}\tilde{y}|\le e^{c p}$ for $p\in\RR$, then  \eqref{eq:convoprod} implies $\mathcal{B}|\tilde{y}^2|\le {B}|\tilde{y}|*{B}|\tilde{y}|\le p e^{c p}\le e^{(c+1)p}$. Then, using this inequality and estimating the number of terms in the innermost sum in Equation \eqref{eq:defmul} by the rough bound $|\mathbf k|^N$ and using  $|\mathbf k|^N \le e^{N|\mathbf k|}$, we get
         $$\left|\mathcal{B}\sum_{\mathbf j=0}^{\bf k}\tilde{y}_{\mathbf{j}}(x)\tilde{y}_{\mathbf{k-j}}(x)\right|\le c_1^2 e^{(c_3+1)p}(c^2_2 e^N)^{|\mathbf k|},$$
         from which the result follows. 
           \end{proof}
    
\begin{Definition}\label{D:LEBstr}  {\rm Let $\tilde T \in \mathbb{T}_{\mathcal B}$ and let $\lambda$, $\beta$ and $x_0$ be as in Proposition \ref{N27}.  Then the Borel sum of $\tilde{T}$ is defined as   
        \begin{equation} 
  \label{eq:BsumT}
  \mathcal{LB}\tilde{T}=\sum_{j=-M}^{-1}x^{\beta_j}e^{\lambda_j x}\mathcal{LB}{\tilde{y}}_j+\sum_{\mathbf{k}\ge 0}x^{\boldsymbol{\beta}\cdot \mathbf{k}}e^{-\mathbf{k} \cdot \boldsymbol{\lambda}_j x}\mathcal{LB}\tilde{y}_{\mathbf{k}} +\tilde{T}_{\ell},
\end{equation}
(see Proposition \ref{N:log} and Definition \ref{D26}). We note that,  by \eqref{eq:converg} and the Weierstrass M-test, the infinite series in Equation \eqref{eq:BsumT} converges uniformly and absolutely (in the analytic sense of convergence)  on the interval
$[x_0,\infty)$.}\end{Definition}

\begin{Proposition}\label{P:bijection}{\rm 
  \begin{enumerate}
  \item[(a)] If $\tilde{T}\in \mathbb{T}_{\mathcal B}$, then $\tilde{T}>0$  if and only if $\mathcal{LB}\tilde{T}>0$ for sufficiently large $x$.
    
    \item[(b)] The kernel of $\mathcal{LB}$ is zero, i.e.,  $\{\tilde{T}\in \mathbb{T}_{\mathcal{B}}: \mathcal{LB}\tilde{T}=0\}=\{0\}$.
  \end{enumerate}
  
}\end{Proposition}
\begin{proof}
  (a). If $\tilde{y}_j\ne 0$ for some $j<0$, then we choose the most negative $j$ with this property, and for $\tilde{y}_j$ to be nonzero there must exist an $m\in\NN$  and a nonzero $c\in\CC$ such that $\tilde{y}_j=c x^{-m}(1+o(1))$. Using Proposition \ref{N27} (a), we see that $\mathcal{LB}\tilde{T}= ce^{\lambda_j x}x^{\beta_j}x^{-m}(1+o(1))$ and the result follows. The proof is very similar if instead $\tilde{y}_j=0$ for all $j<0$ and $P$ or $Q$ is nonzero, or, if $P$ and $Q$ are also zero and for some $ \mathbf{k}$ we have $\tilde{y}_{\bf k}\ne 0$.  

  (b) follows immediately from (a).
  \end{proof}
  
  \begin{Definition}\label{D:decomp-m}
  {\rm  Let $\mathcal{F}_{\mathcal{B}}$ denote the function space  $\{\mathcal{LB}\tilde{T}:\tilde{T}\in\mathbb{T}_{\mathcal{B}}\}$. Mimicking the notation in Definition \ref{NNN} and Proposition \ref{T:uniquedec}, we write:  $\mathcal{F}_{\mathcal{B}}=\mathcal{F}_{+,\mathcal{B}}\oplus \mathcal{F}_{\ell}\oplus \mathcal{F}_{-,\mathcal{B}}$, with $\mathcal{F}_{+,\mathcal{B}}$, $\mathcal{F}_{\ell}$ and $\mathcal{F}_{-,\mathcal{B}}$ understood in the expected manner.}\end{Definition}
 Note that, by Definition \ref{N:log}, Borel summability is the identity on $\mathcal{F}_{\ell}$, and as such we could have written equivalently  $\mathcal{F}_{\ell,\mathcal{B}}$ in place of $\mathcal{F}_{\ell}$.
  \begin{Corollary}\label{C:bijection}{\rm
    \item [(a)] $\mathcal{LB}$ is a  bijection between $\mathbb{T}_{\mathcal{B}}$ and $\mathcal{F}_{\mathcal B}$. \item [(b)] $\mathcal{LB}$ (restricted to $\mathbb{T}_{-,\mathcal{B}}$) is a linear and multiplicative bijection from $\mathbb{T}_{-,\mathcal{B}}$ to $\mathcal{F}_{-,\mathcal{B}}$. 
}\end{Corollary}

\begin{proof}
 (a) is an immediate consequence of Proposition \ref{P:bijection} and Definition \ref{D:decomp-m}. For (b), bijectivity follows from (a) and the definition of $\mathcal{LB}$; and linearity and multiplicativity follow from the fact that $\mathbb{T}_{-,\mathcal{B}}$ is closed under addition and multiplication and a straightforward calculation using the definition of $\mathcal{LB}$ and the fact that $\mathcal{LB}$ is linear and multiplicative on $\mathcal{S}_{\mathcal{B}}$. 
\end{proof}

Note that in virtue of Proposition \ref{T:uniquedec} and Corollary \ref{C:bijection}  the decomposition in Definition \ref{D:decomp-m} is unique.
 
\subsection{Differentiation and antidifferentiation on $\mathbb{T}_{\mathcal{B}}$} \label{S56}
\begin{Lemma}[Differentiation]\label{L:28}{\rm 
   If $\tilde{y}$ is a Borel summable formal power series, then  \begin{enumerate}
    \item $\left(e^{-\lambda x}x^\beta \mathcal{LB}\tilde{y}\right)'=e^{-\lambda x}x^\beta \mathcal{LB}\dot{\tilde{y}}$, where
      \begin{equation}
        \label{eq:eqdefdif}
        \dot{\tilde{y}}=(\beta x^{-1}-\lambda)\tilde{y}+\tilde{y}'\text{ and }\mathcal{B} \dot{\tilde{y}}=\beta  \left(\int_0^p\mathcal{B}\tilde{y}\right)-\lambda \mathcal{B}\tilde{y}-p\mathcal{B}\tilde{y}.
      \end{equation}

    \item Let $\tilde{y}_{\bf k}$ be as in Definition \ref{D26}. Then, for some constants $c'_1,c'_2,c'_3$ depending only on $c_1,c_2,c_3,$ and all $p\in\RR^+$ and ${\bf k}$ we have
      \begin{equation}
        \label{eq:eqfk}
        \left|(\mathcal{B} \dot{\tilde{y}}_{\mathbf{k}})(p)\right|\le c'_1 {c'_2}^{|\mathbf{k}|} e^{c'_3p}.
      \end{equation}
      
    \item The sum
     \begin{equation}
      \label{eq:fk2}
      \sum_{\mathbf{k}\ge 0}x^{\boldsymbol{\beta}\cdot \mathbf{k}}e^{-\mathbf{k} \cdot \boldsymbol{\lambda} x}\mathcal{LB}\dot{\tilde{y}}_{\mathbf{k}}
    \end{equation}
    converges uniformly and absolutely in the analytic sense, for large $x$, and with $\tilde{T}$ as in Definition \ref{D:LEBstr} we have
    \begin{equation}
      \label{eq:BsumTprime}
      (\mathcal{LB}\tilde{T})'= \sum_{j=-M}^{-1}x^{\beta_j}e^{\lambda_j x}\mathcal{LB}{\dot{\tilde{y}}}_j+\sum_{\mathbf{k}\ge 0}x^{\boldsymbol{\beta}\cdot \mathbf{k}}e^{-\mathbf{k} \cdot \boldsymbol{\lambda} x}\mathcal{LB}\dot{\tilde{y}}_{\mathbf{k}}+\tilde{T}_{\ell}'=\mathcal{LB}(\tilde{T})'.
    \end{equation}
     
    \end{enumerate}
 }\end{Lemma}
\begin{proof}
The fact that $\dot{\tilde{y}}$ is given by the first equation in \eqref{eq:eqdefdif}   follows from the isomorphisms induced by Borel summation (see Proposition \ref{PBorel}), \emph{viz}:   \begin{multline}
        \label{eq:dif1}
        \left(e^{-\lambda x}x^\beta \mathcal{LB}\tilde{y}\right)'=e^{-\lambda x}x^\beta \left[\beta  x^{-1}\mathcal{LB}\tilde{y}-\lambda \mathcal{LB}\tilde{y}+\mathcal{LB}(\tilde{y}')\right]\\=e^{-x}x^\beta \mathcal{L}\left[\beta \left(\int_0^p\mathcal{B}\tilde{y}\right) \lambda \mathcal{B}\tilde{y}-p\mathcal{B}\tilde{y}\right].
      \end{multline}
    For part (2), we note that $\int_0^p |\mathcal{B}y_{\bf k}|\le e^{c_3 p}\int_0^p 1=pe^{c_3 p}\le e^{(c_3+1)p}$. The absolute value of the term $-p\mathcal{B}\tilde{y}$ is also bounded by $p|\mathcal{B}y_{\bf k}|\le e^{(c_3+1)p}$. Next, $|\boldsymbol{\beta}||\mathbf{k}|+|\mathbf{k}||\boldsymbol{\lambda}|\le \exp[|\mathbf{k}|(|\boldsymbol{\beta}|+|\boldsymbol{\lambda}|)]$, and so the result follows. Using (2), uniform and absolute convergence, in the analytic sense, are shown as in the first sentence of the proof of Proposition \ref{N27}. The rest follows from an elementary theorem about sequences of functions \cite[p. 321]{Folland}. (The estimates above can be improved substantially, but we do not need this here.)
\end{proof}
The Corollary below is an immediate consequence of Lemma \ref{L:28} and Corollary \ref{C:bijection}.
\begin{Corollary}[Preservation of differentiation]\label{C:29}{\rm 
  The space $\mathbb{T}_{\mathcal{B}}$ is closed under differentiation and, for $\tilde{T}\in\mathbb{T}_{\mathcal{B}}$, we have $(\mathcal{LB}\tilde{T})'=\mathcal{LB}(\tilde{T})'$; $\mathcal{LB}$ is a differential space isomorphism. Restricted to $\mathbb{T}_{-,\mathcal B}$, $\mathcal{LB}$ is a differential algebra isomorphism.
}\end{Corollary}
In the following definition we extend the operator $\mathsf{A_\mathcal{B}}$ of Definition \ref{diff} and Proposition \ref{PADpowser} to include  $\mathcal{F}_{\ell}$ and $\mathcal{F}_{-,\mathcal{B}}$.

\begin{Definition}\label{PB2D}
{\rm For  $\tilde{T} \in \mathbb{T}_{\ell}\oplus  \mathbb{T}_{-,\mathcal{B}}$,  $\mathsf{A}_{\mathbb{T}_{\mathcal{B}}}(\mathcal{LB}\tilde{T}):=\mathcal{LB}(\mathsf{A}_{\mathbb{T}}\tilde{T})$.
}\end{Definition}
As the next proposition shows, $\mathsf{A}_{\mathbb{T}_{\mathcal{B}}}$ is well-defined on the image of $ \mathbb{T}_{\ell}\oplus  \mathbb{T}_{-,\mathcal{B}}$ under $ \mathcal{LB}$, that is, on $\mathcal{F}_{\ell}\oplus  \mathcal{F}_{-,\mathcal{B}}$,  and takes values in $\mathcal{F}_{\ell}\oplus  \mathcal{F}_{-,,\mathcal{B}}$. 
\begin{Proposition}[Antidifferentiation]\label{PB2}{\rm
    \begin{enumerate}
      \item[(a)] The operator $\mathsf{A}_{\mathbb{T}_{\mathcal{B}}}:\mathcal{F}_{\ell}\oplus  \mathcal{F}_{-,\mathcal{B}}\to \mathcal{F}_{\ell}\oplus  \mathcal{F}_{-,\mathcal{B}}$ is well defined.
      \item [(b)]  $\mathsf{A}_{\mathbb{T}_{\mathcal{B}}}$ satisfies Properties i--iv and vi of Definition \ref{Dd2}.
         \item [(c)] The space  $\mathcal{F}_{\ell}\oplus \mathcal{F}_{-,\mathcal B}$ is closed under differentiation and antidifferentiation. 
    \end{enumerate}
}\end{Proposition}

     \begin{proof}

 Clearly, we only need to check the statement on $\mathcal{F}_{-,\mathcal{B}}$. For ${\bf k}=0$, (a) and (b) follow from Definition \ref{diff} and Proposition \ref{PADpowser}. Next, we show (a)  for a term of the form $x^{\boldsymbol{\beta}\cdot \mathbf{k}}e^{-\mathbf{k} \cdot \boldsymbol{\lambda}_j x}\mathcal{LB}\tilde{y}_{\mathbf{k}}$ with $\mathbf{k}\ne 0$. Using the results in \S\ref{S:antider}, we need to prove that the solution of the ODE \eqref{eq:recw} is Borel summable with bounds as in Definition \ref{D26}. These bounds are needed to prove absolute and uniform convergence of the resulting infinite series, as in Proposition \ref{N27}.

        Taking the Borel transform of \eqref{eq:intterm} and letting $W=\mathcal{B}w$ and $Y=\mathcal{B} y_{\mathbf{k}}$, we get
           \begin{equation}
             \label{eq:Borelt1}
        (\mathbf{k}\cdot\boldsymbol{\lambda}+p)    W(p)=\mathbf{k}\cdot\boldsymbol{\beta}\int_0^pW(s)ds+Y(p).
           \end{equation}
          After differentiation in $p$, we get a first order linear ODE that  can be easily solved by quadratures. However, the estimates we need are more difficult to obtain from the explicit solution, and we use a different approach here.
           We rewrite Equation \eqref{eq:Borelt1} in the form
                       \begin{equation}
             \label{eq:Borelt3}
      W(p)=\frac{\mathbf{k}\cdot\boldsymbol{\beta}}{\mathbf{k}\cdot\boldsymbol{\lambda}+p}\int_0^pW(s)ds+\frac{Y(p)}{\mathbf{k}\cdot\boldsymbol{\lambda}+p}. 
    \end{equation} Choose now
  \begin{equation}
    \label{eq:condc}
    c_3'>\sup\left\{\frac{|\mathbf{k}\cdot\boldsymbol{\beta}|}{\mathbf{k}\cdot\boldsymbol{\lambda}+1},c_3;\mathbf{k}\ge 0\right\}.
  \end{equation}
  Let $\mathcal{D}$ be the domain of analyticity of $Y$. It is easy to check that Equation \eqref{eq:Borelt3} is contractive in the Banach space $$\{f \,\text{analytic in $\mathcal{D}$}:\|f\|<\infty,\ \text{where }\|f\|=\sup_{p\in\mathcal{D}}e^{-c_3' |p|}\big|f(p)\big|\}.$$ It follows that the $ \mathcal{LB}w_{\mathbf{k}}$ exist and satisfy the same estimates as the $\mathcal{LB}\tilde{y}_{\mathbf{k}}$ with the triple $(c_1,c_2,c_3)$ replaced by $(c_1,c_2,c'_3)$ with $c'_3$ as in \eqref{eq:condc}, and (a) follows. Part (b) is a consequence of Corollary \ref{C:29} and of the bounds in terms of $(c_1,c_2,c'_3)$ (obtained in (a)) which imply uniform and absolute convergence of the infinite series. And Part (c) is immediate, since $\mathcal{F}_{\ell}$ is closed under differentiation and antidifferentiation and, by the analysis above, so is $\mathcal{F}_{\ell}\oplus \mathcal{F}_{-,\mathcal{B}}$ where differentiation and antidifferentiation require switching a term of the form $a/x^m$ between these two spaces.

  We now turn to (b) for general $\mathbf{k}$. For Property i, first note that we have already shown that, for $\tilde{T}\in  \mathbb{T}_{\ell}\oplus  \mathbb{T}_{-,\mathcal{B}}$ the series through which we defined $\mathsf{A}_{\mathbb{T}_{\mathcal{B}}}\tilde{T} $ is uniformly and absolutely convergent (in the analytic sense). On the other hand, the series whose terms are the derivatives of the terms of $\mathsf{A}_{\mathbb{T}_{\mathcal{B}}}\tilde{T}$ converges uniformly and absolutely to $\mathcal{LB}\tilde{T}$, simply because these terms are, by construction, the terms of $\mathcal{LB}\tilde{T}$. The rest follows from the elementary theorem about sequences of functions \cite[p. 321]{Folland} referred to before. Property ii (i.e. linearity) is immediate. For Property iii (positivity), to understand the monotonicity of $\mathsf{A}_{\mathbb{T}_{\mathcal{B}}}\tilde{T}$, we appeal to Proposition \ref{P:bijection} and Lemma \ref{L:28} to conclude that we only have to examine the dominant term of the {\em transseries of the derivative} of  $\mathsf{A}_{\mathbb{T}_{\mathcal{B}}}\tilde{T}$. Since by assumption $\tilde{T}>0$, by the definition of positivity in \S\ref{STrans}, this dominant term is positive and the property follows. Property iv follows from Equation \eqref{eq:BsumT} by setting all $\tilde{y}_{\bf k}=0$ if $\bf k\ne \rm 0$, choosing $\beta_0=n+1$ and $\tilde y_0=1/x$,  and from the definitions of  $\mathsf{A}_{\mathbb{T}_{\mathcal{B}}}$ and of $\mathsf{A}_{\mathbb T}$. Finally, for Property vi, let $F=\mathcal{LB}\tilde{T}$ and $f=\mathcal{LB}\tilde{t}$. Since $f=(\mathcal{LB}\tilde{T})'$, the rest follows from Definition \ref{PB2D} and Proposition \ref{PB2}.

\end{proof}
Combining the content of the preceding results in this section, we get:
\begin{Theorem}\label{T:T45}{\rm
       $\mathcal{LB}$ is an {\em isomorphism of commutative differential algebras} between $\mathbb{T}_{-,\mathcal{B}}$ and $\mathcal{F}_{-,\mathcal{B}}$. The space $\mathcal{F}_{\ell}\oplus \mathcal{F}_{-,\mathcal{B}}$ is closed under differentiation and antidifferentiation.
       
     } \end{Theorem}

 \begin{Note}\label{Nothertimes}{\rm
      As we mentioned already, to cover generic solutions of nonlinear ODEs we have to allow for more general  than Borel summable series, namely  {\em resurgent} ones. Furthermore, $\mathbb{T}_{+, \mathcal{B}}$ is not closed under antidifferentiation; resurgence tools are required to deal with this space.
     }\end{Note}

{\section{Resurgent functions, resurgent transseries and \'Ecalle-Borel summability: background}\label{SecEBRT}

\subsection{Background: Borel plane singularities along the Laplace transform path and the need to extend Borel summation} So far, antidifferentiation of a transseries $\tilde{T}\in \mathbb{T}_{\mathcal{B}}$ has been defined under the assumption $\tilde{T}_{+,\mathcal{B}}=0$ (cf. Definition \ref{PB2D}). This assumption is needed to ensure Borel summability, as it is manifest in the integral in Equation \eqref{eq:solBorelt} below, involved in the Borel transform of the terms in  $\tilde{T}_{+,\mathcal{B}}$. This condition excludes some very common functions encountered in applications such as ${\rm Ei}_a(x)=\int_a^x \frac{e^s}{s}ds$ (where the integral is understood as a Cauchy principal value if $a\in [-\infty,0)$). Indeed, the transseries of  ${\rm Ei}_a(x)$ is $e^x\tilde{y}(x)+C_a$ where $C_a$ is a constant depending on the endpoint of integration; clearly, in this case, $\lambda=1$ (see Definition \ref{DEBstr}), and classical Borel summation does not apply. 
  
      Furthermore, in transseries arising in applications, the points ${\bf k}\cdot\boldsymbol{\lambda}$ are singularities of the Borel transforms of $y_{\bf k}$. Hence, if $\lambda_j\in\RR^+$, then a generalization of Borel summability is needed. The condition $\lambda_j\notin\RR^+$ may appear to be generic; however, equations arising in applications typically have real coefficients in which case the numbers $\lambda_i$ come in complex conjugate pairs and, more often than not, are purely real.  For instance, for the tronqu\'ee solutions of all Painlev\'e equations $P_{\rm I}- P_{\rm V}$ in normalized (Boutroux) coordinates, the values of $\lambda$ are $\lambda_{1,2}=\pm 1$. 

\'Ecalle introduced significant improvements over Borel summation to address such limitations. Among them are the concepts of \emph{critical times} (see Definition \ref{DefBE}) and \emph{acceleration/deceleration}  to deal with mixed powers of the factorial divergence. Last but not least, and the only additional ingredient we will need, is that of \emph{averaging}.

In linear problems, to avoid the singularities of the integrand  on $\R^+$ (when present),  one can take the half-half average of the Laplace transforms above and below $\R^+$. On the other hand, in nonlinear equations such as nonlinear ODEs, the average  of two solutions is not a solution. \'Ecalle found constructive, universal averages  which successfully replace the naive half-half averages mentioned above; however, it is altogether nontrivial to construct them and show that they work. Of these, we use the so-called {\em organic average}, {\bf mon} 
\cite{Ecalle3,Ecalle4}, which is well suited for our general construction.  Invoking Borel transforms followed by analytic continuation along paths avoiding singularities, followed by taking the organic average of these continuations, and finally applying the Laplace transforms, yields a differential field algebra \cite{Ecalle4,Menous}.

In the remainder of this section we provide an overview of the concepts of averaging and resurgence relevant to the discussion of resurgent functions, resurgent transseries and \'Ecalle-Borel summability below.

\subsection{Averages of Borel transforms (and more general functions in the Borel plane)}\label{EBPa}
Assume that one of the singular directions of the Borel transform $Y=\mathcal B\tilde{y}$ of a normalized series $\tilde{y}$ is $\RR^+$ with a discrete set of singularity locations $\omega_n$, $n\in\NN$.

Since in our discussion the only integration axis that comes into play is $\RR^+$, henceforth we assume that one of the singular directions of the Borel transform of $\tilde{y}$ is $\RR^+$, that $\omega_{0}:=0$, and that for each $n\in\NN$, $\omega_n$ is the $n$th singularity on $\RR^+$, where $\omega_n$ increases with increasing $n$.
 
Now consider the class of all curves going forward (towards $\infty$) that circumvent the just-said singularities. One associates with each such curve a unique series $\epsilon_1,\epsilon_2,...,\epsilon_n,...$ such that for each $n\in\NN$, $\epsilon_n \in \{+,-\}$ where $\epsilon_n = +$ (resp. $-$) indicates that the curve is in the lower (resp. upper) half plane between $\omega_{n-1}$ and $\omega_{n}$.  For instance, $+++---...$ describes a curve starting out in the lower half plane and crossing into and remaining in the upper half plane after the third singularity, $\omega_3$.\footnote{Here we follow \'Ecalle's convention from \cite{Ecalle4}, rather than the convention employed by the first author in \cite{Duke}.} For a point $\zeta$ in the open interval $(\omega_{n},\omega_{n+1})$  along such a curve, the position vector $(\epsilon_1,...,\epsilon_n)$ is called the {\em address of the point $\zeta$} on the curve.

 A {\em uniformizing average} or, more simply, an  {\em average} ${\bf m}:Y \mapsto {\bf m}Y$ of a function $Y$ with singularities at $\omega_i$ as described above is defined using a system of \emph{weights} {\bf m} which, in turn, is defined via the $\omega_n$s and  the $\epsilon_n$s by the following:
 
\begin{equation}
    \label{eq:eqF}
  \mathbf m Y(\zeta) = \sum_{\epsilon_1,...,\epsilon_{n} \in \{+,-\}^n}{\bf m}^{\binom{\epsilon_1,...,\epsilon_{n}}{\omega_1,...,\omega_{n}}} Y^{\binom{\epsilon_1,...,\epsilon_{n}}{\omega_1,...,\omega_{n}}}(\zeta) \;\;\; if\;\; \; \omega_{n} < \zeta < \omega_{n+1},
  \end{equation}
  where $Y^{\binom{\epsilon_1,...,\epsilon_{n}}{\omega_1,...,\omega_{n}}}(\zeta)$ denotes the determination of $Y(\zeta)$ on the open interval $(\omega_n , \omega_{n+1})$ along a curve circumventing the singularities as specified by the position vector $(\epsilon_1,...,\epsilon_{n})$ described above, and ${\bf m}^{\binom{\epsilon_1,...,\epsilon_{n}}{\omega_1,...,\omega_{n}}}$ is its corresponding weight, the definition of which depends on the choice of average in question.

 There are many types of averages  of Borel transforms  to chose from. However, for an average\footnote{Here we refer to {\em continuous averages} as  opposed to the subclass of discrete ones that assume a fixed (for instance periodic) lattice of singularities. Fixed lattice settings simplify the analysis when this analysis is restricted to one particular equation.} $\mathbf m$ to be \emph{well-behaved} with respect to the goals of \'Ecalle-Borel  resummation, following \'Ecalle and Menous (\cite{Ecalle3}, \cite[page 256]{Ecalle4}) there are four conditions it should satisfy: 

 \vspace{5pt}
 $\mathbf A1$. $\mathbf m$ must respect convolution, i.e. $\mathbf m (Y*G)=(\mathbf m Y) (\mathbf m G)$;
 
 $\mathbf A2$. $\mathbf m$ must respect real-valuedness;
 
$\mathbf A3$. $\mathbf m$ must respect lateral growth; 

$\mathbf A4 $. $\mathbf m$ should be scale invariant.
   
   \vspace{5pt}
 $\mathbf A1$ ensures that $\mathbf m$ is an algebra homomorphism;  $\mathbf A2 $ ensures that ${\bf m}Y$ is real whenever $Y$ is real;   $\mathbf A3 $  ensures that exponential bounds, which are needed for the Laplace transform to apply, are maintained by averaging; and  $\mathbf A4 $, which ensures invariance under homothetic rescalings $\zeta\mapsto const. \zeta$ of $\RR^+$, while natural, is not an essential condition. On the other hand, we note that while the functions satisfying $\mathbf A4 $ form an algebra, and functions of ``natural origin'' possess it, such a condition relies on {\em resurgence} in some broad sense to hold\footnote{Indeed, one can use the uniformization theorem to see that  bounds on the ``first'' Riemann sheet do not constrain growth on other sheets.}, and we refer the reader to \cite{Ecalle4} for an in-depth discussion of these issues.

      \subsection{Resurgent series and resurgent functions: definitions}\label{RSRF} 
   
      \begin{Definition}[{\bf Resurgent series and resurgent functions}] \label{D:resurg}{\rm 
          \begin{enumerate} $ $

       \item  In this paper, a normalized series ${\tilde{y}}$ will be said to be {\em resurgent} 
         if its Borel transform $Y=\mathcal B{\tilde{y}}$ is:
         \begin{enumerate}
         \item  analytic or ramified analytic at $p=0$;

         \item endlessly continuable (in the sense that the singularities encountered by analytic continuation along any compact curve segment form a discrete set);
         \item exponentially bounded in every nonsingular direction, while in singular directions $Y$ is in the domain of a  {well-behaved average}.

         \end{enumerate}

  \item Sharing the appellation, $Y=\mathcal B\tilde{y}$ is called {\em resurgent} 
      in the Borel plane or simply resurgent, and $y=\mathcal{L}Y$ is called {\em resurgent} in the physical plane or simply resurgent, if $\tilde{y}$ is resurgent. 

    \item A transseries in $\mathbb{T}_{+}\oplus \mathbb{T}_{\ell}\oplus \mathbb{T}_{-}$ is resurgent if: (a) all component series $\tilde{y}_{\mathbf{k}}(x)$ and ${\tilde{y}}_j(x)$ (see \eqref {trans21} and \eqref{eq:tplus}) are resurgent; (b)  their Borel transforms are  exponentially bounded in every nonsingular direction in a $\mathbf{k},j-$independent way; and, (c) in singular directions, these Borel transforms are in the domain of a well-behaved average.

    \item We denote by $\mathbb{T}_{\mathcal R}$ the space of resurgent transseries in $\mathbb{T}_{+}\oplus \mathbb{T}_{\ell}\oplus \mathbb{T}_{-}$.  Accordingly, $\mathbb{T}_{\mathcal R}=\mathbb{T}_{+,\mathcal{R}}\oplus \mathbb{T}_{\ell}\oplus \mathbb{T}_{-,\mathcal{R}}$, where $\mathbb{T}_{+,\mathcal{R}}$ denotes the space of resurgent transseries in $\mathbb{T}_{+}$ and $\mathbb{T}_{-,\mathcal{R}}$ denotes the space of resurgent transseries in $\mathbb{T}_{-}$. (In virtue of Proposition \ref{N:log}, the members of $\mathbb{T}_{\ell}$ are necessarily resurgent, and so it would be superfluous to write $\mathbb{T}_{\ell, \mathcal{R}}$ in the decomposition of $\mathbb{T}_{\mathcal {R}}$.)
             \end{enumerate}
           }\end{Definition}

                    \subsection{The well-behaved average $\mathbf{mon}$}\label{RSRFa} Condition (1)(c) of Definition \ref{D:resurg} appeals to the notion of a well-behaved average. There is in fact an entire continuum of such averages \cite{Ecalle4}.  As we mentioned above, for our treatment we adopt the average \'Ecalle has dubbed the ``organic average", i.e. the average $\mathbf{mon}$, given by
  
  $$\mathbf{mon}^{\binom{\epsilon_1,...,\epsilon_n}{\omega_1,...,\omega_n}}:=2^{-n}\prod_{i=2}^{n}\left( \left| \epsilon_{i-1} +\epsilon_i\right|-\frac{\epsilon_{i-1} \epsilon_i\omega_i}{\omega_1 + \cdots +\omega_i} \right), $$
    
\noindent  
where, on the right side of the equation, it is understood that $\left| \epsilon_{i-1} +\epsilon_i\right| = 2$ (resp. $0$) if  $\epsilon_{i-1} = \epsilon_i$ (resp. $\epsilon_{i-1} \neq \epsilon_i$)\footnote{In other words, each $+$ (resp. $-$) from the left side of the equation may be regarded as being replaced on the right side of the equation by a $+1$ (resp. $-1$) or alternatively by a  $-1$ (resp. $+1$).} and the $\omega_i$'s are the members of $\mathbb{R}^+$ defined as above \cite[page 272]{Ecalle4}.

Alternatively, $\mathbf{mon}$ can be defined by recursion (\cite[pages 86-87]{Ecalle3}, \cite[page 272]{Ecalle4}), where $\mathbf{mon}^{\binom{\epsilon_1}{\omega_1}}:=\frac12$, and for each $n>1$,
 $$\mathbf{mon}^{\binom{\epsilon_1,...,\epsilon_n}{\omega_1,...,\omega_n}}:=\mathbf{mon}^{\binom{\epsilon_1,...,\epsilon_{n-1}}{\omega_1,...,\omega_{n-1}}}P_{n}, $$
 
 \noindent
 with
  
$$P_n := 1-\frac{1}{2}\frac{\omega_n}{\omega_1 +\cdots +\omega_n} \;\; if \;\; \epsilon_{n-1}=\epsilon_n$$
 
 \noindent
and 

$$P_n := \frac{1}{2}\frac{\omega_n}{\omega_1 +\cdots +\omega_n} \;\; if \;\; \epsilon_{n-1} \neq \epsilon_n.$$

\medskip
In addition to being arguably the simplest of the well-behaved averages, $\mathbf{mon}$ is distinguished by being the lower limit of such averages \cite[page 272]{Ecalle4}(also see \cite{Menous}). It is important to note, however, that when restricted to the class of functions that we are concerned with in this paper, because of nonresonance\footnote{\label{f:4}Nonresonance (see Equation \eqref{eq:nonres2}) is generic: it holds except for a zero measure set in the space of all parameters. 
    As a result of nonresonance there is  one active alien derivative per singular direction \cite{Ecalle4}.  We also note that resonance can lead to multiple \'Ecalle critical times and for resummation {\em \'Ecalle acceleration} is needed. For further details see (\cite{Duke}, \S 1.1.2).} all the well-behaved averages coincide (see Section \ref{S:Uniqueness} ). 
    
    \section{Resurgent functions, resurgent transseries and \'Ecalle-Borel summability}\label{SecEBRT2} 
 
In this section we establish results that we need in \S\ref{CTSN}  to determine the correspondence between the class $\mathcal{F}_\mathcal{R}$  of resurgent functions that we referred to in the introduction and their corresponding classes of transseries and surreal functions.

  \subsection{The resurgent subspace $\mathbb{T}_{\mathcal R}$}\label{EBIc1}
   
     The definitions and propositions in the remainder of this subsection, along with some of the proofs,  largely mimic their counterparts in the section on ordinary Borel summability.

\medskip
The operator  ${\bf mon\,}\!\circ \,\mathcal{B} $  is a proper extension  of $\mathcal{B}$. Henceforth, for brevity, we adopt the following conventions.
   \begin{Definition}\label{D26E}
   \begin{equation}
      \label{eq:defhatb}
     {\hat{\mathcal{B}}:=\bf mon\,}\!\circ \,\mathcal{B} \ \text{\rm and\ }\widehat{\mathcal{LB}}:=\mathcal{L}\!\circ {\bf mon\,}\!\circ \,\mathcal{B}. 
    \end{equation}
    
   {\rm $\widehat{\mathcal{LB}}$ is the \emph{\'Ecalle-Borel summation}, a wide generalization of $\mathcal{LB}$.\footnote{Strictly speaking, there is a different \'Ecalle-Borel summation for each distinct well-behaved average. The summation we employ is based on {\bf mon}. However, as the uniqueness result referred to above and in \S\ref{S:Uniqueness} implies, all \'Ecalle-Borel summations  coincide for the restricted class of functions we are concerned with in this paper.}}
  \end{Definition}
  
  Recalling that well-behaved averages preserve lateral growth, sufficient growth conditions for a transseries to be \'Ecalle-Borel summable are similar to those for usual Borel summability (see Definition \ref{D26}): in particular,} for any small $\epsilon>0$ there are positive constants $c_1,c_2,c_3$ such that for all $\mathbf{k}$
\begin{equation}
  \label{eq:bdscts}
  \left|(\widehat{\mathcal{B}} \tilde{y}_{\mathbf{k}})(p)\right|\le c_1 c_2^{|\mathbf{k}|} e^{c_3 p}\ \text{and}\  \left|(\widehat{\mathcal{B}} \tilde{y}_{\mathbf{k}})(p)\right|\le c_1 c_2^{|\mathbf{k}|} e^{c_3 |p|}\text{ if }|\arg p|\in(\epsilon,2\epsilon).
\end{equation}

 {\begin{Definition}\label{DefLBE}
  {\rm The {\em \'Ecalle-Borel sum} of $\tilde{T}\in \mathbb{T}_{\mathcal R}$ is defined as   
 \begin{equation} 
  \label{eq:BsumTE}
 \widehat{\mathcal{LB}} \tilde{T}=\sum_{\mathbf{k}\ge 0}x^{\boldsymbol{\beta}\cdot \mathbf{k}+1}e^{-\mathbf{k} \cdot \boldsymbol{\lambda} x}\widehat{\mathcal{LB}}\tilde{y}_{\mathbf{k}} +\tilde{T}_{\ell}+\sum_{j=-M}^{-1}x^{\beta_j}e^{\lambda_{j} x}\widehat{\mathcal{LB}}{\tilde{y}}_j(x).
\end{equation}
}\end{Definition}

Using the fact that $\bf mon$ preserves lateral growth it is easy to extend the results obtained in  \S\ref{S:Trans-Borel-summable} to resurgent transseries.  For example, note that in virtue of Definition \ref{DefLBE} we have
  \begin{equation}
    \label{eq:monom}
    \widehat{\mathcal{LB}}(x^{\beta+n}e^{\lambda x}x^{-n})=x^ne^{\lambda x}x^{\beta}x^{-n}= x^{\beta}e^{\lambda x},
  \end{equation}
  which implies  $\widehat{\mathcal{LB}}(x^{\beta}e^{\lambda x})=
   x^{\beta}e^{\lambda x}$.
 \begin{Proposition}\label{P:LBextension}{\rm 
  Acting on $\mathbb{T}_{\mathcal{B}}$, $\widehat{\mathcal{LB}}=\mathcal{LB}$.
}\end{Proposition}
 
\begin{proof}
  For any average, the sum of the weights is $1$. Hence, if $F$ is real-analytic, then ${\bf mon}\,F=F$. 
\end{proof}

In the statement of the following proposition, the positive constants  $c_1,c_2$ and $c_3$ are  the bounds in \eqref{eq:bdscts}.

\begin{Proposition}\label{LEBstrE}  {\rm  Let $\tilde T \in \mathbb{T}_{\mathcal{R}}$. Also, let $\lambda=\min\{\lambda_1,...,\lambda_n\}$ and $\beta=\max\{\beta_1,...,\beta_n\}$. Further, let $x_0$ be such that for all $x>x_0$ we have $c_3e^{-\lambda x}x^\beta <1$. Then:
     \begin{enumerate}
      \item[(a)] For all $ x>\max\{c_3,x_0\}$  we have 
              \begin{equation}
        \label{eq:convergR}
    \sum_{\mathbf{k}>M}x^{\boldsymbol{\beta}\cdot \mathbf{k}}e^{-\mathbf{k} \cdot \boldsymbol{\lambda} x}\left|\widehat{\mathcal{LB}}\tilde{y}_{\mathbf{k}}\right|\le c_1(c_2x^\beta e^{-\lambda x})^{NM}\frac{1}{(1-c_2x^\beta e^{-\lambda x})^N}(x-c_3)^{-1},
  \end{equation} where $N$ is the dimension of the vector $\bf k$. 
  
\item [(b)] The infinite sum in Equation \eqref{eq:BsumTE} converges uniformly and absolutely on the interval  $[r,\infty)$ for any $r\in \RR^+$ satisfying the condition $\nu r-|\beta| \log r > |\log c_2 |$, as well as in the complex strip  $\{x:\Re x\in [r,\infty)\}$.
 \item[(c)] If $M'\le M$ and $\beta=\min\{\beta_0,...,\beta_{M'}\}$, then 
  \begin{equation}
  \label{eq:T+est}
  \sum_{j=-M'}^{-1}x^{\beta_j}e^{\lambda_{j} x}|\widehat{\mathcal{LB}}{\tilde{y}}_j(x)|\le (M'+1)c_1e^{-\lambda_{M'}}x^{\beta}(x-c_3)^{-1}. 
\end{equation}
\item[(d)]$\mathbb{T}_{-,\mathcal R}$ is an algebra, i.e., if $\tilde{T}^{(1)}$ and $\tilde{T}^{(2)}$ are elements of $\mathbb{T}_{-,\mathcal R}$ and $a^{(1)},a^{(2)}\in\RR$, then $a^{(1)} \tilde{T}^{(1)}+a^{(2)} \tilde{T}^{(2)}$ and $\tilde{T}^{(1)}\tilde{T}^{(2)}$ are elements of $\mathbb{T}_{-,\mathcal R}$.
  
\item[(e)]If $\tilde{T}\in \mathbb{T}_{\mathcal R}$, then $\tilde{T}>0$ if and only if $\widehat{\mathcal{LB}}\tilde{T}>0$ for large enough $x$.
    \item[(f)] The kernel of $\widehat{\mathcal{LB}}$ is zero, i.e.,  $\{\tilde{T}\in \mathbb{T}_{\mathcal{R}}: \widehat{\mathcal{LB}}\tilde{T}=0\}=\{0\}$.
      \end{enumerate}
      
    }\end{Proposition}
  \begin{proof}
The proof closely follows the proofs of Proposition \ref{N27}, Proposition \ref{P:bijection} and Corollary \ref{C:bijection}.
\end{proof}

  \begin{Definition}\label{D45}
   {\em Let $\mathcal{F}_{\mathcal R} = \{y:\tilde{y} \in \mathbb{T}_{\mathcal{R}}\}$. Further let $\mathcal{F}_{+,\mathcal{R}}\oplus \mathcal{F}_{\ell}\oplus \mathcal{F}_{-,\mathcal{R}}$ be the decomposition of $\mathcal{F}_{\mathcal{R}}$ induced by the decomposition $\mathbb{T}_{+,\mathcal{R}}\oplus \mathbb{T}_{\ell}\oplus \mathbb{T}_{-,\mathcal{R}}$ of $\mathbb{T}_{\mathcal{R}}$ from Definition \ref{D:resurg}(4).}
\end{Definition}

If in Propositions \ref{P:bijection} through \ref{PB2} we uniformly replace the subscript $\mathcal{B}$ with the subscript $\mathcal{R}$ and uniformly replace the operator $\mathcal{LB}$ with the operator $\widehat{\mathcal{LB}}$, the resulting propositions remain valid, and the changes in their respective proofs are minor. Indeed, the lateral growth conditions on Borel plane functions and their convolutions are at the crux of the proofs, and they are all respected by well-behaved averages. An important such example is the following analog of Lemma \ref{L:28}(3). 

\begin{Proposition}[Differentiation on $\mathbb{T}_{\mathcal R}$] {\rm  Let $\tilde{T}\in\mathbb{T}_{\mathcal R}$. 
  \begin{equation}
      \label{eq:BsumTprime2}
      (\widehat{\mathcal{LB}}\tilde{T})'= \sum_{j=-M}^{-1}x^{\beta_j}e^{\lambda_j x}\widehat{\mathcal{LB}}{\dot{\tilde{y}}}_j+\sum_{\mathbf{k}\ge 0}x^{\boldsymbol{\beta}\cdot \mathbf{k}}e^{-\mathbf{k} \cdot \boldsymbol{\lambda} x}\widehat{\mathcal{LB}}\dot{\tilde{y}}_{\mathbf{k}}+\tilde{T}_{\ell}'=\widehat{\mathcal{LB}}(\tilde{T})'.
    \end{equation}
   Moreover, $\displaystyle  \sum_{\mathbf{k}\ge 0}x^{\boldsymbol{\beta}\cdot \mathbf{k}}e^{-\mathbf{k} \cdot \boldsymbol{\lambda} x}\widehat{\mathcal{LB}}\dot{\tilde{y}}_{\mathbf{k}}$
converges uniformly and absolutely for large $x$.}\end{Proposition}
   \begin{proof}
   As indicated above, the proof closely follows the proof of Lemma \ref{L:28}(3).
  \end{proof}

\begin{Theorem}\label{T:ThmEcalle}{\rm
    (a)  $\widehat{\mathcal{LB}}$ is an {\em isomorphism} between the spaces $\mathbb{T}_{\mathcal{R}}$ and $\mathcal{F}_{\mathcal{R}}$ that preserves differentiation and antidifferentiation.  
    
  (b)  $\widehat{\mathcal{LB}}$ restricted to $\mathbb{T}_{-,\mathcal{R}}$ is an {\em isomorphism} between the algebras $\mathbb{T}_{-,\mathcal{R}}$ and $\mathcal{F}_{-,\mathcal{R}}$ that preserves differentiation and antidifferentiation.

    } \end{Theorem}
\begin{proof}
The proof of Part (a) follows the same steps as those Corollary \ref{C:bijection} through Theorem  \ref{T:T45} together with the aforementioned fact established by \'Ecalle and Menous \cite{Ecalle3} (also see \cite{Ecalle4}) that {\bf mon} respects lateral growth and convolution (and is clearly linear).

   In the same fashion, the proof of the isomorphism of algebras in (b) mirrors, with obvious adaptations, the proof of Proposition \ref{PB2}.
   
     Since the antidifferentiation properties in (b) are important for us,  we provide more detail. To begin with, 
     \begin{equation}
       \label{eq:tplus1}
        \tilde{T}_{+}=\sum_{-M\le j\le -1;\ l\ge 1}c_{j,l}x^{\beta_j}e^{\lambda_j x}x^{-l}=
     \sum_{j=-M}^{-1}x^{\beta_j}e^{\lambda_j x}{\tilde{y}}_j(x).
     \end{equation}
   For each of the terms $x^{\beta_j}e^{\lambda_j x}{\tilde{y}}_j$ we follow the same calculations as in the proof of Proposition \ref{PB2} to obtain the integral equation \eqref{eq:Borelt1}. Now the integral operator is not contractive because of the pole of the denominator. Instead, since the sum in $\tilde{T}_{+}$ is finite, and there are no convergence concerns, rougher estimates suffice. We analyze each term in $\tilde{T}_{+}$ separately. Differentiating now Equation \eqref{eq:tplus1} and proceeding as in the proof of Proposition \ref{PB2}, we get for each term of the sum (setting $\lambda=\lambda_j$ and $\beta=\beta_j$), 
           \begin{equation}
             \label{eq:Borelt2}
      (\lambda -p)W'   +   (\beta-1) W-Y=0,
    \end{equation}
    with the solution (after integration by parts) being
    \begin{equation}
      \label{eq:solBorelt}
      W(p)=-\int_0^p\frac{(-\lambda +p)^{\beta-1}}{(-\lambda +s)^{-\beta-1}}Y(s)ds-(-\lambda +p)^{-1}Y(p)-(-\lambda) ^{-\beta}(-\lambda +p)^{-1}Y(0).
    \end{equation}
    Analyticity of $W$ away from the zeros of the denominators follows from dominated convergence.  Preservation of lateral growth is established as in the case of ordinary Borel summability (see Proposition \ref{PB2} and its proof). The transformations involved in obtaining $W$ from $Y$ are multiplication by $(\lambda-p)^\alpha$, where $\alpha_\pm=\pm \beta-1$ and $f\mapsto \int_0^p f(s)ds=f*1$.  By ${\bf A}1$, convolution preserves growth, while multiplication by $(\lambda-p)^\alpha$ preserves growth along any smooth curve.

\end{proof}

\begin{Note}\label{N:pureseries}{\rm 
It is worth noting that (as follows from Definition \ref{DefLBE} and the just-proved isomorphism theorem), the transseries of the \'Ecalle-Borel sum of a series is the series itself.}
\end{Note}

 \begin{Definition}\label{DPB2R}
  {\rm For $\tilde{T}\in\mathbb{T}_{\mathcal R}$, let $\mathsf{A}(\widehat{\mathcal{LB}}\tilde{T}):=\widehat{\mathcal{LB}}(\mathsf{A}_{\mathbb{T}}\tilde{T})$. 
  }\end{Definition}
The following Corollary is an immediate consequence of Definition \ref{DPB2R} and Theorem \ref{T:ThmEcalle}.
\begin{Corollary}[Antidifferentiation]\label{PB2R}{\rm 
        $\mathsf{A}$ so defined is an antidifferentiation operator (see Definition \ref{Dd2}) on $\mathcal{F}_{\mathcal{R}}$   and
           $$(\mathsf{A}\tilde{T})'=\mathsf{A}(\tilde{T}')=\widehat{\mathcal{LB}}\tilde{T}.$$
                  }\end{Corollary}

\section{Correspondence between resurgent functions, transseries in $\mathbb{T}_1$ and surreal functions: surreal antidifferentiation}\label{CTSN} 

In \S\ref{Sconstr}  we mentioned that to define integrals on {\bf No} we would invoke a pair of isomorphisms--one between a subclass of resurgent functions and a subspace of transseries, and the other between the just-said subspace of transseries and a class of functions on {\bf No}. These are the maps Tr and $\tau$ respectively. We now consider them in turn. 
 \begin{Proposition}\label{Preal-analytic}
     {\rm For each $f\in \mathcal{F}_{\mathcal R}$ (see Definition \ref{D45}) there exists a $c$ such that $f$ is real-analytic on $(c,\infty)$.
     }\end{Proposition}
   \begin{proof}
   Definition \ref{D:resurg} and Theorem \ref{T:ThmEcalle} imply that, for some positive $x_0$, the series of analytic functions on the right side of \eqref{eq:BsumT} converges uniformly on compact sets in a domain $\mathcal{D}=\{x:|x|>x_0,\,|\arg x|<c<\pi/2\}$. Hence,  $\mathcal{L\!\circ {\bf mon}\!\circ B} \tilde f$ is analytic in $\mathcal{D}$. In particular it is real-analytic.
        \end{proof}
   
  \begin{Definition}\label{DTr}
  {\rm Based on the isomorphism in Theorem \ref{T:ThmEcalle},  we define the operator of {\em transseriation} Tr to be the inverse of $\widehat{\mathcal{LB}} $. 
}\end{Definition}

\noindent{\bf{Example E2: The \'Ecalle-Borel summed transseries of $x^\beta e^{{\lambda} x}$}.}\label{EE2} In virtue of Definition \ref{DefLBE} and the fact that $\widehat{\mathcal{LB}}\, 1=1$, we have Tr$(x^\beta e^{\lambda x})=x^\beta e^{\lambda x}$. The equality is also an immediate consequence of Equation \ref{eq:monom}.

\begin{Proposition}\label{Ppos}
     {\rm If $f\in \mathcal{F}_{\mathcal R}$, then there is an $a\in\RR^+$ such that $f(x)>0$ on $(a,\infty)$ if and only if $\mathrm{Tr}f>0$.
  } \end{Proposition}
\begin{proof}
Since well-behaved averages respect lateral growth, the proof mirrors that of Proposition \ref{P:bijection}.
\end{proof}

\subsection{Differentiation of series that are absolutely convergent in the sense of Conway}
\label{SdiffNo} In this subsection, we establish a result that will be needed in the sequel about differentiation of series of surreal functions that converge absolutely in the sense of Conway for any value of the variable belonging to some open interval.

  \begin{Theorem}{\rm 
Let $f_1,...,f_m$ be  twice differentiable infinitesimal functions defined on the positive infinite surreals.   For $n\in\ZZ$, define the function $g$ on each positive infinite surreal $x$  by
    \begin{equation}
      \label{eq:defs}
      g(x)=\sum_{|\mathbf{k}|\ge 0} c_{\bf k} {\bf f}(x)^{\bf k},
    \end{equation}
    where $\{c_{\bf k}\}_{k_i\ge n;i\le m}$ is a sequence of reals.  Then $g$ is differentiable for each such $x$ and its derivative is given as follows by termwise differentiation:
     \begin{equation}
      \label{eq:deri}
      g'(x)=\sum_{|\mathbf{k}|\ge 0} c_{\bf k} \left(\sum_{i=1}^m k_i\frac{f'_i(x)}{f_i(x)}\right){\bf f}(x)^{\bf k},
    \end{equation}
    whereby convention we set $f'_i(x)/f_i(x)=0$ if $f_i(x)=0$. 
  } \end{Theorem}\begin{proof}
We begin with the following simple observation.

\begin{Observation}\label{N:diff}\
{\rm Suppose $f$ is a function such that $f(a+\epsilon)-f(a)=g(a)\epsilon+h(a,\epsilon)\epsilon^2$ where, for some $c>0$ and sufficiently small $\epsilon$ we have  $|h(a,\epsilon)|\le c$. Then $f$ is differentiable at $a$ and $f'(a)=g(a)$. Based on the binomial formula, it is  easy to check that, if $f$ is twice differentiable, $\epsilon$ is infinitesimal and $k\in\NN$, then $f(x+\epsilon)^k-f(x)^k=kf(x)^{k-1}f'(x)\epsilon +k^2f(x)^{k-2}F(x,k;\epsilon)\epsilon^2$ where $F$ is bounded for $k\in\NN$ and infinitesimal $\epsilon$. (Uniform boundedness in $k$ follows from the fact that $0\le m^j|\epsilon|\le 2$ for all $m,j\in\NN$).}
\end{Observation}

   First note that the sum is absolutely convergent in the sense of Conway by Proposition \ref{sur7} and the assumption that the $f_i$ are infinitesimals. We will prove the result for $m=1$; once having done so, the general result follows by induction and the usual decomposition $f_1(x+a)f_2(x+a)-f_1(x)f_2(x)=f_2(x+a)[f_1(x+a)-f_1(x)]+f_1(x)[f_2(x+a)-f_2(x)]$. 
  
  For $m=1$, using Observation \ref{N:diff} and straightforward calculations, it follows that
   \begin{equation}
    \label{eq:ineq2}
 g(x+\epsilon)-g(x)=\epsilon\sum_{k\ge 0} c_{k} kf(x)^{k-1}f'(x)+\epsilon^2 h(x,\epsilon)
\end{equation}
where, using another application of Observation \ref{N:diff}, we see that  $h(x,\epsilon)$ is an absolutely convergent series which is bounded if $\epsilon\ll 1$. \end{proof}
For the second isomorphism, $\tau$, we require the following definition that trades on the intimate relationship between the members of $\mathbb{T}_1$ and certain surreal functions.

\begin{Definition}\label{D:deftau} { \rm In accordance with Definition \ref{NNN} and Proposition \ref{N:log}, each element $ \tilde{T}$ of  $ \mathbb{T}_1$ is a transseries of the form

    \begin{multline}
      \label{eq:deft+}
     \tilde{T}=\sum_{-M\le j\le -1;\ l\ge 1}c_{j,l}x^{\beta_j}e^{\lambda_j}x^{-l}+ \\P(x)\log(x)+Q(x)+R(x)+\\\sum_{\mathbf{k}\ge 0,l\ge 0}c_{\mathbf{k},l}x^{\boldsymbol{\beta}\cdot \mathbf{k}}e^{-\mathbf{k} \cdot \boldsymbol{\lambda} x}x^{-l},
   \end{multline}
   
\noindent   
   where the first sum is in $\mathbb{T}_{+}$, the second sum, in which $P$ and $Q$ are polynomials and  $R$ is a polynomial without constant term, is in $\mathbb{T}_{\ell }$, and the last sum belongs to $\mathbb{T}_{-}$.
   
   With each such $\tilde{T}$ we associate the function $\tilde{T}^f$ consisting of all ordered pairs $(\nu, \tilde{T}^f(\nu))$, where $\nu$ is a positive infinite member of {\bf No}  and $\tilde{T}^f(\nu)$ is the expression that results from first replacing all occurrences of $x$ on the right side of  \ref{eq:deft+} with occurrences of $\nu$, and then replacing (in the resulting expression) the absolutely convergent sum (with bounds $\mathbf{k}\ge 0,l\ge 0$) with the $\mathop{{\rm Lim}}$ sum to which it absolutely converges. That is:

\begin{multline}
      \label{eq:deftau+}
    \tilde{T}^f(\nu):=\sum_{-M\le j\le -1;\ l\ge 1}c_{j,l}\nu^{\beta_j}e^{\lambda_j \nu}\nu^{-l}+\\P(\nu)\log(\nu)+Q(\nu)+R(\nu)+\\\mathop{{\rm Lim}}_{m\to \infty}\sum_{ |{\bf k}|\le m, |l|\le m}c_{\mathbf{k},l}\nu^{\boldsymbol{\beta}\cdot \mathbf{k}}e^{-\mathbf{k} \cdot \boldsymbol{\lambda} \nu}\nu^{-l}.
   \end{multline}}

{\rm Let $\tau:=\{(\tilde{T}, \tilde{T}^{f}) : \tilde{T}  \in \mathbb{T}_1 \}$, $\No^{\tau}:=\{\tilde{T}^{f}:\tilde{T}\in \mathbb{T}_1\}$  (i.e. the range of the map $\tau$) and let $\mathcal{R}_{\bf No}:=\{\tilde{T}^{f}:\tilde{T}\in \mathbb{T}_{\mathcal{R}}\}$. Finally, let} $\No^\tau_-:= \{\tilde{T}^{f}: \tilde{T} \in \mathbb{T}_- \}$. \end{Definition}

\begin{Theorem}\label{taumap}
 {\rm The map $\tau$ in Definition \ref{D:deftau} is an isomorphism of vector spaces endowed with differentiation  between $\mathbb{T}_1$ and $\No^\tau$ and, when restricted to $\mathbb{T}_-$, an isomorphism of differential algebras between $\mathbb{T}_-$ and $\No^\tau_-$.} 
 \end{Theorem}
\begin{proof}
  This follows from the fact that  the transseries topology
  $$ \tilde{T}=\lim_{m\to \infty}\sum_{ |{\bf k}|\le m, |l|\le m}c_{\mathbf{k},l}\nu^{\boldsymbol{\beta}\cdot \mathbf{k}}e^{-\mathbf{k} \cdot \boldsymbol{\lambda} \nu}\nu^{-l} $$
preserves the differential algebra operations (see \S \ref{STrans} and, for proofs, \cite{CostinT}), as does ${\rm Lim}$ (see Proposition \ref{Pcop} and the preceding material in this section).
\end{proof}

To complete the main part of our construction, we need an extension operator $\mathsf{E}$  in the sense of Definition \ref{*E}, and, on restricted domains, a multiplicative extension operator in the sense of Definition \ref{*E10}. It is this to which we now turn. 

\subsection{The extension operator $\mathsf{E}$}\label{EOP}

In the following we introduce an extension operator  $\mathsf{E}$ acting on real-valued functions $f$ on the reals to functions on the finite and positive infinite surreals. Assuming the function $f_+$ defined by $f_+(x)=f(-x)$ is in $\mathcal{F}_{\mathcal{R}}$, the extension of $f$ to negative infinite surreal $x$ is simply defined by $(\mathsf{E}f)(x):=(\mathsf{E}f_+)(-x)$, the subscript + indicting that $f_+$ is defined for $x>0$. In view of this elementary correspondence and to simplify the exposition we will solely focus on finite and positive infinite surreals.
  
   \begin{Definition}\label{defE}{\rm 
    Let $f\in \mathcal{F}_{\mathcal R}$, and let $c\in\RR$ be such that $f$ is real-analytic on $(c, \infty)$ as is assured by Proposition \ref{Preal-analytic}. We extend $f$ to $(c,{\bf On})$ as follows, whereby a \emph{finite} surreal we mean the leading exponent in its normal form is $\leq 0$.
         \begin{enumerate}
         \item  For positive infinite $x\in {\bf No}$ we define  $(\mathsf{E}f)(x)=(\tau\circ\mathrm{Tr}\,\, f)(x)$.

     \item For  finite $x\in {\bf No}$, where $x_0 $ is the real part of $x$ and $\zeta$ is the infinitesimal part of $x $ (see Definition \ref{part}), we define $(\mathsf{E}f)(x)$ by  
  
  \begin{equation}   
    \label{eq:extsr}
    f(x_0+\zeta)=f(x_0)+\sum_{k\ge 1}(k!)^{-1}f^{(k)}(x_0)\zeta^k,
  \end{equation}  
  where the infinite sum is absolutely convergent in the sense of Conway.
               \end{enumerate}

  } \end{Definition}

  Before proceeding further, we offer a couple of observations on the second part of the above definition. To begin with, as above let $f\in \mathcal{F}_{\mathcal R}$ and let $c\in\RR$ be such that $f$ is real-analytic on $(c, \infty)$ as in Proposition \ref{Preal-analytic}. Also let $\epsilon$ be the local radius of convergence of the Taylor series of $f$ at $x_0\in\RR$. For real $|\zeta|<\epsilon$ we have
        \begin{equation}
    \label{eq:extsr1}
    f(x_0+\zeta)=f(x_0)+\sum_{k\ge 1}(k!)^{-1}f^{(k)}(x_0)\zeta^k.
  \end{equation}
  \noindent Substituting $x=1/\zeta$ for the two occurrences of $\zeta$ in Equation \eqref{eq:extsr1}, the right side of the resulting equation is the convergent ({\em a fortiori} Borel and \'Ecalle-Borel summable)  transseries of the left side of the resulting equation  about $x=\infty$.  In particular $f(x_0+x^{-1})$ is resurgent, and Definition \ref{defE} (2) is a special case of (1).  In addition, alternatively and more formally,  we can reduce case (2) of the above definition to case (1) by resorting to  $M$, the M\"obius transformation $x\mapsto x_0+x^{-1}$ (see also Definition \ref {defM})  by defining  $\mathsf{E}$ by $(\mathsf{E}f)(x_0+x^{-1})=[M^{-1}\circ\tau\circ\mathrm{Tr}\circ Mf](x^{-1})$.

 \begin{Theorem}\label{comp} {\rm $\mathsf{E}:\mathcal{F}_{\mathcal R} \rightarrow \mathcal{R}_{\bf No}$ 
 is an isomorphism of linear spaces endowed with differentiation and antidifferentiation. Its restriction to $\mathcal{F}_{-,\mathcal R}$ is an isomorphism of differential algebras.}\end{Theorem}
 
  \begin{proof}
In virtue of the preceding remark, the formula for  $\mathsf{E}$ at an arbitrary point is obtained from the one at $\infty$, hence it is enough to prove the result in the latter case. But at $\infty$ the result follows immediately from Theorems \ref{taumap}  and \ref{T:ThmEcalle}, since $\mathsf{E}$  is a composition of isomorphisms. 
  \end{proof}

 Some special cases of extensions are given below.
\begin{Corollary}\label{P59}{\rm In the following, real functions are assumed to be defined on some interval $(c,\infty)\subset \RR^+$.

  \begin{enumerate}
  \item If $a,b \in\RR$ and $f:\RR^+\to \RR$ is given by $f(x)=x^a e^{bx}$, then $\mathsf{E}f=x^a e^{bx}$ for all positive $x\in \bf No$. 
  
   \item If $P$ is a polynomial and $f:\RR^+\to \RR$ is given by $f=P(x)\log x$, then $\mathsf{E}f=P(x)\log x$ for all positive $x\in {\bf No}$.
   \item  If $f\in \mathcal{F}_{\mathcal R}$ and $f> 0$ on $(c,\infty)$, then  $(\mathsf{E}f)(x) > 0$ for all $x\in {\bf No}$ such that $x>c$.
 
  \end{enumerate}
}\end{Corollary}
\begin{proof}
     (1) and (2) follow immediately  from Definitions \ref{DefLBE} and \ref{defE}.
     
     For (3), note that if  $f(x)>0$ for all real $x\in (c,\infty)$, then, by Proposition \ref{Ppos} we have Tr$f>0$ and, plainly, $\tau\circ \mathrm{Tr} f>0$ for all positive infinite $x\in \bf No$. 

  \end{proof}

  The following result is an immediate consequence of Theorem \ref{comp} and of Corollary \ref{P59}.

   \begin{Theorem}\label{OPE}
   {\rm   $\mathsf E$ is an extension operator in the sense of Definition \ref{*E}. Moreover, $\mathsf E$ restricted to $ \mathcal{F}_{-,\mathcal{R}}$ is a multiplicative extension operator in the sense of Definition \ref{*E10}.}\end{Theorem}

 \noindent   {\bf Example: the special case of functions with convergent transseries at $\infty$.}

First note that, if a convergent transseries is of the type expressed in Equation \eqref{eq:deft+} and its sum is $f$, then $f\in \mathcal{F}_{\mathcal R}$. Indeed, in this case the Borel transform of a convergent series $\sum_{l\ge 0} c_{\mathbf k,l}x^{-l}$ is an entire function, and  \'Ecalle-Borel summation coincides with Borel summation (since in the Borel plane there are no singularities),  and by Proposition \ref{PBorel} (ii) Borel summation is simply the identity. We denote by  $\mathcal{F}_{\mathcal{R},conv}$ the space of the sums (same as Borel sums) of the transseries in $\mathbb{T}_{conv}$. Observe that for $f\in \mathcal{F}_{\mathcal{R},conv}$ we have
 \begin{equation}\label{defTC}
f=   \lim_{L\to \infty}\sum_{L>\mathbf{k}>-M}c_{\mathbf{k}}\boldsymbol{\mu}^{\mathbf{k}}.
  \end{equation}
\begin{Definition}
 {\rm  Let $f\in \mathcal{F}_{\mathcal{R},conv}$. Then $(\mathsf{E}f)(x)$ is defined for positive infinite surreal $x$ as an absolutely convergent series in the sense of Conway, by replacing the exponentials and logarithms in the transseries of $f$ by their surreal counterparts and $\lim$ by Lim.} 
\end{Definition}

\begin{Proposition}\label{PConvTranss}{\rm 
   The operator $\mathsf E | \mathcal{F}_{\mathcal{R},conv}$ (i.e. $\mathsf E$ restricted to $\mathcal{F}_{\mathcal{R},conv}$) is an isomorphism of algebras between the algebra of convergent transseries $\mathcal{F}_{\mathcal{R},conv}$ and its image $\mathcal{F}_{{\bf No},conv}:=\mathsf{E}\mathcal{F}_{\mathcal{R},conv}$.
}\end{Proposition}
\begin{proof}
 This is straightforward since the algebras  $\mathcal{F}_{\mathcal{R},conv}$ and $\mathcal{F}_{{\bf No},conv}$ consist of limits and Limits, respectively, of finite sums.
\end{proof}

 The following result is an immediate consequence of Corollary \ref{P59} and  Proposition \ref{PConvTranss}.

  \begin{Proposition}\label{OPES}
   {\rm  The operator $\mathsf E | \mathcal{F}_{\mathcal{R},conv}$ is an extension operator in the sense of Definition \ref{*E}, and $\mathsf E |\left(\mathcal{F}_{-,\mathcal{R}} \cap  \mathcal{F}_{\mathcal{R},conv}\right)$  is multiplicative in the sense of Definition \ref{*E10}. 
   
  }\end{Proposition}

\subsection{The main theorem on antidifferentiation: the operator $\mathsf{A}_{\bf No}$}

\begin{Definition}\label{defERno}{\rm 
 Let  $\mathsf{E}(\mathcal{F}_{\mathcal{R}}):=\{\mathsf E f:f\in \mathcal{F}_{\mathcal{R}}\}$.
} \end{Definition}

Since $\mathsf{E}:\mathcal{F}_{\mathcal R} \rightarrow \mathcal{R}_{\bf No}$ (from Theorem \ref{comp}) is a surjection, henceforth we write $\mathsf{E}(\mathcal{F}_{\mathcal R})$ in place of $\mathcal{R}_{\bf No}$.

 \begin{Definition}\label{Def59}
  {\rm  By $\mathsf{A}_{\bf No}$ we mean the operator defined by the following conditions: 
   \begin{enumerate}
 \item  for members of $\mathsf{E}(\mathcal{F}_{\mathcal{R}})$, $\mathsf{A}_{\bf No}=\mathsf{E}\mathsf{A}\mathsf{E}^{-1}$, where $\mathsf{E}$ and $\mathsf{A}$, defined on $\mathcal{F}_{\mathcal R}$, are the extension and antidifferentiation operators from Definition \ref{defE} and Definition \ref{DPB2R};  
 \item for $f\in\mathsf{E}(\mathcal{F}_{\mathcal{R}})$ and $\lambda\in {\bf No}$,  $\mathsf{A}_{\bf No}(\lambda f)=\lambda \mathsf{A}_{\bf No}\, f$. 
 \end{enumerate}
 } \end{Definition}

 \smallskip
\noindent{\bf{Example E3: $\mathsf{A}_{\bf No}(e^x)$.}}\label{EE3} Since $\mathsf{A}_{\mathbb T} e^x=e^x$, we obtain, for positive infinite surreal $x$
  \begin{equation}
    \label{eq:antiderexp}
    \mathsf{A}_{\bf No}(e^x)=e^x.
  \end{equation}

It is easy to check that, for any $\lambda\in{\bf No}$ and $f,g\in \mathsf{E}(\mathcal{F}_{\mathcal{R}})$ we have $\mathsf{A}_{\bf No}[\lambda(f+g)]=\lambda\mathsf{A}_{\bf No}\,f +\lambda\mathsf{A}_{\bf No}\,g $. 

We prove in Theorem \ref{TAntidiff} below that $\mathsf{A}_{\bf No}$ is an antidifferentiation operator in the sense of Definition \ref{Dd2}. To prepare the way, we first prove:

\begin{Proposition}\label{P61}{\rm In the following we assume that $c\in\RR$ and $f$ is defined on $\{x\in {\bf No}:x>c\}$.
  \begin{enumerate}
  \item If $f\in \mathsf{E}(\mathcal{F}_{\mathcal{R}})$, then $(\mathsf{A}_{\bf No}f)'=f$.
  \item If $x,y\in(c,\infty)\cap \RR$ and $f\in \mathcal{F}_{\mathcal R}$, then $(\mathsf{A}_{\bf No}f)(y)-(\mathsf{A}_{\bf No}f)(x)=\int_x^yf(s)ds$.
  \item If $f\in\mathsf{E}(\mathcal{F}_{\mathcal{R}})$ is nonnegative and $y>x>c$, then $(\mathsf{A}_{\bf No}f)(y)-(\mathsf{A}_{\bf No}f)(x) \ge 0.$
  \end{enumerate}
  }\end{Proposition}
  \begin{proof} Assume $f\in\mathsf{E}(\mathcal{F}_{\mathcal{R}})$ and let $f_{\RR}:=\mathsf{E}^{-1}f$. By the definition of $\mathsf{E}$, $f(x)=f_{\RR}(x)$ for any real $x\in (c,\infty)$.

    (1)  By Theorem \ref{OPE} and using the construction of $\mathsf{A}_{\bf No}$, we have
    $$(\mathsf{A}_{\bf No}f)'=(\mathsf{E}\mathsf{A}f_{\RR})'=\mathsf{E}(\mathsf{A}f_{\RR})'=\mathsf{E}f_{\RR}=f.$$
    
    (2) The function $g(x,y)=\int_x^y f_{\RR}-[(\mathsf{A}f_{\RR})(y)-(\mathsf{A}f_{\RR})(x)]$ is real-analytic in $(x,y)$ and $\frac{dg}{dy}=0$. Hence $g(x,\cdot)$ is a constant. Since $g(x,x)=0$, $g$ is the zero function.

    (3) Since, as noted, $f_{\RR}$ coincides with $f$ on $(c,\infty)\subset \RR$, we see that $f_{\RR}$ is nonnegative. Let $F_{\RR}=\mathsf{A}f_{\RR}$. By elementary calculus,  for $y>x$ in $(c,\infty)$ we have $F_{\RR}(y)\ge F_{\RR}(x)$. Hence, fixing $x\in (c,\infty)$ with $y$ being the variable, and using the properties of $\mathsf{E}$ and the definition of $\mathsf{A}_{\bf No}$ we obtain ${\rm sign}(y-x)[\mathsf{A}f_{\RR}(y)-\mathsf{A}f_{\RR}(x)]\ge 0$ for all $y\in{\bf No}$. For finite $y_1\le y_2\in{\bf No}$  for which there is an $x\in (c,\infty)$ such that $y_1\le x\le y_2$, we insert an intermediate term to obtain $\mathsf{A}f_{\RR}(y_2)-\mathsf{A}f_{\RR}(x)+\mathsf{A}f_{\RR}(x)-\mathsf{A}f_{\RR}(y_1)\ge 0$. If instead $y_1\le y_2\in{\bf No}$ are finite but there is no such real $x$ satisfying the just-said condition, then the standard parts of $y_1$ and $y_2$ coincide with some $x$ and the property follows from the Taylor expansions of $\mathsf{A}f_{\RR}(x+[y_{1,2}-x])$ around $x$.

      We are left with the analysis of the case when $y_1\le y_2$ are both positive infinite. Let $f_{\RR}=\widehat{\mathcal{LB}}\tilde{T}$. We only analyze the case where the component of $\tilde{T}$ in $\mathbb{T}_+$ is nonzero, say $Cx^{\beta-l}e^{\lambda x}$, as the case where $\tilde{T}\in\mathbb{T}_-\oplus \mathbb{T}_{\ell}$ is  similar. The property that needs to be established is that, if  $\tilde{T}\ge 0$, then $F=\mathsf{A}\widehat{\mathcal{LB}}\tilde{T}$ is an increasing function. From the construction of $\mathsf{A}$, for positive infinite $\nu$, $F(\nu)$ is an absolutely convergent Conway power series with dominant term $ \lambda^{-1}C\nu^{\beta-l}e^{\lambda \nu}$. We distinguish two cases. If $0<y_2-y_1=\epsilon\ll 1$, then each term in the Conway expansion of $F$ can be reexpanded in $\epsilon$. By this we mean the following. Letting $y_1=\nu$ and $\epsilon_1=\epsilon \nu^{-1}$ we have
      \begin{equation}
        \label{eq:expan}
       e^{\lambda_j(\nu+\epsilon)}(\nu+\epsilon)^{\beta_j-l}=e^{\lambda_j \nu}\nu^{\beta_j-l} {\rm Lim}_{m\to \infty}\sum_{r=0}^m\sum_{k=0}^m\frac{\lambda_j^r\epsilon^r}{r!}\binom{\beta-l}{k}\epsilon_1^k,
      \end{equation}
  which we insert in the Lim term above in the first sum in Equation \eqref{eq:deftau+}, and we similarly reexpand the other terms to obtain the Lim as $m\to\infty$ of an $N+2$-dimensional truncated power series.  This expansion shows that the dominant term of $F(\nu+\epsilon)-F(\nu)$ is  $C\nu^{\beta-l}e^{\lambda \nu}\epsilon>0$. If instead, $0<y_2-\nu=a$ is finite, then, with $s$ being some infinitesimal function and $a^\circ$ being the standard part of $a$,  we have $F(\nu+a)/F(\nu)=(1+a/\nu)^{\beta-l}e^{\lambda a}(1+s(\nu+a))/(1+s(\nu))>e^{\lambda a^\circ} >1$, as it is easy to check.
  \end{proof}

It is worth noting that, in virtue of the construction of $\mathsf{A}_{\bf No}=\mathsf{E}\mathsf{A}\mathsf{E}^{-1}$ we have obtained more than just antidifferentiation; in particular, we have obtained the operator $\mathsf{E}$, which in turn is used in the construction of $\mathsf{A}_{\bf No}$.  $\mathsf{A}_{\bf No}$ provides the solutions of  equations of the form $f'=g$, whereas, by virtue of the fact that $\mathsf{E}$ preserves the operations of differential algebra, $\mathsf{E}$ can be used to solve classes of nonlinear equations, such as  ODEs, and difference equations.
}
This brings us to the main theorem on antidifferentiation.

\begin{Theorem}\label{TAntidiff}{\rm 
  $\mathsf{A}_{\bf No}$ is an antidifferentiation operator in the sense of Definition \ref{Dd2}.
}\end{Theorem}
\begin{proof} The satisfaction of condition i of Definition \ref{Dd2} follows from Proposition \ref{P61} (1); the satisfaction of ii follows from Definition \ref{Def59} (1) and (2) and the linearity of $\widehat{\mathcal{LB}} $; the satisfaction of  iii follows from  Proposition \ref{P61} (2); and the satisfaction of iv and v follows from Proposition \ref{P59} (1). For the satisfaction of vi, let $f=Ef_{\RR}$, and note that $f'=0$ means $(\mathsf{E}f_{\RR})'=0$, thereby implying $f'_{\RR}=0$. Hence there is a $C\in\RR$ such that $f_\RR=C$, implying $f=C$.
\end{proof}

In virtue of Equation \eqref{eq:defint1}, Proposition \ref{existint} and Theorem \ref{TAntidiff}, we now have:
\begin{Corollary}{\rm
$\int_{x}^{y} f=\mathsf{A}_{\bf No}(f)(y)-\mathsf{A}_{\bf No}(f)(x)$ is an integral operator on the domain of $\mathsf{A}_{\bf No}$.}
\end{Corollary}

\subsection{Uniqueness}\label{S:Uniqueness}The existence of a continuum of nonequivalent well-behaved averages induces an apparent nonuniqueness of the operators $\mathsf{E}$ and $\mathsf{A}_{\bf No}$. However, as we mentioned above, when restricted to the class of functions with which we are concerned with in this paper (see the introduction as well as Footnote \ref{f:4} on nonresonance and the remark it is appended to) all such averages coincide, thereby resulting in unique operators when $\mathsf{E}$ and $\mathsf{A}_{\bf No}$ are thus restricted. A detailed analysis of this will be the subject of a different paper. 

\section{The extension and antidifferentiation operators $\mathsf{E}^*$ and $\mathsf{A}_{\bf No}^*$}\label{EA} 

Often singular behavior occurs in other limits than $x\to\infty$. For instance, for a modular form such as the elliptic theta function $\theta_3$, the unit circle in $\CC$ is a natural boundary, and the limits of interest on the real line are $\pm 1$ (see \S\ref{Stheta3}).   Here, by changes of variable,  we expand the domain of our extension and antidifferentiation operators to handle arbitrary points.

The extension operator $\mathsf{E}^*$ is constructed in two stages. We begin by defining $\mathsf{E}_{\infty}$ acting on functions that are resurgent at $x_0=\infty$, with values in surreal functions defined for positive infinite surreals, namely we define  $(\mathsf{E}_{\infty}f)(x)=(\tau\circ\mathrm{Tr}\,\, f)(x)$. For functions that are resurgent at finite $x_0$, or $x_0=-\infty$ we simply change variables to bring the case to $x_0=\infty$. For example, the function $t\mapsto \exp(-1/t)$ is real-analytic on  $(0,1)$ but not at zero; extending it to positive infinitesimal $t$ is done by writing $t=1/x$ and extending the new function $e^{-x}$ to positive infinite values of $x$ using $\mathsf{E}_{\infty}$. That is, $ (\mathsf{E} f)(1/t):=(\mathsf{E}_{\infty} f)(x)$ with $x=1/t$. For the sake of completeness we formalize this  process in the paragraphs below.

\begin{Definition}\label{defM}{\rm 
Let $x_0\in\RR,a\in \RR^+$ and $f:D(f)\to\RR$ be a real-analytic function.  If  $D(f)=(a,\infty)$, then we let $m(x)=x$, the identity. If $D(f)=(-\infty,-a)$, we let $m(x)=-x $; if $D(f)=(x_0,x_0+a)$, then we let $m(x)=x_0+1/x$, and if $D(f)= (x_0-a,x_0)$, then $m(x)=x_0-1/x$. We then define $Mf=f\circ m$.  The domain of $Mf$ is $(a^{-1},\infty)$ in the first two cases and $(a,\infty)$ in the last two.}  \end{Definition}

 The class of functions we have heretofore been concerned with that we call ``resurgent" are the members of $\mathcal{F}_{\mathcal R}$; see Definition \ref{D45}. The following definition expands the class of functions we subsume under this appellation.
 
    \begin{Definition}\label{defEE}{\rm 
If  $Mf\in \mathcal{F}_{\mathcal R}$ with $M$ and $f$ as in Definition \ref{defM}, we say $f$ is {\em resurgent}. Let  $\mathcal{F}_\mathcal{R}^{*}$ be the set of all resurgent functions in the just-said sense. If  $y \in \mathcal{F}_{\mathcal R}^*$, we say $y$ is {\em resurgent at} $x_0=\infty$  (resp. $-\infty$) if $x\mapsto y(x)$ (resp. $x\mapsto y(-x)\in \mathcal{F}_{\mathcal{R}}^*$, and we say $y$ is  {\em resurgent to the right} (resp. {\em left}) of $x_0\in\RR$ if $x\mapsto y(x_0+1/x)$ (resp. $x\mapsto y(x_0-1/x) \in \mathcal{F}_{\mathcal{R}}^*$.} 
 \end{Definition}   
       
\begin{Definition}\label{defEE1}{\rm 
       Suppose $f$ is resurgent (in the sense of Definition \ref{defEE}). If $M$ is the identity (i.e. if $f$ is resurgent at $\infty$)  we define $\mathsf{E}_{\infty}^*f$ for $\infty<x\in \bf No$ by $(\mathsf{E}_{\infty}^*f)(x):=(\tau\circ\mathrm{Tr}\,\, f)(x)$ and let $\mathsf E^*=\mathsf E_{\infty}^*$. More generally, in all four cases of Definition \ref{defM}, we set $\mathsf E^*:= M^{-1}\mathsf{E}_{\infty}^* M$. Also set  $\mathsf E^*(\mathcal{F}_{\mathcal{R}}^*):=\{\mathsf E^*f: f \in \mathcal{F}_{\mathcal{R}}^*\}$. }
\end{Definition}

Notice that $\mathcal{F}_{\mathcal{R}} \subset \mathcal{F}_{\mathcal{R}}^*$ and, hence,  $\mathsf E(\mathcal{F}_{\mathcal{R}}) \subset \mathsf E^*(\mathcal{F}_{\mathcal{R}}^*)$.

 \begin{Theorem}\label{OPE*}
   {\rm   $\mathsf E^*$ is an extension operator in the sense of Definition \ref{*E}. Moreover, $\mathsf E^*$ restricted to $\{f:Mf\in \mathcal{F}_{-,\mathcal{R}}^*\}$ is a multiplicative extension operator in the sense of Definition \ref{*E10}.}\end{Theorem}
 \begin{proof}
Proposition \ref{OPE} shows that $\mathsf E^*_{\infty}$ has the properties stated in the theorem.  Conjugation through $M$,  $M^{-1} (\cdot) M$ is an obvious  structural isomorphism, ensuring preservation of the required properties.   
 \end{proof}
 
 \subsection{Definition of  $\mathsf{A}_{\bf No}^*$} 
  We  now make use of Definition \ref{defM} to define integration by changes of variable via $M$. Recall that $\mathsf{A}_{\bf No}$ has the intuitive interpretation  $(\mathsf{A}_{\bf No}f )(x)=\int_{\infty}^x f$. To extend $\mathsf{A}_{\bf No}$ to other points $x_0 \in \RR$ (or $-\infty$), we change the variable of integration to map $x_0 \in \RR$ (resp. $-\infty$) to $\infty$. In the change of variable $Mf=f\circ m$, $m$ is a one-to-one rational function which coincides with its surreal extension.

  For example, if $D(f)=(0,\epsilon)$ we have  $m(s)=1/s$ and we note that, heuristically, $$\displaystyle \int_0^s f(u) du=\int_{\infty}^{1/s}\left(-\frac{1}{v^2}\right)f\left(\frac1{v}\right)dv.$$ More generally, we obtain (still heuristically) $(\mathsf{A}_{\bf No}^* f)m(x)= (\mathsf{A}_{\bf No}m' f\circ m )(x)$, which motivates the following definition.

\begin{Definition}\label{RD}{\rm 
For $Mf\in \mathcal{F}_{\mathcal{R}}$, let  $(\mathsf{A}_{\bf No}^* f)\circ m:=\mathsf{A}_{\bf No}(m' Mf)$.
}\end{Definition}

\begin{Theorem}\label{TAntidiff*}{\rm 
  $\mathsf{A}_{\bf No}^*$ is an antidifferentiation operator on  $\mathsf{E}^*(\mathcal{F}_{\mathcal{R}}^*)$.
}\end{Theorem}
\begin{proof}
  Using the properties of  $\mathsf{A}_{\bf No}$ and of differentiation it is straightforward to check that the properties listed in Definition \ref{Dd2} hold. We check only i.; the others being similar and in fact simpler.

    We rewrite the  definition as $[(\mathsf{A}_{\bf No}^* f)](s)=[\mathsf{A}_{\bf No}(m'\,\, f\circ m) ](m^{-1}(s))$ and, using the chain rule  together with the fact that $\mathsf{A}_{\bf No}$  is an antidifferentiation operator, we get
  \begin{multline}
    [(\mathsf{A}_{\bf No}^* f)]'=\frac{[\mathsf{A}_{\bf No} (m'\,\, f\circ m)]' (m^{-1}(s))}{m'(m^{-1}(s))}=\frac{m'(m^{-1}(s))\,\, (f\circ m)(m^{-1}(s))}{m'(m^{-1}(s))}=f(s).
  \end{multline}

\end{proof}
\begin{Note}{\rm 
  The reader can see that  $\mathsf{A}_{\bf No}^*$ is obtained from $ \mathsf{A}_{\bf No}$ simply by changes of variables, in the same way  $\mathsf{E}^*$ was obtained from $\mathsf{E}$, with the consequence that the antiderivative of the extension of $f \in \mathcal{F}_{\mathcal{R}}^*$ is the extension of the antiderivative of $f$.
}\end{Note}
\begin{Corollary}{\rm
$\int_{x}^{y} f=\mathsf{A}_{\bf No}^*(f)(y)-\mathsf{A}_{\bf No}^*(f)(x)$ is an integral operator on  $\mathsf{E}^*(\mathcal{F}_{\mathcal{R}}^*)$.}
\end{Corollary}

\section{Illustrations of extensions and antidifferentiations}\label{ILL} 

Many of the familiar functions have convergent expansions at points on the real line, where the actions of the extension and antidifferentiation operators are easy to obtain, as in Equation \eqref{eq:extsr} and the cases mentioned in the comments following Definition \ref{defE}. Accordingly, in the first subsection, we focus on calculating functions for positive infinite surreal values when these functions have a singularity at $\infty$. In illustrations (i) and (ii), we go over all details of the analysis. 

\subsection{Functions having a  singularity at $\infty$}
(i) First consider the \emph{exponential function} $x\mapsto e^{x}$. Using Example E3, we have $\mathsf{A}_{\No}(e^x)=e^x$. Hence, by Equation \eqref{eq:defint1}, we have 
\begin{equation}\label{eq:expf}
 \int_0^{x} e^sds=e^{x}-1,
 \end{equation} 
 as expected.  This stands in contrast to Norton's aforementioned proposed definition of integration which was shown by Kruskal to integrate $e^x$ over the range $[0,\omega]$ to the wrong value $e^\omega$ \cite[p. 228]{CO2}. We note that \cite{RSAS} also obtains Equation \eqref{eq:expf}.

\medskip
(ii) Next consider $t^{-1}e^t$. Its antiderivative is the {\em exponential integral} Ei.

By changes of variables, we have
\begin{equation}\label{eq:expint}
\text{Ei}(x):={\rm PV}\int_{-\infty}^x t^{-1}e^t dt =e^x {\rm PV}\int_0^\infty\frac{e^{-xp}}{1-p}dp=e^x {\rm PV}\mathcal{L}\frac{1}{1-p}.
\end{equation}
We have $\widehat{\mathcal{LB}}\tilde{y}= {\rm PV}\mathcal{L}(1-p)^{-1}$.
Hence, Tr$(\text{Ei}(x))=e^x\tilde{y}=\mathsf{A}_{\mathbb T}[t^{-1}e^t]$. Using Definition \ref{Def59}, this gives, for all positive infinite surreal $x$, 
\begin{equation}
  \label{eq:finresEi}
  (\mathsf{A}_{\bf No}[t^{-1}e^t])(x)=e^x\sum_{k=0}^\infty\frac{k!}{x^{k+1}}.
\end{equation}

The values of Ei for positive finite surreal $x$ are obtained simply from the local Taylor series at the real part of $x$ (see Definition \ref{part}), as explained in Equation \eqref{eq:extsr} contained in Definition \ref{defE}.

The fact that Equation \eqref{eq:finresEi} should hold up to an additive constant was known to Conway and Kruskal but the value of this constant resisted their years long effort\footnote{Oral communication by Martin Kruskal to the first author.}. As \eqref{eq:finresEi} shows, this constant is zero.

\medskip
(iii) The \emph{imaginary error function} erfi. To calculate 
\begin{equation}
\label{eq:erfi1}
f(x)=\int_0^x e^{s^2}ds=\frac{\sqrt{\pi}}{2}\mathrm{erfi}(x)
\end{equation}
for positive infinite surreal $x$, we first find the \'Ecalle-Borel summed transseries of $f$. In this example, the \'Ecalle critical time is not the original variable. By applying integration by parts to the integral in Equation \eqref{eq:erfi1} we obtain:
\begin{equation}
\label{eq:erfi}
\int_0^x e^{s^2}ds\sim e^{s^2}\left(\frac{1}{2x}+\frac{1}{4x^2}+\cdots\right).
\end{equation}
We notice here that the exponent is $t=s^2$, which is the \'Ecalle critical time (see Definition \ref{DefBE}), and therefore  we have to pass to the variable $t$. We then observe that $f'(x)=e^{x^2}$, where $f(0)=0$.  With the substitution  $f(x)=x\exp(x^2)g(x^2)$ with $x^2=t$ we get 
\begin{equation}
\label{eq:eqgt}
g'+\left(1+\frac{1}{2t} \right)g =\frac{1}{2t}.
\end{equation}
The transseries of $g$ is simply a power series whose Borel transform $G$ satisfies, in accordance with \eqref{eq:eqgt},
\begin{equation}
\label{eq:eqGp}
(p-1)G'+\frac12 G=0,\  
\end{equation}
where $G(0)=1/2$. In this case {\bf mon} gives
\begin{equation}\label{eq:eqGp}
  \frac14\left(\int_0^{\infty-0i}+\int_0^{\infty+0i}\right)\frac{e^{-tp}}{\sqrt{1-p}}dp\\=\frac12\int_0^1\frac{e^{-tp}}{\sqrt{1-p}}dp,
\end{equation}
which leads to
\begin{equation}
  \label{eq:finEi}
  f(x)=\frac12 xe^{x^2}\int_0^1\frac{e^{-x^2 p}}{\sqrt{1-p}}dp.
\end{equation}
We note that $f(0)=0$, and so this expression satisfies both the differential equation and the initial condition. The transseries of $f$ is easily obtained from \eqref{eq:eqGp} and Lemma \ref{L:Watson} combined with the binomial formula. Ultimately, we arrive at
\begin{equation}
\label{eq:trerfi}
{\rm Tr}(f)(x)= \frac{1}{\sqrt{\pi}}e^{x^2}\sum_{n=0}^\infty \frac{\Gamma(n+\frac12)}{x^{2n+1}}
\end{equation}
and, hence,
\begin{equation}
  \label{eq:surrerfi}
  \int_0^x e^{s^2}ds= \frac{1}{\sqrt{\pi}}e^{x^2}\sum_{n=0}^\infty \frac{\Gamma(n+\frac12)}{x^{2n+1}}\ \ \ {\rm for \; all\; surreal\;} x>\infty.
\end{equation}
We note that the expression above is {\bf not} valid for $x<\infty$, let alone at $x=0$, and we don't expect it to satisfy $f(0)=0$. As is indicated above, for finite $x$ the values of $f$ for Ei are obtained from the local Taylor series.

(iv) The {\em Airy functions} ${\rm Ai}$ and ${\rm Bi}$. These are two special solutions of the Airy equation
  $$y''=z y.$$
  Their asymptotic expansions for large  $z\in\RR^+$ are 
\begin{equation}
    \label{eq:Ai}
{\rm Ai}(z)  \displaystyle\sim\frac{e^{-\zeta}}{2\sqrt{\pi}z^{1/4}}\sum_{k=0}^{\infty}(-1)^%
{k}\frac{u_{k}}{\zeta^{k}};\ \ \zeta:=\frac23 z^{3/2}
\end{equation}
and
\begin{equation}
  \label{eq:Bi}
{\rm Bi}(z)\sim\displaystyle\frac{e^{\zeta}}{\sqrt{\pi}z^{1/4}}\sum_{k=0}^{\infty}\frac{u%
_{k}}{\zeta^{k}},  
\end{equation}
respectively (see \cite{DLMF}). We note that the asymptotic expansions are given in \cite{DLMF} in terms of $\zeta$, which is precisely the \'Ecalle critical time (see Definition \ref{DefBE}).

Changing the variable $z$ to $\zeta$, the inverse Laplace transform of $\exp(-\zeta) $Ai$(\zeta) \zeta^{-1/6}$ is the hypergeometric function
    $_2F_1\left(\frac{7}{6},\frac{11}{6};2;-\frac{p}{2}\right)$ which is analytic on $\RR^+$, and hence usual Borel summabilty applies.

(v) The {\em Gamma} and log-{\em gamma} functions. The Borel summed transseries of the log-gamma function, $\log\Gamma(x)$, is\footnote{This is a simplification of the form given in  \cite[p. 96]{Book}.}
\begin{equation}
  \label{eq:loggamma}
  \log\Gamma(x)=x(\log(x)-1)-\frac12\log\frac{x}{2\pi}+\int_0^\infty e^{-xp}\frac{p\coth(p/2)-2}{2p^2}dp.
\end{equation}
Using the generating function of the Bernoulli numbers we get, for small $p$,
\begin{equation}
  \label{eq:bern2}
\frac{p\coth(p/2)-2}{2p^2} =\sum_{n=0}^\infty\frac{ B_{2 n}}{(2 n)!}p^n.
\end{equation}
Hence, the transseries of $\log\Gamma(x)$ is given by  
\begin{equation}
  \label{eq:transloggamma}
  \log\Gamma(x)=x(\log(x)-1)-\frac12\log\frac{x}{2\pi}+\sum_{n=0}^\infty\frac{B_{2 n}n!}{(2 n)!x^{n+1}}.
\end{equation}
Then, for positive infinite surreal $x$, $ \log\Gamma(x)$ is given by the right side of Equation \eqref{eq:transloggamma} where the infinite series are now interpreted as absolutely convergent series in the sense of Conway.

The Gamma function is simply obtained from the log-gamma function by exponentiation.

\subsection{An example of a function whose singularities are at finite points: Jacobi's elliptic function $\theta_3$ } \label{Stheta3} Jacobi's elliptic function $\theta_3$ is defined by
   \begin{equation}
     \label{eq:theta3}
     \theta_3(q)=\sum_{n\in\ZZ} q^{n^2}=1+2 \sum_{n\in\NN} q^{n^2};\ \ |q|<1
   \end{equation}
(see e.g. \cite{RW20}).  Clearly $\theta_3$ is analytic in the complex unit disk, in particular, it is real-analytic on $(-1,1)$. Since the series $ \sum_{n\in\NN} q^{n^2}$ is {\em lacunary}, the unit circle is a natural boundary (see e.g. \cite{Mandelbrojt}). In particular, the points $\pm 1 $ are singular. However,  as we will now see $\theta_3\in \mathcal{F}^*_{\mathcal R}$ and therefore it extends to a left (resp. right) surreal neighborhood of $1$ (resp. $-1$).
   Define $t$ by  $q=e^{-t}$ and let $q_1=e^{-\pi^2/t}$.  We note that $q\to 1$ is equivalent to $t\to 0$.  Jacobi's modular transformations applied to $\theta_3$ give:
   \begin{equation}
     \label{eq:Jacobi}
    \theta_3(q)=\sqrt {\frac{\pi}{t}} \theta_3(q_1)=\sqrt {\frac{\pi}{t}}\left(1+2 \sum_{n\in\NN} e^{-{n^2\pi^2/t}}\right).
   \end{equation}
   This is a convergent transseries, and applying the definition of $\mathsf E^*$ we obtain, for positive infinitesimal $\zeta$,
    \begin{equation}
     \label{eq:Jacobi2}
  (\mathsf E^*  \theta_3)(e^{-\zeta})=\sqrt {\frac{\pi}{\zeta}}\left(1+2 \sum_{n\in\NN} e^{-{n^2\pi^2/\zeta}}\right).
   \end{equation}
   For $q\to -1$ we have similar formulas, and we omit the intermediate steps. Indeed, $\theta_3(q)=\theta_4(-q)=\sqrt {\frac{\pi}{t}}\theta_2(-q_1)=2 q_1^{1/4}\sum_{n\ge 0} q_1^{n(n+1)}$ (see \cite{RW20}), which implies for $q$ infinitesimally greater than $-1$, 
  \begin{equation}
     \label{eq:Jacobi3}
  (\mathsf E^*  \theta_3)(-e^{-\zeta})=2\sqrt {\frac{\pi}{\zeta}}e^{-\frac14\pi^2/\zeta}\sum_{n\ge 0} e^{-{n(n+1)\pi^2/\zeta}}.
   \end{equation}

\section{The theory of surreal integration: a generalization}\label{GEN}

It is natural to inquire in which ordered exponential subfields $(K, \exp_K)$ of $(\No, \exp)$ the above theory of surreal integration restricted to $K$ continues to be applicable. In this section we show that a sufficient condition is that $K$ is \emph{closed under absolute convergence in the sense of Conway}, that is, for each formal power series $f$ in $n\ge 0$ variables with  coefficients in $\RR$, $f\left (a_1,..., a_n \right)$ is absolutely convergent in the sense of Conway in $K$ for every choice of infinitesimals $a_1,..., a_n$ in $K$. After having demonstrated this, we will exhibit ordered exponential subfields of $(\No, \exp)$ that are closed in this sense. 

Note that, for $n=0$ the ring of power series in $n$ variables with coefficients in $\RR$ is $\RR$ itself,
and so henceforth we may assume that all references to reals are references to members of $\RR \subseteq K \subseteq \No$ and furthermore that $(\RR, e^x)\subseteq (K, \exp_K)\subseteq (\No, \exp)$.

\begin{Lemma}\label{GENERAL1}{\rm
If $(K, \exp_K)$ is an ordered exponential subfield of $(\No,\exp)$ that is closed under absolute convergence in the sense of Conway, then for each $f \in  \mathsf{E^*}(\mathcal{F}_{\mathcal{R}}^*)$ (see Definition \ref{defEE1}) and each $x\in {\rm dom}(f)\cap K$, $f(x) \in K$. }
\end{Lemma}
\begin{proof}

Suppose the hypothesis holds and further suppose $f \in  \mathsf{E^*}(\mathcal{F}_{\mathcal{R}}^*)$ and $x\in {\rm dom}(f)\cap K$.

{\bf Case 1}. $f \in  \mathsf{E}(\mathcal{F}_{\mathcal{R}})$ and, hence, $\mathsf{E}^{-1}f$ is resurgent at $\infty$.

If $x=x_0 +\zeta$, where $x_0$ is in the real-analyticity domain of $\mathsf{E}^{-1} f$ and $\zeta$ is infinitesimal, then in virtue of Definitions \ref{defM} and \ref{defE} and Corollary \ref{P59}(4),\footnote{Since  $f$ is an extension of $\mathsf{E}^{-1}f$, $(\mathsf{E}^{-1}f)(x_0)=f(x_0)$ and $(\mathsf{E}^{-1}f)^{(k)}(x_0)=f^{(k)}(x_0)$.} 
   $$\mathsf{E}f(x)=\mathsf{E}f(x_0+\zeta)=(\mathsf{E}^{-1}f)(x_0)+\sum_{k\ge 1}(k!)^{-1}(\mathsf{E}^{-1}f)^{(k)}(x_0)\zeta^k. $$

\noindent
 As such, since $K$ is closed under absolute convergence in the sense of Conway, $(\mathsf{E}f)(x) \in K$.

If $x$ is positive infinite, then in virtue of Definitions \ref{D:deftau} and the first part of Definition \ref{defE}, $(\mathsf{E}f)(x)$ assumes the form

  \begin{multline*}
     \sum_{-M\le j\le -1;\ l\ge 1}c_{j,l}x^{\beta_j}e^{\lambda_j}x^{-l}+ \\P(x)\log(x)+Q(x)+R(x)+\\\sum_{\mathbf{k}\ge 0,l\ge 0}c_{\mathbf{k},l}x^{\boldsymbol{\beta}\cdot \mathbf{k}}e^{-\mathbf{k} \cdot \boldsymbol{\lambda} x}x^{-l},
   \end{multline*}
  
\noindent 
where $P$, $Q$ and $R$ are polynomials, $R$ being without constant term, $M$ is a natural number, the coefficients and powers are real numbers, and the terms of the form $x^{\beta_j'}$, etc., (or their multiplicative inverses), can be written as exponentials (or multiplicative inverses of exponentials) using the identity $x^{a}=e^{a\log x}$. Accordingly, since $P(x)\log(x)+Q(x)+R(x)$ and the finite sum over $M$ are clearly both in $K$, to show the entity denoted by the full expression is in $K$ it remains to note that the $\mathop{{\rm Lim}}$ term is in $K$ in virtue of $K$'s closure under absolute convergence in the sense of Conway.

{\bf Case 2.}  $f \notin \mathsf{E}(\mathcal{F}_{\mathcal{R}})$. The extension operator $\mathsf E^*$ is defined as $\mathsf E^*:= M^{-1}{\mathsf E}_{\infty}^* M$ where $(Mf)(x)=f(m(x))$. This clearly preserves the range of $f$:  the values of $Mf$ are in $K$ if and only if the values of $f$ are in $K$.  
\end{proof}

For each $f \in  \mathcal{F}_{\mathcal{R}}^*$, 
let $\mathsf{E}_{K}^*(f):=\mathsf{E}^*(f)|K$, i.e. $\mathsf{E}^*(f)$ restricted to $K$, and for each $f \in  \mathsf{E}^*(\mathcal{F}_{\mathcal{R}}^*)$, let  $\mathsf{A}_{\bf No}^{*K}(f):=\mathsf{A}_{\bf No}^*(f)|K$. Also, let $\mathsf{E}^*(\mathcal{F}_{\mathcal{R}}^*)|K:=\{f|K: f \in \mathsf{E}^*(\mathcal{F}_{\mathcal{R}}^*)\}$.

\begin{Theorem}\label{GENERAL2}{\rm 
Let $(K, \exp_K)$ be an ordered exponential subfield of $(\No,\exp)$ that is closed under absolute convergence in the sense of Conway, and let $x,y \in K$. Then: 
\begin{enumerate}
\item $\mathsf{E}_{K}^*$ is an extension operator on $\mathcal{F}_{\mathcal{R}}^*$ in the sense of Definition \ref{*E};
\item $\mathsf{A}_{\bf No}^{*K}$ is an antidifferentiation operator on $\mathsf{E}^*(\mathcal{F}_{\mathcal{R}}^*)|K$ in the sense of Definition  \ref{Dd2};
\item $\int_{x}^{y} f=\mathsf{A}_{\bf No}^{*K}(f)(y) - \mathsf{A}_{\bf No}^{*K}(f)(x)$ is an integral operator on $\mathsf{E}^*(\mathcal{F}_{\mathcal{R}}^*)|K$ in the sense of Proposition \ref{existint}.

\end{enumerate}
 }
\end{Theorem}

\begin{proof}
(1) follows from Lemma \ref{GENERAL1} and Theorem \ref{OPE*}. In addition, since the antiderivative of the extension of $f \in  \mathcal{F}_{\mathcal{R}}^*$ is the extension of the antiderivative of $f$, to establish (2) we need only further appeal to (1) and Theorem \ref{TAntidiff*}. (3) follows from (2) together with Equation \eqref{eq:defint1} and Proposition \ref{existint}.
\end{proof}

Our first example of a structure that satisfies the hypothesis of Theorem \ref{GENERAL2} is given by:

\begin{Theorem}\label{ILL1}
{\rm The ordered exponential subfield $\mathbb{R}((\omega))^{LE}$ of $(\No, \exp)$ is closed under absolute convergence in the sense of Conway.}
\end{Theorem}
\begin{proof}
As was mentioned in \S\ref{STrans}, $\mathbb{R}((\omega))^{LE}$ is an ordered exponential subfield of $\No$ that is isomorphic to the ordered exponential field $\mathbb{T}$ of transseries. Moreover, $\mathbb{T}$ is equal to the union $\bigcup_{i \in I}H_i$ of a family $\{H_{i}: i \in I\}$ of Hahn fields having the property: for all $i,j \in I$ there is a $k \in I$ such that $H_{i}, H_{j} \subseteq H_k$ (see \cite[Appendix A]{ADH1}). Plainly then, for each finite set $a_1,..., a_n$ of infinitesimals in $\bigcup_{i \in I}H_i$, there is an $m \in I$ such that  $a_1,..., a_n \in H_m$. Moreover, since $H_m$ is a Hahn field, by a classical result of B. H. Neumann \cite{N}, for each $f(x_1,..., x_n) \in \RR[[X_1,\dots,X_n]]$, $f(a_1,..., a_n) \in H_{m}$. But then $f(a_1,..., a_n) \in  \bigcup_{i \in I}H_i$, which proves the proposition. 
\end{proof}

Our next group of examples of structures that satisfy the hypothesis of Theorem \ref{GENERAL2} comes from work of van den Dries and the second author \cite{vdDE}. The demonstration that these structures do indeed satisfy the said hypothesis rests largely on Propositions \ref{DE1}, \ref{DE2} and \ref{DE3} below, the formulations of which require the following definitions.
\begin{sloppypar}
If $\mathit{\Gamma}$ is a subgroup of $\No$ whose universe is a set, then
there is a canonical isomorphism $f$ of the Hahn field $\RR((\tau^{\mathit{\Gamma}}))$ into $\No$ for which $$f(\sum\limits_{\alpha  < \beta } {\tau^{y_\alpha  } .r_\alpha})= 
\sum\limits_{\alpha  < \beta } {\omega ^{y_\alpha  } .r_\alpha}.$$
The image of $f$, denoted $\RR((\omega^{\mathit{\Gamma}}))$, is the \emph{Hahn field in $\No$ generated by} $\mathit{\Gamma}$. By $\RR((\omega^{\mathit{\Gamma}}))_\lambda$ we mean the set consisting of all elements of $\RR((\omega^{\mathit{\Gamma}}))$ having supports $(y_{\alpha})_{\alpha < \beta <\lambda}$, where $\lambda$ is a fixed ordinal. Moreover, for each ordinal $\lambda$, let  $\No(\lambda):=\{x \in \No: {\rm tree\; rank\; of\;} x < \lambda \}$ (see \S\ref{PreS}). Finally, as the reader will recall, an \emph{additively indecomposable} ordinal is an ordinal of the form $\omega^{\alpha}$ for some $\alpha \in {\bf On}$, and an \emph{$\epsilon$-number} is an ordinal $\lambda$ such that $\omega^{\lambda}=\lambda$. 
\end{sloppypar}

\begin{Proposition}{\rm (\cite[Corollary 3.1]{vdDE})\label{DE1}}
{\rm $\No(\lambda)$ (with sums and order inherited from $\No$) is an ordered abelian group whenever $\lambda$ is an additively indecomposable ordinal.}
\end{Proposition}

\begin{Proposition} {\rm (\cite[Lemma 4.6]{vdDE})\label{DE2}}
{\rm Let $\lambda$ be an $\epsilon$-number and let $\mathit{\Gamma}$ be a subgroup of $\No$. Then
$\RR((\omega^{\mathit{\Gamma}}))_\lambda$ (with sums, products and order inherited from $\No$) is an ordered field closed under absolute convergence in the sense of Conway. }
\end{Proposition} 

It should be noted that of the portion of the above result concerned with closure under absolute convergence in the sense of Conway is not explicitly stated as a result in \cite{vdDE}, but rather is proved (without the current terminology) in the course of proving the weaker condition of closure under restricted analytic functions (see \cite[page 11]{vdDE}\footnote{More specifically, the authors write: ``Finally, let $F(X_1,\dots,X_n)\in \RR[[X_1,\dots,X_n]]$ be
a formal power series in the indeterminates $X_1, \dots,X_n$
with real coefficients. Let $\epsilon_1, \dots,\epsilon_n$ be infinitesimals
in $\RR(( \tau^\mathit{\Gamma} ))_\lambda$. Since $F$ is not assumed to be a
convergent power series, we actually prove more than
closure under restricted analytic functions by showing that
$F(\epsilon_1,\dots,\epsilon_n)\in \RR(( \tau^\mathit{\Gamma} ))_\lambda$."}).

\begin{Proposition}{\rm (\cite[Proposition 4.7 (1) and (2)]{vdDE})\label{DE3}}
{\rm Let $\lambda$ be an $\epsilon$-number. Then: 
\begin{enumerate}

\item $\No(\lambda)$ (with sums, products and order inherited from $\No$) is a real-closed field closed under exponentiation and under taking logarithms of positive elements. Indeed, $\No(\lambda)$ equipped with restricted analytic functions (defined via Taylor expansions) and exponentiation induced by $\No$ is an elementary substructure of $(\No_{\rm an}, \exp)$ and an elementary extension of $(\mathbb{R}_{\rm an}, e^x)$. 
\item $\No(\lambda)=\bigcup_{\mu}\RR((\omega^{\No(\mu)}))_\lambda$, where $\mu$ ranges over all additively indecomposable ordinals less than $\lambda$. 
\end{enumerate}}
\end{Proposition}

It follows from much the same argument employed in the proof of Theorem \ref{ILL1} that the union of a chain of ordered fields that are closed under absolute convergence in the sense of Conway is itself closed under absolute convergence in the sense of Conway. In virtue of this and Propositions \ref{DE1}, \ref{DE2} and \ref{DE3}, we now have:

\begin{Theorem}\label{ILL2}
{\rm For each $\epsilon$-number $\lambda$, the ordered exponential subfield $\No(\lambda)$ of $(\No, \exp)$ is closed under absolute convergence in the sense of Conway.}
\end{Theorem}

Like $\No$ and the extension theory developed in the earlier sections of the paper, Lemma \ref{GENERAL1} and Theorem \ref{GENERAL2} as well as the existence of $\mathbb{R}((\omega))^{LE}$ and the $\No(\lambda)$ are provable in ${\rm NBG^-}$, and are therefore constructive in this sense.\footnote{The second author wishes to thank Elliot Kaplan for helpful comments on an earlier version of this section of the paper, and especially for his observation that the proofs of Lemma \ref{GENERAL1} and Theorem \ref{GENERAL2} given above continue to hold without the previously stated additional assumption that $(K, \exp_K)$ is initial.}

\section{Some open questions and a remaining problem}\label{OQ}

We draw the positive portion of the paper to a close by stating a problem and two open questions that naturally
arise from the material in preceding sections.

The mathematical theory of resurgent functions for height one transseries has long been worked out in great detail. In a far ranging recent work \cite{Ecalle5}, however, \'Ecalle has provided what he describes as an ``exploratory rather than systematic" presentation of an extension of his theory of resurgent functions, including \'Ecalle-Borel summability, beyond height one transseries to transseries having arbitrary heights and depths. This naturally suggests:

\medskip
 {\bf Problem 1.} Based on a rigorous theory of arbitrary height and depth transseries, generalize our ``constructive'' treatment of extension and antidifferentiation operators to all resurgent functions. 
\medskip

A related and perhaps much deeper issue is broached by:
\medskip

{\bf Question 1.} Do well-behaved extension operators exist for broad classes of functions that cannot be obtained  from the inductive construction yielding transseries? More specifically, do well-behaved extension operators exist for broad classes of functions defined on surreals of arbitrary length (or at least having lengths larger than $\omega^\omega$).

\medskip
The answer to this question would shed light on the very important but much less understood subject of {\em formalizability} of functions.\footnote{``[Reducibility] to a properly structured set of real coefficients'' \cite[p. 75]{Ecalle2}.} We note that Jean \'Ecalle offered the following very interesting observation in response to this question.

\begin{quote}
... under the (reasonable) assumption that the limits for extending operations such as integration on functions of $\No$  to $\No$  roughly coincide with the
limits for the effective (bi-constructive) formalization for real germs at $\infty$, one
falls back on a subject on which much thought has already been spent, and I
think one can confidently predict the broad outline of the answer. The ultimate
constructive extensions would:

\medskip
\noindent
1. \emph{include} all formal transfinite iterates of order $\alpha < \omega^\omega$ of the exponential,
together with a coherent system of incarnations as real germs (and while
the search for \emph{one} privileged system of incarnations is hopeless, comfort
may be taken from the fact that all coherent incarnations are isomorphic). 

\medskip
\noindent
2. \emph{exclude} the full set of so-called nested expansions (even well-nested ones),
for there mutual compatibility conditions would have to be met, which
could not possibly be ensured \emph{constructively}, i.e. without massive recourse to
AC. 

\end{quote}

The answer to the following question will shed more light on the deeper structure of the surreal universe.

\medskip
{\bf Question 2.} Can the theory of extension, antidifferentiation and integral operators presented in the previous sections of the paper be given a genetic (simplicity-hierarchical) formulation in the {\emph{inductive} sense (mentioned in the introduction) that was sort after by Conway, Kruskal and Norton? 

\medskip
The authors do in fact know how to provide a simplicity-hierarchical account for much of the theory in terms of Conway's $\{L|R\}$ notation and hope to present it in a future paper. However, the definitions in terms of Conway's $\{L|R\}$ notation employed in the account are not inductive, and therefore are not genetic in Conway's sense (see Footnote \ref{f1}).

\section{Negative and independence results}\label{Neg}
With our positive results now at hand, we switch directions by showing that a constructive proof of the existence of analogous extensions and integrals of substantially more general types of functions than those treated above is obstructed by considerations from the foundations of mathematics. These considerations apply not only to the surreals, but to any non-Archimedean ordered field $\mathbb{F}$ that extends $\RR$ and whose existence can be proved in ${\rm NBG}^-$. To establish the result, we will direct our attention to a list of very basic properties of antidifferentiation and a space $\mathcal{H}$ of functions with ``very good properties''.

By a classical result of Mostowski \cite{Most}, Wang \cite {HW}, Novack \cite{Nov}, Rosser and Wang \cite{RW} and Schoenfield \cite{SH}, ${\rm NBG^-}$ is a \emph{conservative extension} of ZF; that is, every theorem of ZF is a theorem of ${\rm NBG}^-$, and every theorem of ${\rm NBG}^-$ that can be expressed in ZF (i.e. in the language of sets) is itself a theorem of ZF. Consequently, ${\rm NBG^-}$ is not only equiconsistent with ZF, but if $T$ is a theory obtained from ZF by supplementing it with a set of axioms $A$ which involve only sets, and $T'$ is obtained from ${\rm NBG^-}$ by supplementing it with the same set of axioms $A$, then $T$ is consistent if and only if $T'$ is consistent (see, for example, \cite[p. 132]{FR}). As a result, the above said relations holding between ${\rm NBG^-}$ and ZF also hold between  ${\rm NBG}^- +\,$C (the Axiom of Choice for sets) and ZFC as well as between ${\rm NBG^-}$+DC and ZF+DC, where DC is the Axiom of Dependent Choice (\cite{SH}, \cite{Most}, \cite{Nov}, \cite{RW}). Accordingly, though the main result in this section is concerned with arbitrary non-Archimedean ordered fields that extend $\RR$ whose existence can be proved in ${\rm NBG}^-$, including $\mathbf{No}$ itself, the preliminary results are about subsets of these structures and as such, when appropriate, to prove these results we freely make use of techniques and results established about or in ZFC, ZF+DC or ZF+DC supplemented with other assertions about sets.

As usual, let $\ell^\infty$ denote the space of all bounded real-valued sequences, whose members we write as $\{s_n\}$ in place of $\{s_n\}_{n \in \NN}$.  As the reader will recall,  $\phi:\ell^\infty \rightarrow  \RR$ is said to be a \emph{Banach limit} if it is a continuous linear functional satisfying the following conditions: (a) ({\em positivity}) if $\{s_n\}$ is a nonnegative sequence, then  $\phi(\{s_n\})\ge 0$; (b)  ({\em shift-invariance}) for any sequence $\{s_n\}\in\ell^\infty$, we have  $\phi(\{s_n\})=\phi(\{s_{n+1}\})$; and (c) ($\phi$ {\em is a limit}) if $\{s_n\}$ is convergent with limit $l$, then  $\phi(\{s_n\})=l$.

To establish our negative result we will make use of one direction of the following metamathematical result concerning  the existence of Banach limits (EBL) which is a simple consequence of results from the literature.\footnote{The second author greatfully acknowledges helpful discussions of this matter with Emil Je\u{r}\'abek \cite{EJ}, Wojowu (Wojtek Wawr\'ow) \cite{WW} and others on MathOverflow.} 

\begin{Proposition}\label{EBL}
{\rm EBL is independent of
 ${\rm NBG^-}$+DC (if  ${\rm NBG}^-$ is consistent).} 
\end{Proposition}
\begin{proof}
Since ${\rm NBG}^-$+DC is a conservative extension of ZF+DC, it suffices to prove the proposition for ZF+DC. EBL is consistent with  ZF+DC, if  ZF is consistent, since HB (the Hahn-Banach theorem) implies EBL (e.g. \cite[Theorem III.7.1]{JBC}) and ZF+DC+HB is consistent, if ZF is consistent (\cite{PS}, \cite[p. 516]{Hand}). Moreover, -EBL is consistent with ZF+DC, if ZF is consistent, since BP (the assertion every set of reals has the Baire Property) implies -EBL \cite[Theorem 44]{Karagila}, and there is a model of ZF+DC+BP, if ZF is consistent  (\cite{Shelah}, \cite[p. 516]{Hand}).
\end{proof}

As we shall see, the following concept is closely related to a Banach limit.

 \begin{Definition}\label{SL}{\rm 
    We call $L$ a  \emph{sublimit} on $\ell^\infty$ if: 
    \begin{enumerate}
    \item $L$ is linear, i.e., if $s,t\in\ell^\infty$ and $a,b\in\RR$, then $L(as+bt)=aL(s)+bL(t)$;
      
    \item for every  $\{s_n\}\in \ell^\infty$, $L(\{s_n\})\le \limsup_{n\to\infty}s_n$. 
    \end{enumerate}
      }\end{Definition}
    Using linearity and the fact that $\liminf_{n\to\infty}s_n=-\limsup_{n\to\infty}(-s_n)$, we see that condition (2) of Definition \ref{SL} is equivalent to
    \begin{equation}
      \label{eq:dblelimit}
       \liminf_{n\to\infty} s_n\le L(\{s_n\})\le \limsup_{n\to\infty}s_n.
    \end{equation}

  \begin{Lemma}\label{L:equiv-lim}
  {\rm ZF proves that a Banach limit exists if and only if a sublimit exists.} 
  \end{Lemma}
  \begin{proof}
  
      Let $B$ be a Banach limit. We show that $B$  is a sublimit. Indeed, by definition, $B$ is linear. Now assume $\limsup_{n\to \infty} s_n=l$. Let $\epsilon>0$  and let $N\in \NN$ be such that for all $m\ge N$ we have $s_m\le l+\epsilon$.   Let $S$ be the shift operator, $S(\{s_n\})=\{s_{n+1}\}$. We have, using positivity,  $B (\{s_n\})=B S^N (\{s_n\})\le l+\epsilon$ where, as usual,  $S^N$ is $S$ applied $N$ times. Since $\epsilon >0$ is arbitrarily chosen, $B (\{s_n\})\le l$.
  
Now let $L$ be a sublimit. We define the Ces\`aro summation operator by $C(\{s_n\})=\{n^{-1}\sum_{j=1}^n s_n\}$. Note that $C$ is a continuous operator on $\ell^\infty$ of norm $1$, and so is $L$, by \eqref{eq:dblelimit}. We claim that $LC$ is a Banach limit. We just showed that $LC$ is continuous. Moreover, clearly $C$ is a positive operator and so is $L$ by \eqref{eq:dblelimit}.  Since $\big[C(\{s_n\})-CS(\{s_n\})\big]_m=m^{-1}(s_1-s_{m+1})$, we have $\liminf_{n\to\infty}\big[LC(\{s_n\})-LCS(\{s_n\})\big]=0$ and $\limsup_{n\to\infty}\big[LC(\{s_n\})-LCS(\{s_n\})\big]=0$, and hence, $LC=LCS$, thereby proving shift-invariance. Finally, if $\{s_n\}$ is convergent with limit $l$, then plainly $C(\{s_n\})$ is convergent with limit $l$, and then \eqref{eq:dblelimit} shows that $LC(\{s_n\})$ is convergent with limit $l$, completing the proof. 
  \end{proof}

  The extension and antidifferentiation operators introduced in \S\ref{MainDef} are linear,  positive and satisfy a number of other desirable properties. However, the \emph{proto-antidifferentiation}  and \emph{proto-extension} operators defined below are only assumed to be linear, positive and satisfy a compatibility condition. By showing that the existence of these proto-operators can neither be proved nor disproved in ${\rm NBG}^-$+DC, we show that the existence of any extension and antidifferentiation operators having even these minimal properties can neither be proved nor disproved in ${\rm NBG}^-$+DC, let alone in ${\rm NBG}^-$. As such, any such operators whose existence can be established in NBG would necessarily be less constructive in nature  than $\mathsf{E}$ and $\mathsf{A}_{\bf No}$.

 Henceforth, let $\mathbb{F}$ be a non-Archimedean ordered field that extends $\RR$ whose existence can be proved in ${\rm NBG}^-$, and let $\rho$ be an arbitrarily selected and fixed positive infinite element of $\mathbb{F}$.  Also let $\mathcal{H}$ be the space of functions which are real-analytic, that
  extend to entire functions (holomorphic in $\CC$), and decay at least as rapidly as $x^{-2}$ in the sense that $f\in\mathcal{H}$ if and only if $\sup_{x\in\RR^+}x^2 |f(x)|<\infty$.

 \begin{Definition}\label{D:defle}
    \begin{sloppypar}
    {\rm 
   Let $ A_{\rho}:=\{(x,y) \in (\RR^+\cup\{\rho,\rho^2\})^2: x<y\}$. A proto-antidifferentiation operator on  $\mathcal{H}\times A_{\rho}$ is an operator $\lambda$ having the following properties for all  $f,g\in \mathcal{H}$, all $(x,y)\in A_{\rho}$, and all $\alpha,\beta \in\RR^+$. 
      \begin{enumerate}
      \item Linearity: $\lambda(\alpha f+\beta g,x,y)=\alpha\lambda( f,x,y)+\beta \lambda(g,x,y)$.
         \item Positivity: If for some $c\in\RR^+$ and all $x\in (c,\infty)$ we have $f(x)\ge 0$, then for all $(x,y)\in A_{\rho}$ with $x\in (c,\infty)$ we have $\lambda(f,x,y)\ge 0$.
     \item Compatibility with the weight of $\mathcal{H}$: $\lambda(x^{-2},x,y)=x^{-1}-y^{-1}$. 
        \end{enumerate}}
           \end{sloppypar}
    \end{Definition}
  It is easy to see that the operator $(f,x,y)\mapsto \int_x^yf(s)ds$ satisfies the above properties for all real $x>0$.
 
     \begin{Definition}\label{D:deflee}{\rm 
    Let $ E_{\rho}:=\RR^+\cup\{\rho\}$. A proto-extension operator on $\mathcal{H}\times E_{\rho}$ is an operator $\Lambda$ having the following properties for all $f \in \mathcal{H}$, all $x\in E_{\rho}$ and all $\alpha,\beta \in\RR^+$.
      \begin{enumerate}  
    \item Linearity: $\Lambda(\alpha f+\beta g,x)=\alpha\Lambda( f,x)+\beta \Lambda(g,x)$.  
    \item Positivity: If for some $c\in\RR^+$ and all $x\in (c,\infty)$ we have $f(x)\ge 0$, then for all $x\in (c,\infty)$, we have  $\Lambda( f,x)\ge 0$. 
   \item  Compatibility with the weight of $\mathcal{H}$: $\Lambda(x^{-2},x)=x^{-2}$.
          \end{enumerate}
         } \end{Definition}

Henceforth, let EPA and EPE be the following statements with $\mathbb{F}$ and $\rho$ understood as above: ``There exists a proto-anti-differentiation operator as in Definition \ref{D:defle}, and ``There exists a proto-extension operator as in  Definition \ref{D:deflee} '', respectively. Moreover, henceforth by $x^\circ$ we mean the \emph{standard part} of a finite member $x$ of an ordered field. When $x$ is a finite surreal number, $x^\circ$ is the real part of $x$ (see Definition \ref{part}).

\begin{Lemma}\label{X}
{\rm EPA implies EBL and EPE implies EBL in ${\rm NBG}^-$}
\end{Lemma}
  {
\begin{proof}
  Let $\mathbb{F}$ be an ordered field extending $\RR$ that exists in ${\rm NBG}^-$, $\rho$ be a positive infinite member of $\mathbb{F}$ and $\Lambda $ be a proto-extension operator.  

For each $s=\{s_n\}\in \ell^\infty$ define $f_s:\RR\to\RR$ by  $f_s(x)=s_1$ for $x\le 1$, and for $x=(1-t)n+t (n+1)$, where $t\in [0,1]$ and $n\in\NN$, by $f_s(x)=(1-t)s_n+t s_{n+1}$. Now let $m\in\NN$ and suppose $x>m$. It is easy to check that
     \begin{equation}
       \label{eq:ineqs0}
x^{-2}   \inf_{n\ge m } s_n\le x^{-2} f_s(x) \le    x^{-2} \sup_{n\ge m }   s_n.
     \end{equation}
   We are now going to approximate $f_s$ by entire functions.   For $\epsilon>0$ let
\begin{equation}
  \label{eq:defnu}
 \nu_{\epsilon}:= \nu= 2\|s\|_{\infty}\pi^{-\frac12}\epsilon^{-1}
\end{equation}
and consider the  mollification $ f_{s;\epsilon}(x):=\pi^{-\frac12}\nu \int_{-\infty}^{\infty} e^{-\nu^2(x-t)^2}f_s(t)dt$.   By standard complex analysis,  $f_{s;\epsilon}$ is entire, and straightforward estimates show that $\sup_{z\in\CC}|e^{-\nu^2|z|^2}f_{s;\epsilon}(z)|<\infty$. Note that,  by construction,  $|f_s(t)-f_s(x)|\le 2\|s\|_{\infty}|t-x|$. Thus, \eqref{eq:defnu} implies
\begin{multline}\label{eq:mestdif}
  |f_{s;\epsilon}(x)-f_s(x)|=\pi^{-\frac12}\nu\left|\int_{-\infty}^{\infty} e^{-\nu^2(x-t)^2}(f_s(t)-f_s(x))dt\right|\\\le 2\|s\|_{\infty}\pi^{-\frac12}\nu\left|\int_{-\infty}^{\infty} e^{-\nu^2v^2}|v|dv\right|= 2\|s\|_{\infty}\pi^{-\frac12}\nu^{-1}\le\epsilon.
\end{multline}
Conditions (2) and (3) of Definition \ref{D:deflee} imply that for any $x\in\RR^+$ and $\epsilon,\epsilon'>0$ we have
\begin{equation}
  \label{eq:diffs}
 |x^{-2} f_{s;\epsilon}(x)-x^{-2}f_{s;\epsilon'}(x)|\le x^{-2}(\epsilon+\epsilon'),
\end{equation}
and hence,
\begin{equation}
  \label{eq:diffs}
 |\rho^2\Lambda (\rho^{-2} f_{s;\epsilon},\rho)-\rho^2\Lambda (\rho^{-2} f_{s;\epsilon'},\rho)|\le (\epsilon+\epsilon').
\end{equation}
Since Equation \eqref{eq:mestdif} and the triangle inequality imply that $|\rho^2\Lambda (\rho^{-2} f_{s;\epsilon},\rho)|\le \epsilon+\|s\|_{\infty}$, $\left(\rho^2\Lambda (\rho^{-2} f_{s;\epsilon},\rho)\right)^\circ$ exists, and, by the same argument, so does $\left(\rho^2\Lambda (\rho^{-2} f_{s;\epsilon'},\rho)\right)^\circ$.  Hence, by taking standard parts in \eqref{eq:diffs}, we get
\begin{equation}
  \label{eq:diffs2}
 \left|\left(\rho^2 \Lambda (\rho^{-2}f_{s;\epsilon},\rho)\right)^\circ-\left(\rho^2 \Lambda (\rho^{-2}f_{s;\epsilon'},\rho)\right)^\circ\right|\le \epsilon+\epsilon',
\end{equation}
which in turn implies that
\begin{equation}
  \label{eq:diffs3}
B_{L}(s):=\lim_{\epsilon\to 0}\left(\rho^2 \Lambda (\rho^{-2}f_{s;\epsilon},\rho)\right)^\circ
\end{equation}
exists.
Moreover, it follows from Equation \eqref{eq:mestdif} that  for any $n\in\NN$ and $x\ge n+1$, we have 
  \begin{equation}
       \label{eq:ineqs0}
x^{-2} f_{s,\epsilon}(x) \le    x^{-2} (\sup_{n\ge m }   s_n+\epsilon).
\end{equation}
Hence, since $\rho\ge n$ for any $n\in\NN$, we have
 \begin{equation}
       \label{eq:ineqs0}
\rho^2 \Lambda (\rho^{-2}f_{s,\epsilon},\rho) \le    \limsup_{n\to\infty}   s_n+\epsilon,
\end{equation}
which implies
\begin{equation}
  \label{eq:diffs4}
B_L(s)\le \limsup_{n\to \infty } s_n.
\end{equation}
And so a sublimit exists, and hence, by Lemma \ref{L:equiv-lim}, a Banach limit exists.

For the portion of the theorem concerned with proto-antidifferentiations (in place of proto-extensions) we change $\Lambda(\rho^{-2}f_{s;\epsilon},\rho)$ to $\lambda(\rho^{-2}f_{s;\epsilon},\rho,\rho^2)$ and the prefactor $\rho^2$ in front of $\Lambda$ to a prefactor $\rho$ in front of $\lambda$. Since for $a\in\RR$,  $\rho(a\rho^{-1}-\rho^{-2})^\circ$= $\rho(a\rho^{-1})^\circ=a$, the proof is, mutatis mutandis, the same.
\end{proof}

\begin{Lemma}\label{L:converse}
{\rm EBL implies EPA and EPE in ${\rm NBG}^-$}.
\end{Lemma}
\begin{proof}
  Let $B_L$ be a Banach limit.  To show EPA, define $\lambda(f,x,y)=\int_{x}^y f(s)ds$ if $(x,y)\in (\RR^+)^2$, $\lambda(f,x,\rho)=x^2f(x)(x^{-1}-\rho^{-1})$ for $x\in\RR^+$,  and $\lambda(f,\rho,\rho^2)=B_L[n^2f(n)](\rho^{-1}-\rho^{-2})$. It is clear that $\lambda$ is linear and, if $f\ge 0$ is positive and $(x,y)\in A_{\rho}$, then $\lambda(f,x,y)\ge 0$. For condition (3), we note that if $f(x)=x^{-2}$ then $\lim_{n\to\infty} n^2 f(n)=1=B_L(n^2f(n))$ by the definition of a Banach limit, and the property follows.

  To show EPE, define $\Lambda(f,x)=f(x)$ if $x\in \RR^+$ and $\Lambda(f,\rho)=B_L[n^2f(n)]\rho^{-2}$. It is clear that $\Lambda$ is linear and, if $f\ge 0$ is positive and $x\in E_{\rho}$, then $\Lambda(f,x)\ge 0$. For condition (3), we note that if $f(x)=x^{-2}$ then $\lim_{n\to\infty} n^2 f(n)=1=B_L(n^2f(n))$ by the definition of a Banach limit, and the property follows.
\end{proof}

\begin{Theorem}\label{T:equiv} {\rm 
     \begin{enumerate}
     \item ${\rm NBG}^-$ proves that proto-antidifferentiation operators exist if and only if Banach limits exist.
     \item ${\rm NBG}^-$proves that proto-extension operators exist if and only if Banach limits exist.
       
     \item EPA and EPE are independent of ${\rm NBG}^-$+DC (if  ${\rm NBG}^-$+DC is consistent).
         \end{enumerate}

  } \end{Theorem}

 \begin{proof}
 (1) and (2) follow from Lemmas \ref{X} and \ref{L:converse}, and (3) is an immediate consequence  of (1), (2) and Proposition \ref{EBL}.

 \end{proof}
 }
 \begin{Note}{\rm 
   Other types of negative results can be obtained via Pettis' theorem of automatic continuity, whereby the existence of various other types of desirable extensions would imply the existence of Baire non-measurable sets. This will be explored further in the future paper referred to at the end of the introduction.
 }\end{Note}

  \section{Acknowledgments} The authors wish to express their gratitude to Jean \'Ecalle for his important comments and suggestions on an earlier version of this work. We also wish to thank him
for granting us permission to quote in \S\ref{OQ} his comments regarding the possible extension of our
theory.

    The work of OC is supported in part by the U.S. National Science Foundation, Division of Mathematical Sciences, Award NSF DMS-2206241.

\end{document}